\pdfoutput=1
\newif\ifpersonal
\personaltrue 
\RequirePackage[l2tabu,orthodox]{nag} 
\documentclass[10pt,a4paper,reqno]{amsart} 
\usepackage{amsmath,amsthm,amssymb,mathrsfs,mathtools,bm,eucal,tensor} 
\usepackage{microtype,fixltx2e,lmodern} 
\usepackage[utf8]{inputenc} 
\usepackage[T1]{fontenc} 
\usepackage{enumerate,comment,braket,xspace,tikz-cd,csquotes} 
\usepackage[all,cmtip]{xy} 
\usepackage[centering,vscale=0.7,hscale=0.8]{geometry}
\usepackage[hidelinks]{hyperref}
\usepackage[capitalize]{cleveref}

\theoremstyle{plain}
\newtheorem{thm-intro}{Theorem}
\newtheorem{thm}{Theorem}[section]
\newtheorem*{thm*}{Theorem}

\newtheorem{lem}[thm]{Lemma}
\newtheorem{prop}[thm]{Proposition}

\newtheorem{cor}[thm]{Corollary}

\theoremstyle{definition}
\newtheorem{defin}[thm]{Definition}
\newtheorem{notation}[thm]{Notation}

\newtheorem{warning}[thm]{Warning}
\theoremstyle{remark}
\newtheorem{rem}[thm]{Remark}
\numberwithin{equation}{section}

\ifpersonal
\newcommand*{\personal}[1]{\textcolor[rgb]{0.6,0.6,1}{(Personal: #1)}}
\newcommand*{\todo}[1]{\textcolor{red}{(Todo: #1)}}
\else
\newcommand*{\personal}[1]{\ignorespaces}
\newcommand*{\todo}[1]{\ignorespaces}
\fi

\newcommand{\rB}{\mathrm B}
\newcommand{\rD}{\mathrm D}
\newcommand{\rH}{\mathrm H}
\newcommand{\rI}{\mathrm I}
\newcommand{\rL}{\mathrm L}

\newcommand{\rR}{\mathrm R}
\newcommand{\rb}{\mathrm b}

\newcommand{\rh}{\mathrm h}

\newcommand{\fU}{\mathfrak U}
\newcommand{\fV}{\mathfrak V}
\newcommand{\fW}{\mathfrak W}
\newcommand{\fX}{\mathfrak X}

\newcommand{\cA}{\mathcal A}

\newcommand{\cC}{\mathcal C}
\newcommand{\cD}{\mathcal D}
\newcommand{\cE}{\mathcal E}
\newcommand{\cF}{\mathcal F}
\newcommand{\cH}{\mathcal H}
\newcommand{\cG}{\mathcal G}

\newcommand{\cJ}{\mathcal J}
\newcommand{\cK}{\mathcal K}

\newcommand{\cM}{\mathcal M}

\newcommand{\cO}{\mathcal O}

\newcommand{\cS}{\mathcal S}
\newcommand{\cT}{\mathcal T}
\newcommand{\cU}{\mathcal U}
\newcommand{\cV}{\mathcal V}

\newcommand{\cX}{\mathcal X}

\newcommand{\cZ}{\mathcal Z}
\DeclareFontFamily{U}{BOONDOX-calo}{\skewchar\font=45 }
\DeclareFontShape{U}{BOONDOX-calo}{m}{n}{<-> s*[1.05] BOONDOX-r-calo}{}
\DeclareFontShape{U}{BOONDOX-calo}{b}{n}{<-> s*[1.05] BOONDOX-b-calo}{}
\DeclareMathAlphabet{\mathcalboondox}{U}{BOONDOX-calo}{m}{n}

\makeatletter
\let\save@mathaccent\mathaccent
\newcommand*\if@single[3]{%
	\setbox0\hbox{${\mathaccent"0362{#1}}^H$}%
	\setbox2\hbox{${\mathaccent"0362{\kern0pt#1}}^H$}%
	\ifdim\ht0=\ht2 #3\else #2\fi
}
\newcommand*\rel@kern[1]{\kern#1\dimexpr\macc@kerna}
\newcommand*\widebar[1]{\@ifnextchar^{{\wide@bar{#1}{0}}}{\wide@bar{#1}{1}}}
\newcommand*\wide@bar[2]{\if@single{#1}{\wide@bar@{#1}{#2}{1}}{\wide@bar@{#1}{#2}{2}}}
\newcommand*\wide@bar@[3]{%
	\begingroup
	\def\mathaccent##1##2{%
		\let\mathaccent\save@mathaccent
		\if#32 \let\macc@nucleus\first@char \fi
		\setbox\z@\hbox{$\macc@style{\macc@nucleus}_{}$}%
		\setbox\tw@\hbox{$\macc@style{\macc@nucleus}{}_{}$}%
		\dimen@\wd\tw@
		\advance\dimen@-\wd\z@
		\divide\dimen@ 3
		\@tempdima\wd\tw@
		\advance\@tempdima-\scriptspace
		\divide\@tempdima 10
		\advance\dimen@-\@tempdima
		\ifdim\dimen@>\z@ \dimen@0pt\fi
		\rel@kern{0.6}\kern-\dimen@
		\if#31
		\overline{\rel@kern{-0.6}\kern\dimen@\macc@nucleus\rel@kern{0.4}\kern\dimen@}%
		\advance\dimen@0.4\dimexpr\macc@kerna
		\let\final@kern#2%
		\ifdim\dimen@<\z@ \let\final@kern1\fi
		\if\final@kern1 \kern-\dimen@\fi
		\else
		\overline{\rel@kern{-0.6}\kern\dimen@#1}%
		\fi
	}%
	\macc@depth\@ne
	\let\math@bgroup\@empty \let\math@egroup\macc@set@skewchar
	\mathsurround\z@ \frozen@everymath{\mathgroup\macc@group\relax}%
	\macc@set@skewchar\relax
	\let\mathaccentV\macc@nested@a
	\if#31
	\macc@nested@a\relax111{#1}%
	\else
	\def\gobble@till@marker##1\endmarker{}%
	\futurelet\first@char\gobble@till@marker#1\endmarker
	\ifcat\noexpand\first@char A\else
	\def\first@char{}%
	\fi
	\macc@nested@a\relax111{\first@char}%
	\fi
	\endgroup
}
\makeatother

\newcommand{\oR}{\widebar R}

\newcommand{\hf}{\widehat f}
\newcommand{\hA}{\widehat A}

\newcommand{\hC}{\widehat C}

\newcommand{\hZ}{\hat Z}

\newcommand{\hPhi}{\widehat{\mathfrak{Coh}}}
\newcommand{\hPsi}{\widehat{\mathfrak W}}

\newcommand{\hrP}{\widehat{\mathfrak{Perf}}}

\newcommand{\tR}{\widetilde R}

\newcommand{\oF}{\overline{F}}

\newcommand{\PSh}{\mathrm{PSh}}
\newcommand{\Sh}{\mathrm{Sh}}

\newcommand{\Geom}{\mathrm{Geom}}
\newcommand{\LPr}{\mathcal{P}\mathrm{r}^\rL}
\newcommand{\RPr}{\mathcal{P}\mathrm{r}^\rR}
\newcommand{\LPromega}{\mathcal{P}\mathrm{r}^{\rL, \omega}}
\newcommand{\LPromegast}{\mathcal{P}\mathrm{r}^{\rL, \omega}_{\mathrm{Ex}}}

\newcommand{\rSet}{\mathrm{Set}}

\newcommand{\tauet}{\tau_\mathrm{\acute{e}t}}

\newcommand{\Modh}{\textrm{-}\mathrm{Mod}^\heartsuit}

\newcommand{\Coh}{\mathrm{Coh}}
\newcommand{\Cohb}{\mathrm{Coh}^\mathrm{b}}
\newcommand{\Cohh}{\mathrm{Coh}^\heartsuit}
\newcommand{\QCohh}{\mathrm{QCoh}^\heartsuit}

\newcommand{\Sch}{\mathrm{Sch}}

\newcommand{\Aff}{\mathrm{Aff}}

\newcommand{\bfMap}{\mathbf{Map}}

\newcommand{\Ank}{\mathrm{An}_k}

\newcommand{\trunc}{\mathrm{t}_0}

\newcommand{\CAlg}{\mathrm{CAlg}}

\newcommand{\Ind}{\mathrm{Ind}}

\newcommand{\St}{\mathbf{St}}
\newcommand{\dSt}{\mathbf{dSt}}

\newcommand{\IndCoh}{\mathrm{IndCoh}}
\newcommand{\QCoh}{\mathrm{QCoh}}
\newcommand{\Perf}{\mathrm{Perf}}

\newcommand{\fib}{\mathrm{fib}}
\newcommand{\cofib}{\mathrm{cofib}}
\newcommand{\stMap}{\mathrm{Map}^{\mathrm{st}}}

\newcommand{\Cat}{\mathrm{Cat}}
\newcommand{\AbCat}{\mathrm{AbCat}}
\newcommand{\bfCoh}{\mathbf{Coh}}
\newcommand{\bfPerf}{\mathbf{Perf}}
\newcommand{\bfQCoh}{\mathbf{QCoh}}
\newcommand{\Catst}{\Cat_\infty^{\mathrm{Ex}}}
\newcommand{\Catstmon}{\Cat_\infty^{\mathrm{Ex},\otimes}}
\newcommand{\Catstidem}{\Cat_\infty^{\mathrm{Ex}, \mathrm{idem}}}
\newcommand{\Catstlc}{\Cat_\infty^{\mathrm{Ex}, \mathrm{l.c.}}}
\newcommand{\Catstlb}{\Cat_\infty^{\mathrm{Ex}, \mathrm{l.b.}}}

\newcommand{\bfBun}{\mathbf{Bun}}
\newcommand{\Bun}{\mathrm{Bun}}
\newcommand{\Bunhat}{\mathbf{B\widehat{un}}}
\newcommand{\fCoh}{\mathfrak{Coh}}
\newcommand{\fPerf}{\mathfrak{Perf}}
\newcommand{\oPhi}{\overline{\fCoh}}

\newcommand{\llb}{[\![}
\newcommand{\rrb}{]\!]}

\newcommand{\an}{^\mathrm{an}}

\newcommand{\et}{_\mathrm{\acute{e}t}}
\newcommand{\Et}{_\mathrm{\acute{E}t}}

\newcommand{\inv}{^{-1}}

\newcommand{\canal}{$\mathbb C$-analytic\xspace}

\newcommand{\op}{^\mathrm{op}}
\newcommand{\Cech}{\check{\mathcal C}}

\usetikzlibrary{decorations.markings} 
\tikzset{
  closed/.style = {decoration = {markings, mark = at position 0.5 with { \node[transform shape, xscale = .8, yscale=.4] {/}; } }, postaction = {decorate} },
  open/.style = {decoration = {markings, mark = at position 0.5 with { \node[transform shape, scale = .7] {$\circ$}; } }, postaction = {decorate} }
}

\DeclareMathOperator{\Fun}{Fun}

\DeclareMathOperator{\Hom}{Hom}

\DeclareMathOperator{\Map}{Map}

\DeclareMathOperator{\Sp}{Sp}

\DeclareMathOperator{\SpB}{Sp_\mathrm{B}}
\DeclareMathOperator{\Spec}{Spec}
\DeclareMathOperator{\Spf}{Spf}

\DeclareMathOperator{\Sym}{Sym}

\DeclareMathOperator*{\colim}{colim}

\begin{document}

\title{Formal glueing along non-linear flags}

\author{Benjamin HENNION}
\address{Benjamin HENNION, Max Planck Institute für Mathematik, Bonn}
\email{hennion@mpim-bonn.mpg.de }

\author{Mauro PORTA}
\address{Mauro PORTA, University of Pennsylvania, David Rittenhouse Laboratory, Pennsylvania, PA, USA}
\email{mauro.porta@imj-prg.fr}

\author{Gabriele VEZZOSI}
\address{Gabriele VEZZOSI, Dipartimento di Matematica ed Informatica ``Ulisse Dini'', Firenze, Italy}
\email{gabriele.vezzosi@unifi.it}
\date{\today}

\begin{abstract}
	In this paper we prove formal glueing along an arbitrary closed substack $Z$ of an arbitrary Artin stack $X$ (locally of finite type over a field $k$), for the stacks of (almost) perfect complexes , and of $G$-bundles on $X$ (for $G$ a smooth affine algebraic $k$-group scheme). By iterating this result, we get a decomposition of these stacks along an arbitrary nonlinear flag of closed substacks in $X$. By taking points over the base field, we deduce from this both a formal glueing, and a flag-related decomposition formula for the corresponding symmetric monoidal derived $\infty$-categories of (almost) perfect modules.
	When $X$ is a quasi-compact and quasi-separated scheme, we also prove a localization theorem for almost perfect complexes on $X$, which parallels Thomason's localization results for perfect complexes. This is one of the main ingredients we need to provide a global characterization of the category of almost perfect complexes on the punctured formal neighbourhood.
	We then extend all of the previous results - i.e. the formal glueing and flag-decomposition formulas - to the case when $X$ is a derived Artin stack (locally almost of finite type over a field $k$), for the derived versions of the stacks of (almost) perfect modules, and of $G$-bundles on $X$. We close the paper by highlighting some expected progress in the subject matter of this paper, related to a Geometric Langlands  program for higher dimensional varieties. In an Appendix (for $X$ a variety), we give a precise comparison between our formal glueing results  and the rigid-analytic approach of Ben-Bassat and Temkin. 
\end{abstract}

\maketitle

\tableofcontents

\section{Introduction}

Let $k$ be field. For a $k$-scheme $X$ and a closed subscheme $Z$, with open complement $U$, roughly speaking, \emph{formal glueing} (whenever it exists) is a way of describing geometric objects of some kind defined on $X$, as geometric objects of the same kind defined on $U$ and on the formal completion $\hZ$ of $Z$ in $X$, that are suitably \emph{compatible}. The notion of compatibility here is a delicate one. Morally, guided by a purely topological intuition, one would like to say that a geometric object $\cF_U$ and a geometric object $\cF_{\hZ}$ are compatible if they restrict to the same object (or to an equivalent one) on the intersection $U \cap \hZ= \hZ \smallsetminus Z$.
However, strictly speaking, in algebraic geometry this intersection is empty, so the notion of compatibility cannot be the naive one, at least if we want to remain within the realm of algebraic geometry, i.e. we want to keep viewing $X$, $U$ and $\hZ$ as algebro-geometric objects. 
Many formal glueing results have been proved, and we give here only a probably non exhaustive list: \cite{Weil_Adeles_1982}, \cite{ArtinII}, \cite{Ferrand-Raynaud}, \cite{Beauville-Laszlo}, \cite{MB}, \cite{Ben-Bassat_Temkin_Tubular_2013}, \cite{Bhatt_algebraization_2014}, \cite{Schappi_descent_2015}, \cite{Hall_Rydh_2016}. \medbreak

In this paper we address the general formal glueing problem when the geometric objects above are
\begin{itemize}
	\item almost perfect modules (also known as pseudo-coherent complexes) on $X$ ,
	\item perfect modules on $X$,
	\item $G$-bundles on $X$, for $G$ a smooth affine $k$-group scheme,
\end{itemize}
and $X$ is an arbitrary Artin stack locally of finite presentation over $k$.\\
By carefully studying descent (cf.\ Sections \ref{subsec:etale_descent}, \ref{subsec:perfect_complexes}, and \ref{section:gbundles}), we are able to prove formal glueing for the \emph{stacks} $\bfCoh_{X}^{\otimes,-}$, $\bfPerf^\otimes_X$, $\mathbf{Bun}_G(X)$ of these geometric objects (cf.\ \cref{formalglueing} and \cref{section:gbundles}), understood as stacks with values in symmetric monoidal $k$-dg-categories or, equivalently, in symmetric monoidal stable $k$-linear $\infty$-categories, for $\bfCoh_{X}^{\otimes,-}$, $\bfPerf^\otimes_X$, and as a stack in $\infty$-groupoids for $\mathbf{Bun}_G(X)$.
To our knowledge, this is the first result where the formal reconstruction for \emph{stacks} of geometric objects like $\bfCoh_{X}^{\otimes,-}$, and $\bfPerf^\otimes_X$ on $X$, is considered. As a consequence, by taking $k$-points of these stacks, we get more classical versions of formal glueing for the derived $\infty$-categories of almost perfect, and perfect modules on $X$.  None of this results seem to be direct consequences of existing formal glueing results in the literature. Notice also, that while the formal glueing problem for $\mathbf{Bun}_G(X)$ has been considered before (e.g. by Beauville-Laszlo in \cite{Beauville-Laszlo}), the generality of the results proved in this paper seems to be new. \medbreak

Our main idea is to interpret the compatibility between geometric objects over $U$ and over $\hZ$ in non-commutative algebraic geometry. Even though the (formal) scheme-theoretic intersection $U \cap \hZ = \hZ \smallsetminus Z$ does not exist, we can construct a functor that plays the role of this punctured formal neighborhood (see  Section \ref{subsec:punctured_formal_neighbourhood}, where it is denoted by $\hPsi^\circ_{Z/X}$ and it is a functor from the big affine \'etale site of $X$ to stacks).
It is important to remark that even though this functor does not have descent, (almost) perfect modules, and $G$-bundles on this functor (simply defined by post-composition with $\bfCoh^{\otimes,-}, \bfPerf^\otimes \colon \St_k\op \to \Catstmon$, and with $\mathbf{Bun}_G \colon \St_k\op \to \cS$\footnote{This is our notation for the $\infty$-category of $\infty$-groupoids.}) \emph{do} satisfy descent.
This is indeed the content of \cref{thm:etale_descent} and \cref{prop:bunGpunctdescent}.
This brings us to the first main results of this paper, that we gather in the following:

\begin{thm-intro}[{see\ \cref{def:coherentPFN} and Theorems \ref{prop:one_step_flag_decomposition_coh} \& \ref{thm:one_step_flag_decomposition_perf} \& \ref{thm:gbundles}}] \label{thm-intro:one_step_decomposition}
	Let $X$ be an Artin stack, locally of finite type over $k$, $Z \hookrightarrow X$ be a closed substack and $U = X \smallsetminus Z$ be the open complement.
	Then there exist stacks $\bfCoh^{\otimes,-}_{\hZ \smallsetminus Z}$, $\bfPerf^\otimes_{\hZ \smallsetminus Z}$ (over $(\Aff_k, \tauet)$ and with values in $\Catstmon$), and $\mathbf{Bun}_G(\hZ \smallsetminus Z)$ (over $(\Aff_k, \tauet)$ and with values in $\infty$-groupoids),  equipped with canonical maps
	\begin{gather*}
		\bfCoh^{\otimes,-}_U \to \bfCoh^{\otimes,-}_{\hZ \smallsetminus Z}, \quad \bfCoh^{\otimes,-}_{\hZ} \to \bfCoh^{\otimes,-}_{\hZ \smallsetminus Z} , \\
		\bfPerf^\otimes_U \to \bfPerf^\otimes_{\hZ \smallsetminus Z}, \quad \bfPerf^\otimes_{\hZ} \to \bfPerf^\otimes_{\hZ \smallsetminus Z} , \\
		\mathbf{Bun}_G(U) \to \mathbf{Bun}_G(\hZ \smallsetminus Z), \quad \mathbf{Bun}_G(\hZ) \to \mathbf{Bun}_G(\hZ \smallsetminus Z)
	\end{gather*}
	that yield equivalences
	\begin{gather*}
		\bfCoh^{\otimes,-}_X \simeq \bfCoh^{\otimes,-}_U \times_{\bfCoh^{\otimes,-}_{\hZ \smallsetminus Z}} \bfCoh^{\otimes,-}_{\hZ}, \\
		\bfPerf^\otimes_X \simeq \bfPerf^\otimes_U \times_{\bfPerf^\otimes_{\hZ \smallsetminus Z}} \bfPerf^\otimes_{\hZ} ,
	\end{gather*}
	and 
	\[ \mathbf{Bun}_G(X) \simeq \mathbf{Bun}_G(U) \times_{\mathbf{Bun}_G(\hZ \smallsetminus Z)} \mathbf{Bun}_G(\hZ) . \]
\end{thm-intro}

We then use these stacky formal glueing results for $\bfCoh_{X}^{\otimes,-}$, $\bfPerf^\otimes_X$ in \cref{sec:formal_glueing} to prove the result that actually was at the origin of this paper. Given a  flag of substacks in $X$, we can, with some care, iterate our formal glueing \cref{thm-intro:one_step_decomposition}, to get a \emph{flag decomposition} of the stacks of symmetric monoidal derived $\infty$-categories of (almost) perfect modules on $X$, and for the stack of $G$-bundles on $X$:

\begin{thm-intro}[{see Corollaries \ref{cor:full_flag_decomposition} \& \ref{flagbung}}] \label{thm-intro:full_flag_decomposition}
	Let $X$ be an Artin stack, locally of finite type over $k$, and let $\cZ \coloneqq (Z_0, Z_1, \ldots, Z_n)$ be a flag on $X$.
	Let $\fV_{i+1}$ be the formal open substack which is complementary to the closed immersion $Z_{i+1} \hookrightarrow \widehat{Z_i}$.
	Then there are canonical equivalences (of appropriate stacks)
	\[ \bfCoh^{\otimes,-}_X \simeq \bfCoh^{\otimes,-}_{\fV_1} \times_{\bfCoh^{\otimes,-}_{\widehat{Z_1} \smallsetminus Z_1}} \left( \bfCoh^{\otimes,-}_{\fV_2} \times_{\bfCoh^{\otimes,-}_{\widehat{Z_2} \smallsetminus Z_2}} \left( \cdots \left( \bfCoh^{\otimes,-}_{\fV_n} \times_{\bfCoh^{\otimes,-}_{\widehat{Z_n} \smallsetminus Z_n}} \bfCoh^{\otimes,-}_{\widehat{Z_n}} \right) \cdots \right) \right) ,\]
	\[ \bfPerf^\otimes_X \simeq \bfPerf^\otimes_{\fV_1} \times_{\bfPerf^\otimes_{\widehat{Z_1} \smallsetminus Z_1}} \left( \bfPerf^\otimes_{\fV_2} \times_{\bfPerf^\otimes_{\widehat{Z_2} \smallsetminus Z_2}} \left( \cdots \left( \bfPerf^\otimes_{\fV_n} \times_{\bfPerf^\otimes_{\widehat{Z_n} \smallsetminus Z_n}} \bfPerf^\otimes_{\widehat{Z_n}} \right) \cdots \right) \right) , \]
	and 
	\[ \bfBun_G(X) \simeq \bfBun_G(\fV_1) \times_{\bfBun_G(\widehat{Z_1} \smallsetminus Z_1)} \left( \bfBun_G(\fV_2) \times_{\bfBun_G(\widehat{Z_2} \smallsetminus Z_2)} \left( \cdots \left( \bfBun_G(\fV_n) \times_{\bfBun_G(\widehat{Z_n} \smallsetminus Z_n)} \bfBun_G(\widehat{Z_n}) \right) \cdots \right) \right) . \]
\end{thm-intro}

These decomposition equivalences depend on the fixed flag, so that different flags will a priori give different decompositions.
This is of particular interest in our program, of which this paper can be considered a first step, of understanding what the Geometric Langlands program for higher dimensional varieties should really be. We will say a few more words about this at the end of this Introduction.\medbreak

\begin{rem} Let us mention that, when $X$ is a noetherian scheme, a reconstruction result in the spirit of our \cref{thm-intro:full_flag_decomposition}, has been recently established by M.\ Groechenig in the very interesting paper \cite{Groechenig_Adelic_2015}. Groechenig proves an \emph{all-flags} reconstruction, i.e. an adelic decomposition result, where one considers \emph{all} flags in $X$ at once (instead of just a fixed, non-necessarily full one). The two results are related, and we suspect that the main results in \cite{Groechenig_Adelic_2015} should follow from \cref{thm-intro:full_flag_decomposition} (by applying it to each simplicial level), but the exact relation between the two results certainly requires some further investigation.
\end{rem}

At the initial stage of this project, we drew some inspiration from the paper of O.\ Ben-Bassat and M.\ Temkin \cite{Ben-Bassat_Temkin_Tubular_2013}.
In \cite{Ben-Bassat_Temkin_Tubular_2013}, for $X$ a $k$-scheme of finite type, the punctured formal neighbourhood is constructed as a \emph{non-archimedean} space, while in our setting it is constructed as an object in non-commutative geometry.
However, putting aside the choice of the language to be used, our paper introduces several novelties to the subject.
It seems relevant to emphasise them here:
\begin{enumerate}
	\item In \cite{Ben-Bassat_Temkin_Tubular_2013}, only a weaker version of \cref{thm-intro:full_flag_decomposition} is obtained: on one side, only flags of length one are considered, and, on the other side, only the \emph{abelian category} of (discrete) coherent sheaves is taken into account. Our result implies the main result of \cite{Ben-Bassat_Temkin_Tubular_2013} by considering flags of length one and taking $k$-points in the equivalences of \cref{thm-intro:full_flag_decomposition} (see \cref{cor:comparison} for a precise comparison between our construction and the one of \cite{Ben-Bassat_Temkin_Tubular_2013}).
	\item In \cite{Ben-Bassat_Temkin_Tubular_2013} it is always assumed that $X$ and $Z$ are $k$-varieties. Here we show that this assumption is not necessary. This is an important point for us: the proof of our decomposition theorem for flags of length bigger that one relies heavily on the fact that our one-step decomposition holds for non-reduced schemes (and, actually, Artin stacks).
\end{enumerate}
There is a third important point that deserves to be underlined, which is tightly related to the third main result of this paper.
In order to better explain this point, let us remark that the stable $\infty$-category of almost perfect modules on the formal punctured neighbourhood
\[ \Coh^-(\hZ \smallsetminus Z) \coloneqq \bfCoh^{\otimes,-}_{\hZ \smallsetminus Z}(\Spec(k)) \]
is an important \emph{global} invariant of the embedding $Z \hookrightarrow X$: its construction is the very heart of the proof of \cref{thm-intro:full_flag_decomposition}.
However, despite being a global invariant, the easiest way to define it is by gluing local data.
In fact, if $X = \Spec(A)$ and $Z = \Spec(A/J)$, then it is certainly not surprising that there is a natural equivalence
\[ \Coh^-(\hZ \smallsetminus Z) \simeq \Coh^-(\Spec(\widehat{A}) \smallsetminus Z) , \]
where $\widehat{A}$ is the formal completion of $A$ at $J$, and $Z$ is viewed as a closed subscheme of $\Spec(\widehat{A})$ via the isomorphism $\widehat{A} / J \simeq A / J$.
Now, \cref{thm:etale_descent} guarantees that the assignment sending an affine $U = \Spec(A_U)$ mapping to $X$ to the stable $\infty$-category $\Coh^-(\widehat{Z_U} \smallsetminus Z_U)$ (where $Z_U = U \times_X Z$ and $\widehat{Z_U}$ denotes the formal completion of $U$ along $Z_U$) is actually a sheaf.
This enables us to define $\Coh^-(\hZ \smallsetminus Z)$ even in the case where $X$ is not affine, via a \emph{sheafification} procedure.
However, an explicit description of the associated sheaf is, as usual, not easy to get. When $X$ is a quasi-compact and quasi-separated \emph{scheme}, we manage to improve this situation (see \cref{subsec:case_schemes}). Our main result in this respect is the following:

\begin{thm-intro}[{see \cref{prop:gluing_right_adjoints} and \cref{thm:Zariski_computational_tool}}] \label{thm-intro:case_schemes}
	Let $X$ be a \emph{quasi-compact} and \emph{quasi-separated} scheme locally of finite type over $k$, and let $Z \hookrightarrow X$ be a closed subscheme.
	Then the canonical functor
	\[ j^* \colon \Coh^-(\hZ) \to \Coh^-(Z) \]
	admits a right adjoint $j_* \colon \Coh^-(Z) \to \Coh^-(\hZ)$.
	Furthermore, if $\cK(X)$ denotes the smallest full stable subcategory of $\Coh^-(\hZ)$ which is left complete and contains the essential image of $j_*$, then
	\[ \cK(X) \hookrightarrow \Coh^-(\hZ) \to \Coh^-(\hZ \smallsetminus Z) \]
	is a cofiber sequence in $\Catst$.
\end{thm-intro}

It seems worth spending a few more words about the proof of \cref{thm-intro:case_schemes}.
This is a rather technical proof with two main ingredients. The first one is a very delicate base-change property for coherent modules on the punctured formal neighborhood (see \cref{cor:PFN_right_adjointable}). The second one is a \emph{localization theorem} for almost perfect modules, in the same vein of Thomason's localization theorem for perfect modules (cf.\ \cite{Thomason_Trobaugh_Algebraic_K_theory_1990,Neeman_Connection_K_theory_localization}).
More precisely, in \cref{sec:localization_almost_perfect} we prove:

\begin{thm-intro}[{see \cref{thm:localization_for_coh_minus}}] \label{thm-intro:localization_almost_perfect}
	Let $X$ be a quasi-compact quasi-separated scheme which is locally Noetherian over $k$.
	Let $i \colon U \to X$ be a quasi-compact open immersion.
	Let $\cK$ be the kernel of the pullback $i^* \colon \Coh^-(X) \to \Coh^-(U)$.
	Then
	\[ \cK \hookrightarrow \Coh^-(X) \to \Coh^-(U) \]
	is a cofiber sequence in $\Catst$.
\end{thm-intro}

We could not locate this result in the literature.
Since this seems a result of independent interest, that may have other interesting applications, we decided to include a full proof.
The closest results existing in the literature we are aware of, are the already mentioned localization theorem of Thomason, as well as a localization theorem for $\Cohb$ proved by D.\ Gaitsgory in \cite[\S 4.1]{Gaitsgory_IndCoh}.
The proof of the latter result passes through a similar localization theorem for $\IndCoh$, which in turn can be deduced from the analogous property of $\QCoh$.
On the other hand, our proof of \cref{thm-intro:localization_almost_perfect} reduces to the already known statement for $\Cohb$.
This reduction is achieved via a very careful analysis of the behaviour of $t$-structures with respect to the operation of forming Verdier quotients in $\Catst$.
As a byproduct of our proof, we obtain that, keeping the same notations as in \cref{thm-intro:localization_almost_perfect}, and letting $\cK^b$ be the kernel of the pullback $i^* \colon \Cohb(X) \to \Cohb(U)$, the sequence
\[ \cK^b \hookrightarrow \Cohb(X) \to \Cohb(U) \]
is a cofiber sequence not only in $\Catstidem$ but also in $\Catst$.
This is, in itself, an improvement on the result obtained in \cite{Gaitsgory_IndCoh}.
\medbreak

We now come to the last main result in this paper. The main motivation that brought us to writing this paper has been the attempt to understanding the current status of the Geometric Langlands program for surfaces (see below for a more detailed account of our motivations).
From this point of view, Theorems \ref{thm-intro:one_step_decomposition} and \ref{thm-intro:full_flag_decomposition} are not quite sufficient: what is actually needed is a version of the same results stated there, but for the \emph{derived versions} of the stacks involved.
It is not particularly hard to obtain such derived versions: this is precisely the content of \cref{sec:DAG}, where we show how the derived counterparts of our main flag decomposition results can be deduced from the underived statements.
We can therefore summarize the last couple of main results of this paper as follows:

\begin{thm-intro}[{see \cref{def:derived_coherent_PFN} and Theorems \ref{prop:one_step_flag_decomposition_coh_derived} \& \ref{thm:flag_bung_derived}}]
	Let $X$ be a \emph{derived} Artin stack locally of finite presentation over $k$, $Z \hookrightarrow X$ be a \emph{derived} closed substack, and $U = X \smallsetminus Z$ be the open complement.
	Then there exist \emph{derived} stacks $\rR \bfCoh^{\otimes,-}_{\hZ \smallsetminus Z}$, $\rR \bfPerf^\otimes_{\hZ \smallsetminus Z}$ (over $(\mathrm{dAff}_k, \tauet)$ and with values in $\Catstmon$), and $\rR \mathbf{Bun}_G(\hZ \smallsetminus Z)$ (over $(\mathrm{dAff}_k, \tauet)$ and with values in $\infty$-groupoids),  equipped with canonical maps
	\begin{gather*}
		\rR \bfCoh^{\otimes,-}_U \to \rR \bfCoh^{\otimes,-}_{\hZ \smallsetminus Z}, \quad \rR \bfCoh^{\otimes,-}_{\hZ} \to \rR \bfCoh^{\otimes,-}_{\hZ \smallsetminus Z} , \\
		\rR \bfPerf^\otimes_U \to \rR \bfPerf^\otimes_{\hZ \smallsetminus Z}, \quad \rR \bfPerf^\otimes_{\hZ} \to \rR \bfPerf^\otimes_{\hZ \smallsetminus Z} , \\
		\rR \mathbf{Bun}_G(U) \to \rR \mathbf{Bun}_G(\hZ \smallsetminus Z), \quad \rR \mathbf{Bun}_G(\hZ) \to \rR \mathbf{Bun}_G(\hZ \smallsetminus Z)
	\end{gather*}
	that yield equivalences
	\begin{gather*}
		\rR \bfCoh^{\otimes,-}_X \simeq \rR \bfCoh^{\otimes,-}_U \times_{\rR \bfCoh^{\otimes,-}_{\hZ \smallsetminus Z}} \rR \bfCoh^{\otimes,-}_{\hZ}, \\
		\rR \bfPerf^\otimes_X \simeq \rR \bfPerf^\otimes_U \times_{\rR \bfPerf^\otimes_{\hZ \smallsetminus Z}} \rR \bfPerf^\otimes_{\hZ} ,
	\end{gather*}
	and 
	\[ \rR \mathbf{Bun}_G(X) \simeq \rR \mathbf{Bun}_G(U) \times_{\rR \mathbf{Bun}_G(\hZ \smallsetminus Z)} \rR \mathbf{Bun}_G(\hZ) . \]
\end{thm-intro}
It is important to notice that this theorem is stronger than \cref{thm-intro:one_step_decomposition} even when $X$ is a variety. It indeed extends the formal glueing to the derived structures of the involved stacks.

From Theorem 5, we can easily deduce, as in the underived setting, the following derived flag decomposition theorem:

\begin{thm-intro}[{see Corollaries \ref{cor:full_flag_decomposition_derived} \& \ref{cor:flag_bung_derived}}]
	Let $X$ a \emph{derived} Artin locally of finite presentation over $k$ and let $\cZ \coloneqq (Z_0, Z_1, \ldots, Z_n)$ be a flag on $X$.
	Let $\fV_{i+1}$ be the formal open derived substack which is complementary to the closed immersion $Z_{i+1} \hookrightarrow \widehat{Z_i}$.
	Then there are canonical equivalences (of appropriate \emph{derived} stacks)
	\begin{gather*}
		\rR \bfCoh^{\otimes,-}_X \simeq \rR \bfCoh^{\otimes,-}_{\fV_1} \times_{\rR \bfCoh^{\otimes,-}_{\widehat{Z_1} \smallsetminus Z_1}} \left( \rR \bfCoh^{\otimes,-}_{\fV_2} \times_{\rR \bfCoh^{\otimes,-}_{\widehat{Z_2} \smallsetminus Z_2}} \left( \cdots \left( \rR \bfCoh^{\otimes,-}_{\fV_n} \times_{\rR \bfCoh^{\otimes,-}_{\widehat{Z_n} \smallsetminus Z_n}} \rR \bfCoh^{\otimes,-}_{\widehat{Z_n}} \right) \cdots \right) \right) , \\
		\rR \bfPerf^\otimes_X \simeq \rR \bfPerf^\otimes_{\fV_1} \times_{\rR \bfPerf^\otimes_{\widehat{Z_1} \smallsetminus Z_1}} \left( \rR \bfPerf^\otimes_{\fV_2} \times_{\rR \bfPerf^\otimes_{\widehat{Z_2} \smallsetminus Z_2}} \left( \cdots \left( \rR \bfPerf^\otimes_{\fV_n} \times_{\rR \bfPerf^\otimes_{\widehat{Z_n} \smallsetminus Z_n}} \rR \bfPerf^\otimes_{\widehat{Z_n}} \right) \cdots \right) \right) ,
	\end{gather*}
	and 
	\[ \begin{split}
	\rR \bfBun_G(X) \simeq \rR \bfBun_G(\fV_1) \times_{\rR \bfBun_G(\widehat{Z_1} \smallsetminus Z_1)} \bigg( \rR & \bfBun_G (\fV_2) \times_{\rR \bfBun_G(\widehat{Z_2} \smallsetminus Z_2)} \bigg( \cdots \\
	& \cdots \left( \rR \bfBun_G(\fV_n) \times_{\rR \bfBun_G(\widehat{Z_n} \smallsetminus Z_n)} \rR \bfBun_G(\widehat{Z_n}) \right) \cdots \bigg) \bigg) .
	\end{split} \]
\end{thm-intro}

\medbreak

\paragraph{\textbf{Motivations and vistas.}}

This paper grew out of an attempt to understand the current status and difficulties of a Geometric Langlands program for surfaces (and more generally, for varieties of dimension $>1$), and, though still a far cry from that, it might be regarded as our first step in this direction.
Despite some very interesting attempts (see \cite{Ginzburg_Kapranov_Langlands_reciprocity_1995, Kapranov_Analogies_Langlands_1995}), such a program has no precise and definitive formulation at the moment, and the main purpose of our long term program is to uncover at least some of the geometric structures that should be involved in such a formulation. \medbreak

Our basic idea is the following (we spell it out for a surface $X=S$, for simplicity of exposition), and the reader will find some more details on this in \cref{vistas}.
Let $\mathbf{D}_S$ be either one of the stacks $\bfCoh_S^{-}$, $\bfPerf_S$ and $\mathbf{Bun}_G(S)$, or any of their corresponding derived versions,  considered above.
To any flag $u = (S, C, x)$ in $S$, we may associate the stack appearing in the right hand side of our decomposition theorem \ref{thm-intro:full_flag_decomposition}:
\[ \mathbf{D}^{\mathrm{fl}}_{u,S} \coloneqq \mathbf{D}_{S \smallsetminus C} \times_{\mathbf{D}_{\hC \smallsetminus C}} \left( \mathbf{D}_{\hC \smallsetminus x} \times_{\mathbf{D}_{\hat{x} \smallsetminus x}} \mathbf{D}_{\hat{x}} \right) . \]
To $u$, we can also associate the category of $k$-points of $\mathbf{D}^{\mathrm{fl}}_{u,S}$, which we denote $\rD^{\mathrm{fl}}_{u,S}$.
Then, \cref{thm-intro:full_flag_decomposition} itself allows us to informally view $\mathbf{D}^{\mathrm{fl}}_{S} \colon u \mapsto \rD^{\mathrm{fl}}_{u, S}$ as a locally constant stacks of stable $\infty$-categories over the \emph{flag Hilbert scheme} $\cH_S^{\mathrm{fl}}$ of $S$.
In particular, this suggests the existence of an action of the \'etale fundamental group of $\cH_S^{\mathrm{fl}}$ on the stack $\mathbf{D}_S$, and on the derived $\infty$-category $\rD(S)$.
This action is interesting per se, but we would like to think that this is probably just a small part of a more general ``geometric structure'' acting on $\mathbf{D}_S$.
One can in fact guess that various geometric operations on flags in $S$ (like taking disjoint unions, intersections, curve contractions when allowed, etc.) might organize themselves into some kind of operadic-like structure that acts on $\mathbf{D}_S$.
Alternatively, following the analogy with what happens in the Geometric Langlands program for curves, one might expect that $\mathbf{D}^{\mathrm{fl}}_{S}$ (or an appropriate modification thereof) carries the structure of a kind of higher factorizable sheaf on a \emph{Ran version} of the flag Hilbert scheme  $\cH_S^{\mathrm{fl}}$.
Either of these intuitions, if correct, would provide very interesting geometrical actions on $\mathbf{D}_S$, and in particular on $D(S)$.
And there is a hope to recover from this action, the Grojnowski-Nakajima's Heisenberg Lie algebra action. \medbreak

This investigation looks extremely interesting to us, especially in the case $\mathbf{D}_S = \rR \mathbf{Bun}_{G}(S)$, the derived stack of $G$-bundles on $S$, for its possible relation with the Geometric Langlands program for surfaces. We are currently investigating what the geometric action alluded above (coming from geometric operations on flags in $S$) does on $\rR \mathbf{Bun}_{G}(S)$ or, possibly, on various kind of ``sheaves'' on $\rR \mathbf{Bun}_{G}(S)$, and if it is possible to use it in order to define a Hecke-like action on $\rR \mathbf{Bun}_{G}(S)$, as it happens in the case when $S$ is a curve. More about this topic will appear in a forthcoming, separate paper. \medbreak

\paragraph{\textbf{Acknowledgments}}

GV would like to thank Bhargav Bhatt for very interesting e-mail exchanges at the beginning of our investigation on the subject matter of this paper, Michael Groechenig, Ian Grojnowski, and Nick Rozenblyum for useful discussions. BH wishes to thank Benjamin Antieau for very useful comments on a preliminary version of the paper. 
MP would like to thank Andrew Macpherson, Marco Robalo and Tony Yue Yu for several useful discussions related to the subject of this paper. BH and MP would like to thank the Perimeter Institute for providing an excellent scientific environment in which part of this project has been carried out.
We are also grateful Mikhail Kapranov, Tony Pantev, and Carlos Simpson for some inspiring conversations, suggesting that our long-term program, might be not completely out of reach.
We also thank Michael Temkin for giving us his opinion on some of the rigid-analytic proofs contained in an earlier version of this paper.

Finally, during the preparation of this paper, MP was partially supported by the Simons Collaboration on Homological Mirror Symmetry.

\begin{center} \Huge{$\sim$} \end{center}

\bigskip

\paragraph{\textbf{Notations and conventions.}}

All the categories are $\infty$-categories unless explicitly stated.
We made an effort to keep statements and proof model independent, but, whenever a model is required, we choose quasi-categories and we refer to \cite{HTT} for notations and terminology.
The symbol $\cS$ is reserved for the $\infty$-category of spaces.
We let $\LPr$ (resp.\ $\RPr$) the $\infty$-categories of presentable $\infty$-categories and functors that are left adjoints (resp.\ right adjoints) between them (cf.\ \cite[\S 5.5.3]{HTT}).
We denote by $\Catst$ the $\infty$-category of stable $\infty$-categories and exact functors between them (cf.\ \cite[\S 1.1]{Lurie_Higher_algebra}).
We let $\Catstlc$ (resp.\ $\Catstlb$) denote the $\infty$-category of stable $\infty$-categories equipped with left complete (resp.\ left bounded) $t$-structures and functors that are $t$-exact.
Furthermore, we denote by $\Catstmon$ the $\infty$-category of symmetric monoidal stable $\infty$-categories (cf.\ \cite[2.1.4.13]{Lurie_Higher_algebra}).
Finally, if $\cC$ is a stable $\infty$-category, we denote by $\stMap_\cC$ the \emph{stable} mapping space, i.e.\ the mapping space enriched in the $\infty$-category $\Sp$ of spectra. \medbreak

If $\cT$ is a presentable $\infty$-category and $(\cC, \tau)$ is a Grothendieck site, we denote by $\Sh_\cT(\cC, \tau)$ the (presentable) $\infty$-topos of $\tau$-sheaves on $\cC$ with values in $\cT$.
If $\cT = \cS$, we simply write $\Sh(\cC, \tau)$ instead of $\Sh_{\cS}(\cC, \tau)$.
We write $\St_\cT(\cC, \tau)$ for the $\infty$-topos of hypercomplete sheaves on $(\cC, \tau)$: $\St_{\cT}(\cC, \tau) \coloneqq \Sh_{\cT}(\cC, \tau)^\wedge$ (cf.\ \cite[\S 6.5.3]{HTT}).
Given a diagram of $\infty$-categories
\[ \begin{tikzcd}
	\cC \arrow{r}{L} \arrow{d}{F} & \cD \arrow{d}{G} \\
	\cC' \arrow{r}{L'} & \cD'
\end{tikzcd} \]
which commutes up to a specified homotopy $\alpha \colon L' \circ F \to G \circ L$, we say that the square is \emph{right adjointable} if $L$ and $L'$ have right adjoints $R \colon \cD \to \cC$ and $R' \colon \cD' \to \cC'$ and the induced push-pull transformation
\[ F \circ R \to R' \circ G \]
is an equivalence (cf.\ \cite[7.3.1.1]{HTT}).
\medbreak

We work over a fixed base field $k$ (of any characteristic).
We denote by $\Aff_k$ the category of affine schemes \emph{of finite presentation} over $k$.
We equip it with the standard \'etale topology, denoted $\tauet$.
Moreover, we write $\St_k$ for the $\infty$-topos $\St(\Aff_k, \tauet)$.
For a fixed scheme $X$, we denote by $\QCoh(X)$ the stable $\infty$-category of quasi-coherent sheaves on $X$.
We endow $\QCoh(X)$ with its standard $t$-structure.
The heart $\QCohh(X)$ is the $1$-category of (discrete) quasi-coherent sheaves on $X$.
Following \cite[\S 1.3]{DAG-VIII}, we can promote the assignment $X \mapsto \QCoh(X)$ to an $\infty$-functor
\[ \bfQCoh^\otimes \colon \St_k\op \to \Catstmon . \]
Finally, $\bfCoh^{\otimes,-}$ and $\bfPerf^\otimes$ denote the full substacks of $\bfQCoh^\otimes$ spanned respectively by almost perfect modules\footnote{Let $X$ be a noetherian scheme and let $\cF \in \QCoh(X)$ be a quasi-coherent complex on $X$. Then $\cF$ is almost perfect if and only if the cohomology sheaves $\rH^i(\cF)$ are coherent sheaves and vanish for $i \gg 0$} and by perfect modules. \medbreak

In \cref{appendix:comparison} we consider $k$ as a field endowed with the trivial valuation.
We denote by $\Ank$ the category of $k$-analytic spaces, which are understood in the sense of Berkovich.
We refer to \cite{Conrad_Several_approaches_2008,Bosch_Lectures_2014,Berkovich_Vanishing_II_1996} for some background on $k$-analytic spaces.

\section{Localization for almost perfect modules} \label{sec:localization_almost_perfect}

The goal of this first section is to prove a localization theorem for almost perfect modules, in the same spirit of the localization for $\Perf$ proved by Thomason (cf.\ \cite{Thomason_Trobaugh_Algebraic_K_theory_1990,Neeman_Connection_K_theory_localization}).
This theorem is used later in \cref{subsec:case_schemes} in order to provide a global description of the category of almost perfect modules on the punctured formal neighbourhood (cf.\ \cref{thm:Zariski_computational_tool}).

We could not locate this localization result in the literature, and on the other side we expect it to be of independent interest.
For these reasons, we decided to include a detailed proof, which occupies this whole section.
Before plunging into the details, let us spend a couple of words on the proof.
The localization theorem of Thomason takes place in the $\infty$-category $\Catstidem$ of stable $\infty$-categories which are idempotent complete.
It states that if $X$ is a quasi-compact and quasi-separated scheme, $i \colon U \to X$ is a quasi-compact open embedding and $\cK \coloneqq \ker( i^* \colon \Perf(X) \to \Perf(U) )$, then
\[ \cK \hookrightarrow \Perf(X) \to \Perf(U) \]
is a cofiber sequence in $\Catstidem$.
With the technology of $\infty$-categories available nowadays, the proof can be summarized as follows: since the Ind construction establishes an equivalence of $\infty$-categories between $\Catstidem$ and $\LPromegast$, it is equivalent to prove the localization result for the $\infty$-category $\Ind(\Perf(X))$.
Since $X$ is quasi-compact and quasi-separated, we know that $\Ind(\Perf(X)) \simeq \QCoh(X)$.
Then, the result follows immediately from the fact that the functor $i_* \colon \QCoh(U) \to \QCoh(X)$ is fully faithful and commutes with arbitrary colimits.

In \cite{Gaitsgory_IndCoh} an analogous result is proven for $\IndCoh$ instead of $\QCoh$.
The key step is to prove that
\[ i_*^{\IndCoh} \colon \IndCoh(U) \to \IndCoh(X) \]
is fully faithful and commutes with arbitrary colimits.
This is achieved by reduction to the already known case of $\QCoh(X)$.
From here, taking compact objects, one deduces that the category $\Cohb(X)$ has the localization property in $\Catstidem$.

Nevertheless, this approach does not work equally well for almost perfect modules.
The reason is to be found in the difficulty of defining a pushforward functoriality for $\Ind(\Coh^-(X))$.
Our strategy consists rather in a reduction to the case of $\Cohb(X)$ via the equivalence of $\infty$-categories $\Catstlb \simeq \Catstlc$ between stable $\infty$-categories with left bounded $t$-structures (and $t$-exact functors between them) and stable $\infty$-categories with left complete $t$-structures (and $t$-exact functors between them).
In order to apply this strategy, we need some general results concerning Verdier quotients of stable $\infty$-categories and $t$-structures.
This is the content of the next subsection.

\subsection{Verdier quotients and $t$-structures} \label{subsec:Verdier_quotient_t_structure}

The key result of this subsection is \cref{prop:quotient_t_structure}, which studies the conditions under which a Verdier quotient of a stable $\infty$-category $\cC$ equipped with a $t$-structure inherits a $t$-structure in such a way that the quotient map is both left and right $t$-exact.

We start with the following well known lemma.
We include the proof for completeness.

\begin{lem} \label{lem:t_structure_Ind}
	Let $\cC$ be a stable $\infty$-category equipped with a $t$-structure $(\cC^{\le 0}, \cC^{\ge 0})$.
	Then $(\Ind(\cC^{\le 0}), \Ind(\cC^{\ge 0}))$ defines a $t$-structure on $\Ind(\cC)$ satisfying the following requirements:
	\begin{enumerate}
		\item the Yoneda embedding $\cC \to \Ind(\cC)$ is $t$-exact;
		\item the $t$-structure on $\Ind(\cC)$ is compatible with filtered colimits.
	\end{enumerate}
\end{lem}

\begin{proof}
	Let $\Ind(\cC)^{\le 0}$ be the smallest full subcategory of $\Ind(\cC)$ containing $\cC^{\le 0}$ and closed under colimits.
	It follows from \cite[1.4.4.11]{Lurie_Higher_algebra} that this defines a (unique) $t$-structure on $\Ind(\cC)$.
	Let us first prove that the Yoneda embedding $j \colon \cC \to \Ind(\cC)$ is $t$-exact.
	Since it is left $t$-exact by construction, it is enough to prove that $j$ is right $t$-exact.
	To see this, we first remark that the induced functor $\Ind(\cC^{\le 0}) \to \Ind(\cC)$ is fully faithful.
	Since the inclusion $\cC^{\le 0} \hookrightarrow \cC$ commutes with colimits, it follows that precomposition with $i$
	\[ i^* \colon \PSh(\cC) \to \PSh(\cC^{\le 0}) \]
	takes presheaves that commutes with finite limits to presheaves with the same property.
	Therefore, $i^*$ restricts to a right adjoint for $\Ind(\cC^{\le 0}) \to \Ind(\cC)$.
	This shows that $\Ind(\cC^{\le 0})$ is closed under colimits in $\Ind(\cC)$, and therefore that $\Ind(\cC)^{\le 0} \subseteq \Ind(\cC^{\le 0})$.
	The reverse inclusion follows directly from the definition, and so we obtain the equality
	\[ \Ind(\cC)^{\le 0} = \Ind(\cC^{\le 0}) . \]
	Let now $X \in \cC^{\ge 1}$ and $Y \in \Ind(\cC)^{\le 0}$.
	We can write
	\[ Y = \colim j(Y_\alpha) , \]
	where $Y_\alpha \in \cC^{\le 0}$.
	As consequence, we obtain
	\begin{align*}
	\Map_{\Ind(\cC)}(Y, j(X)) & = \lim \Map_{\Ind(\cC)}(j(Y_\alpha), j(X)) \\
	& = \lim \Map_{\cC}(Y_\alpha, X) = 0
	\end{align*}
	This shows that $j(X) \in \Ind(\cC)^{\ge 1}$ and thus completes the proof that $j$ is $t$-exact.
	
	Let us now prove that the $t$-structure on $\Ind(\cC)$ is compatible with filtered colimits.
	Let $X \colon I \to \Ind(\cC)^{\ge 0}$ be a filtered diagram and let $Y \in \Ind(\cC)^{\le 0}$.
	Write again
	\[ Y = \colim j(Y_\alpha) , \]
	with $Y_\alpha \in \cC^{\le 0}$. Then
	\[ \Map_{\Ind(\cC)}(Y, \colim_\beta X_\beta) = \lim_\alpha \colim_\beta \Map_{\cC}(Y_\alpha, X_\beta) = 0 , \]
	and so $\colim_\beta X_\beta \in \Ind(\cC)^{\ge 0}$.
	
	Let us finally prove that $\Ind(\cC)^{\ge 0} = \Ind(\cC^{\ge 0})$.
	Since $j$ is $t$-exact and $\Ind(\cC)^{\ge 0}$ is closed under filtered colimits, we see that
	\[ \Ind(\cC^{\ge 0}) \subseteq \Ind(\cC)^{\ge 0} . \]
	Now let $X \in \Ind(\cC)^{\ge 0}$.
	Write
	\[ X = \colim j(X_\alpha) . \]
	Since $\tau_{\ge 0} \colon \Ind(\cC) \to \Ind(\cC)^{\ge 0}$ is left adjoint to the inclusion $\Ind(\cC)^{\ge 0} \subset \Ind(\cC)$ (cf.\ \cite[1.2.1.5]{Lurie_Higher_algebra}), we conclude that
	\[ X \simeq \tau_{\ge 0}(X) \simeq \colim j( \tau_{\ge 0}(X_\alpha) ) . \]
	It follows that $X \in \Ind(\cC^{\ge 0})$, and therefore that
	\[ \Ind(\cC)^{\ge 0} = \Ind(\cC^{\ge 0}) . \]
	The proof is thus complete.
\end{proof}

\begin{lem} \label{lem:t_structure_Ind_restriction}
	Let $\cC$ be a stable $\infty$-category equipped with a $t$-structure $(\cC^{\le 0}, \cC^{\ge 0})$ and endow $\Ind(\cC)$ with the $t$-structure provided by \cref{lem:t_structure_Ind}.
	Let $\cD$ be any other stable $\infty$-category and let $f \colon \cC \to \cD$ be an exact functor.
	If $f$ is essentially surjective, the following conditions are equivalent:
	\begin{enumerate}
		\item $\cD$ admits a $t$-structure such that $f$ is $t$-exact.
		\item $\Ind(\cD)$ admits a $t$-structure such that $\Ind(f) \colon \Ind(\cC) \to \Ind(\cD)$ is $t$-exact.
	\end{enumerate}
\end{lem}

\begin{proof}
	Suppose first that $\cD$ admits a $t$-structure and that $f$ is $t$-exact.
	Since $\Ind(f)$ takes $\Ind(\cC^{\le 0})$ (resp.\ $\Ind(\cC^{\ge 0})$) to $\Ind(\cD^{\le 0})$ (resp.\ $\Ind(\cD^{\ge 0})$), it follows directly from \cref{lem:t_structure_Ind} that $\Ind(f)$ is $t$-exact.
	
	Suppose now that $\Ind(\cD)$ admits a $t$-structure such that $\Ind(f)$ is $t$-exact.
	Define
	\[ \cD^{\le 0} \coloneqq \cD \cap \Ind(\cD)^{\le 0}, \qquad \cD^{\ge 0} \coloneqq \cD \cap \Ind(\cD)^{\ge 0} . \]
	Since $j_\cD \colon \cD \to \Ind(\cD)$ is fully faithful we immediately see that for $X \in \cD^{\le 0}$ and $Y \in \cD^{\ge 1}$ one has
	\[ \Map_{\cD}(X, Y) = \Map_{\Ind(\cD)}(j_\cD(X), j_\cD(Y)) = 0 . \]
	Similarly, since $j_\cD$ is exact, we obtain
	\[ \cD^{\le 0}[1] \subseteq \cD^{\le 0}, \qquad \cD^{\ge 0}[-1] \subseteq \cD^{\ge 0} . \]
	Now let $X \in \cD$.
	We need to exhibit a fiber sequence
	\[ X' \to X \to X'' , \]
	where $X' \in \cD^{\le 0}$ and $X'' \in \cD^{\ge 1}$.
	Since $f$ is essentially surjective, we can choose an object $Y \in \cC$ and an equivalence $f(Y) \simeq X$.
	Pick a fiber sequence in $\cC$
	\[ Y' \to Y \to Y'' , \]
	where $Y' \in \cC^{\le 0}$ and $Y'' \in \cC^{\ge 1}$.
	Then
	\[ f(Y') \to f(Y) \to f(Y'') \]
	is again a fiber sequence in $\cC$.
	Furthermore, using the fact that $\Ind(f)$ is $t$-exact, we obtain:
	\[ j_\cD(f(Y')) = \Ind(f)(j_\cC(Y')) \in \Ind(\cD)^{\le 0}, \qquad j_\cD(f(Y'')) = \Ind(f)(j_\cC(Y'')) \in \Ind(\cD)^{\ge 1} . \]
	This shows that $f(Y') \in \cD^{\le 0}$ and $f(Y'') \in \cD^{\ge 1}$, and thus that $(\cD^{\le 0}, \cD^{\ge 0})$ defines a $t$-structure on $\cD$.
	Furthermore, it shows that, with respect to this $t$-structure, $f$ is $t$-exact.
\end{proof}

\begin{lem} \label{lem:localization_essential_surjective}
	Let $\cC$ be an $\infty$-category and let $S$ be a collection of arrows in $\cC$.
	Let $\rh_n \colon \Cat_\infty \to \Cat_\infty^{\le n}$ be the fundamental $n$-category functor (cf.\ \cite[2.3.4.12]{HTT}).
	Then $\rh_n(\cC[S\inv])$ can be identified with the localization in $\Cat_\infty^{\le n}$ of $\rh_n(\cC)$ at the image of $S$ via the canonical functor $\cC \to \rh_n(\cC)$.
	In particular, the localization functor $\cC \to \cC[S\inv]$ is essentially surjective.
\end{lem}

\begin{proof}
	It follows from \cite[2.3.4.12(4)]{HTT} that there for every $n$-category $\cD$ there is a canonical equivalence
	\[ \Fun(\rh_n(\cC[S\inv]), \cD) \simeq \Fun(\cC[S\inv], \cD) \simeq \Fun_S(\cC, \cD) \simeq \Fun_S(\rh_n(\cC), \cD) . \]
	This proves the first statement.
	To see that $\cC \to \cC[S\inv]$ is essentially surjective, it is enough to apply the lemma with $n = 1$: this shows that $\rh(\cC[S\inv])$ is the $1$-categorical localization of $\rh(\cC)$ at the collection of maps $S$. The functor $\rh(\cC) \to \rh(\cC)[S\inv]$ is essentially surjective because of the explicit description of the $1$-categorical localization given in \cite{Gabriel_Zisman_Calculus_of_fractions}, and this is enough to complete the proof.
\end{proof}

\begin{lem} \label{lem:quotient_t_structure_left_complete}
	Let $\cC$ be a stable $\infty$-category equipped with a left complete $t$-structure $(\cC^{\le 0}, \cC^{\ge 0})$.
	Let $\cK \subset \cC$ be a full stable subcategory and let $\cD \coloneqq \cC / \cK$ be the Verdier quotient of $\cC$ by $\cK$, computed in $\Catst$.
	Suppose that:
	\begin{enumerate}
		\item $\cK$ is the kernel of the quotient map $L \colon \cC \to \cD$;
		\item $\cD$ admits a $t$-structure $(\cD^{\le 0}, \cD^{\ge 0})$ such that $L$ is $t$-exact.
	\end{enumerate}
	Then the following conditions are equivalent:
	\begin{enumerate}
		\item the $t$-structure on $\cD$ is left complete;
		\item the full subcategory $\cK$ is left complete in the sense that $X \in \cC$ belongs to $\cK$ if and only if $\tau_{\ge n}(X) \in \cK$ for every $n \in \mathbb Z$.
	\end{enumerate}
\end{lem}

\begin{proof}
	Suppose first that the $t$-structure on $\cD$ is left complete.
	Let $Y \in \cC$ and suppose that $\tau_{\ge n}(Y) \in \cK$ for all $n$.
	Since $L$ is $t$-exact, we deduce that $\tau_{\ge n}(L(Y)) = L(\tau_{\ge n}(Y)) \simeq 0$, and therefore that $L(Y) \simeq 0$.
	Since $\cK$ is the kernel of $L$, we can conclude that $Y \in \cK$, so that $\cK$ is left complete.
	
	Suppose now that $\cK$ is left complete.
	Let
	\[ X \in \bigcap_{n \in \mathbb Z} \cD^{\le n} . \]
	Combining \cite[Theorem 1.3]{Drew_Verdier_quotients_2015} and \cref{lem:localization_essential_surjective}, we see that $L$ is essentially surjective.
	Thus, we can choose $Y \in \cC$ such that $L(Y) \simeq X$.
	Since $L$ is $t$-exact, we have
	\[ L(\tau_{\ge n}(Y)) = \tau_{\ge n}(L(Y)) = 0 . \]
	Since $\cK$ is the kernel of $L$, this implies that $\tau_{\ge n}(Y) \in \cK$.
	Thus, if $\cK$ is left complete, we conclude that $Y \in \cK$ and therefore that $X \simeq L(Y) \simeq 0$, i.e.\ the $t$-structure on $\cD$ is left complete.
\end{proof}

\begin{lem} \label{lem:Verdier_fiber_sequence}
	Let $\cC$ be a stable $\infty$-category and let $i \colon \cK \hookrightarrow \cC$ be the inclusion of a full stable subcategory.
	Let $L \colon \cC \to \cC / \cK$ be the canonical quotient map to the Verdier quotient.
	Suppose that:
	\begin{enumerate}
		\item $i \colon \cK \hookrightarrow \cC$ has a right adjoint $R \colon \cC \to \cK$;
		\item $L \colon \cC \to \cC / \cK$ has a fully faithful right adjoint $j \colon \cC / \cK \to \cC$;
		\item $\cK$ is the kernel of $L \colon \cC \to \cC / \cK$.
	\end{enumerate}
	Then for any $X \in \cC$ there is a canonical fiber sequence
	\[ \begin{tikzcd}
		iR(X) \arrow{r}{\varepsilon_X} & X \arrow{r}{\eta_X} & jL(X) ,
	\end{tikzcd} \]
	where $\varepsilon_X$ and $\eta_X$ are respectively the counit of $(i,R)$ and the unit of $(L,j)$.
\end{lem}

\begin{proof}
	Let $Y \coloneqq \fib( X \xrightarrow{\eta_X} jL(X) )$.
	Then $L(Y) \simeq 0$, and therefore $Y \in \cK$ by assumption (3).
	Furthermore, if $Y' \in \cK$ and $Y' \to X$ is any given map, we see that
	\[ \Map_\cC(Y', jL(X)) \simeq \Map_{\cC / \cK}(L(Y'), L(X)) \simeq 0 . \]
	In particular, the map factors up to a contractible space of choices through $Y$.
	This shows that $Y$ is a localization of $X$ with respect to $\cK$.
	In other words, $Y \simeq iR(X)$, and the map $Y \to X$ is equivalent to $\varepsilon_X \colon iR(X) \to X$.
\end{proof}

\begin{rem}
	The hypotheses of the above lemma are in particular satisfied if $\cC$ is a presentable, stable $\infty$-category and $\cK$ is a full presentable, stable $\infty$-category closed under colimits.
	Indeed, the first two conditions are well known to be satisfied in this setting, and for the third one we refer to \cite[Proposition 1.18]{Robalo_Homotopy_noncommutative_spaces}.
\end{rem}

We now turn to the main result of this subsection.
Before stating it, we need to introduce a couple of notations.
Let $\cC$ be a presentable, compactly generated, stable $\infty$-category equipped with an accessible $t$-structure $(\cC^{\le 0}, \cC^{\ge 0})$.
Let $i_\cK \colon \cK \hookrightarrow \cC$ be the inclusion of a full stable subcategory which is closed under colimits, and let $R \colon \cC \to \cK$ be a right adjoint for $i_{\cK}$.
Suppose furthermore that $i_\cK$ preserves compact objects.
Let $\cD \coloneqq \cC / \cK$ be the Verdier quotient of $\cC$ by $\cK$ (computed in $\LPromega_{\mathrm{st}}$).
Combining \cite[Theorem 5.7]{Blumberg_Gepner_Universal_2013} and \cite[5.5.4.20]{HTT}, we can identify $\cD$ with the full subcategory of $\cC$ spanned by objects that are right orthogonal to $\cK$.
Let $i_\cD \colon \cD \to \cC$ be the fully faithful right adjoint to the quotient map $L$.
Using \cite[1.4.4.11]{Lurie_Higher_algebra} we see that there exists a unique $t$-structure on $\cD$ whose connective part is the smallest full subcategory of $\cD$ which is closed under colimits and contains $L(\cC^{\le 0})$. We denote this $t$-structure by $(\cD^{\le 0}, \cD^{\ge 0})$.
Observe that the quotient map $L \colon \cC \to \cD$ is left $t$-exact by construction.
A standard adjunction argument shows that $i_\cD$ is right $t$-exact.
With these notations, we can prove:

\begin{prop} \label{prop:quotient_t_structure}
	Let $\cC$ be a presentable, compactly generated, stable $\infty$-category equipped with an accessible $t$-structure $(\cC^{\le 0}, \cC^{\ge 0})$.
	Let $i_{\cK} \colon \cK \hookrightarrow \cC$ be a full stable subcategory which is closed under colimits and preserves compact objects and let $R \colon \cC \to \cK$ be a right adjoint for $i_{\cK}$.
	Finally, let $\cD \coloneqq \cC / \cK$ be the Verdier quotient (computed in $\LPromega_{\mathrm{st}}$), equipped with the $t$-structure $(\cD^{\le 0}, \cD^{\ge 0})$ constructed above.
	The following conditions are equivalent:
	\begin{enumerate}
		\item the quotient map $L \colon \cC \to \cD$ is $t$-exact;
		\item there is a $t$-structure on $\cK$ such that the inclusion $i_\cK \colon \cK \to \cC$ is $t$-exact and, furthermore, for every $X \in \cC^{\ge 0}$, the canonical map $i_{\cK}(R(X)) \to X$ induces a monomorphism on $\rH^0$.
	\end{enumerate}
\end{prop}

\begin{rem}
	A similar result has been obtained independently by B.\ Antieau, D.\ Gepner and J.\ Heller in \cite{Antieau_Gepner_Heller_Negative_heart}.
\end{rem}

\begin{proof}
	We start by observing that, for every $X \in \cC$, there is a canonical fiber sequence
	\begin{equation} \label{eq:localization_fiber_sequence}
	i_\cK(R(X)) \to X \to i_\cD(L(X)) .
	\end{equation}
	Remark that in this situation $\cK$ is forcibly the kernel of the quotient map $L \colon \cC \to \cD$.
	
	Suppose first that $L$ is $t$-exact.
	If $L(X) \simeq 0$, then
	\[ L(\tau_{\le n}(X)) \simeq \tau_{\le n}(L(X)) = 0, \qquad L(\tau_{\ge n}(X)) \simeq \tau_{\ge n}(L(X)) \simeq 0 . \]
	It follows that both $\tau_{\le n}(X)$ and $\tau_{\ge n}(X)$ belong to $\cK$, and therefore that the $t$-structure on $\cC$ restricts to a $t$-structure on $\cK$.
	Let now $X \in \cC^{\ge 0}$.
	Since $L$ is $t$-exact, a standard adjunction argument shows that $i_\cD$ is right $t$-exact.
	In particular, $i_\cD(L(X)) \in \cC^{\ge 0}$.
	Passing to the long exact sequence in the fiber sequence \eqref{eq:localization_fiber_sequence} we obtain
	\[ 0 = \rH^{-1}(i_\cD(L(X))) \to \rH^0(i_\cK(R(X))) \to \rH^0(X) \to \rH^0(i_\cD(L(X))) , \]
	from which (2) follows.
	
	Suppose now that (2) holds.
	By construction, $L$ is left $t$-exact.
	It will therefore be sufficient to show that $L$ is right $t$-exact.
	Let $X \in \cC^{\ge 0}$.
	We want to show that $L(X) \in \cD^{\ge 0}$.
	For this, it is enough to show that for every $X' \in \cC^{\le -1}$ one has
	\[ \Map_{\cD}(L(X'), L(X)) = \Map_{\cC}(X', i_\cD(L(X))) = 0 . \]
	In other words, it is enough to prove that $i_\cD(L(X)) \in \cC^{\ge 0}$.
	Since $i_\cK \colon \cK \hookrightarrow \cC$ is $t$-exact, we see that $R$ is right $t$-exact.
	In particular, $R(X) \in \cK^{\ge 0}$, and therefore $i_\cK(R(X)) \in \cC^{\ge 0}$.
	Using the fiber sequence \eqref{eq:localization_fiber_sequence}, we get that $i_\cD(L(X))$ belongs to $\cC^{\ge -1}$. Moreover, with the long exact sequence
	\[ 0=\rH^{-1}(X) \to \rH^{-1}(i_\cD(L(X))) \to \rH^{0}(i_\cK(R(X))) \overset{f}{\to} \rH^{0}(X) . \]
	and the assumption that $f$ is a monomorphism, we get $\rH^{-1}(i_\cD(L(X))) = 0$.
	This completes the proof.
\end{proof}

\begin{cor} \label{cor:quotient_t_structure}
	Let $\cC$ be a stable $\infty$-category equipped with a $t$-structure $(\cC^{\le 0}, \cC^{\ge 0})$.
	Let $i \colon \cK \subseteq \cC$ be a full stable subcategory of $\cC$.
	Suppose that:
	\begin{enumerate}
		\item the $t$-structure on $\cC$ restricts to a $t$-structure on $\cK$;
		\item the induced functor $i \colon \cK^\heartsuit \to \cC^\heartsuit$ has a right adjoint $R \colon \cC^\heartsuit \to \cK^\heartsuit$ and the counit transformation $i(R(X)) \to X$ is a monomorphism for every $X \in \cC^\heartsuit$.
	\end{enumerate}
	Let $\cD \coloneqq \cC / \cK$ be the Verdier quotient (computed in $\Catst$), and let $L \colon \cC \to \cD$ be the quotient map.
	Then there exists a unique $t$-structure on $\cD$ such that $L$ is $t$-exact.
	Furthermore, if $\cK$ is the kernel of $L$ and the $t$-structure on $\cC$ is left complete, the same goes for the one induced on $\cD$.
\end{cor}

\begin{proof}
	The universal property of the $\Ind$-construction shows that $\Ind(\cD)$ can be canonically identified with the Verdier quotient $\Ind(\cC) / \Ind(\cK)$ (computed in $\LPromega_{\mathrm{st}}$).
	Using \cref{lem:t_structure_Ind} we see that $(\Ind(\cK^{\le 0}), \Ind(\cK^{\ge 0}))$ defines a $t$-structure on $\Ind(\cK)$.
	Furthermore, the inclusion $\Ind(\cK) \hookrightarrow \Ind(\cC)$ is $t$-exact.
	Therefore, the first requirement in condition (2) of \cref{prop:quotient_t_structure} is satisfied.
	Let us prove that the second requirement is satisfied as well.
	
	Let us observe that $\Ind(i) \colon \Ind(\cK) \hookrightarrow \Ind(\cC)$ is fully faithful, exact and commutes with filtered colimits by construction.
	It follows that it admits a right adjoint $\oR \colon \Ind(\cC) \to \Ind(\cK)$.
	Since $\Ind(i)$ is $t$-exact, adjointness shows that $\oR$ is right $t$-exact.
	In particular, the functor $\tR \coloneqq \tau_{\le 0} \circ \oR \colon \Ind(\cC)^\heartsuit \to \Ind(\cK)^\heartsuit$ is right adjoint to the induced functor
	\[ \Ind(j) \colon \Ind(\cK)^\heartsuit \hookrightarrow \Ind(\cC)^\heartsuit . \]
	Let $j_\cK \colon \cK \to \Ind(\cK)$ and $j_\cC \colon \cC \to \Ind(\cC)$ be the Yoneda embeddings.
	Since they are $t$-exact by \cref{lem:t_structure_Ind}, we obtain a commutative square
	\[ \begin{tikzcd}
	\Ind(\cK)^\heartsuit \arrow{r}{\Ind(i)} & \Ind(\cC)^\heartsuit \\
	\cK^\heartsuit \arrow{u}{j_\cK} \arrow{r}{i} & \cC^\heartsuit \arrow{u}{j_\cC} .
	\end{tikzcd} \]
	Recall that, by assumption, the functor $i \colon \cK^\heartsuit \to \cC^\heartsuit$ has a right adjoint $R$.
	Using fully faithfulness of $j_\cK$ and $j_\cC$, we obtain that, for every $X \in \cK^\heartsuit$ and every $Y \in \cC^\heartsuit$, one has:
	\begin{align*}
	\Hom_{\Ind(\cK)^\heartsuit}(j_\cK(X), j_\cK(R(Y)) & = \Hom_{\cK^\heartsuit}(X, R(Y)) \\
	& = \Hom_{\cC^\heartsuit}(i(X), Y) \\
	& = \Hom_{\Ind(\cC)^\heartsuit}(j_\cC(i(X)), j_\cC(Y) ) \\
	& = \Hom_{\Ind(\cC)^\heartsuit}(\Ind(i)(j_\cK(X)), j_\cC(Y)) \\
	& = \Hom_{\Ind(\cK)^\heartsuit}(j_\cK(X), \tR(j_\cC(Y)) .
	\end{align*}
	This implies that there is a natural isomorphism
	\[ j_\cK \circ R \simeq \tR \circ j_\cC . \]
	Since $j_\cK$ commutes with limits, we conclude that $j_\cK$ preserves monomorphisms.
	In particular, for every $X \in \cC^\heartsuit$, the canonical map
	\[ \rH^0(\Ind(i)(R(X))) \to \rH^0(X) \]
	is a monomorphism.
	Since if $X \in \cC^{\ge 0}$ one has
	\[ \rH^0(\Ind(i)(R(X))) = \rH^0(\Ind(i)(R(\tau_{\le 0} X))) , \]
	we see that the same conclusion holds for every $X \in \cC^{\ge 0}$.
	Finally, since filtered colimits commute with finite limits and therefore with monomorphism, we conclude that the second requirement of condition (2) of \cref{prop:quotient_t_structure} is satisfied.
	
	Thus, \cref{prop:quotient_t_structure} implies that $\Ind(\cD)$ admits a unique $t$-structure such that the quotient map $\Ind(\cC) \to \Ind(\cD)$ is $t$-exact.
	Combining \cref{lem:localization_essential_surjective} and \cref{lem:t_structure_Ind_restriction} we see that this $t$-structure induces a $t$-structure on $\cD$ in such a way that the localization map $L \colon \cC \to \cD$ becomes $t$-exact.
	Finally, we can use \cref{lem:quotient_t_structure_left_complete} to deduce that if the $t$-structure on $\cC$ is left complete, then the same goes for the one on $\cD$.
\end{proof}

\subsection{Localization}

We now turn to the main result of the section, which is the localization property for almost perfect modules.
We start with the following useful result:

\begin{prop} \label{prop:internal_characterisation_kernel_localisation}
	Let $X$ be a quasi-compact quasi-separated scheme which is locally Noetherian over $k$.
	Let $i \colon U \to X$ be a quasi-compact open immersion and let $j \colon Z \to X$ be the inclusion of the complementary closed (reduced) subscheme.
	Then the kernel $\cK$ of the pullback map $i^* \colon \Coh^-(X) \to \Coh^-(U)$ coincides with the smallest full stable subcategory of $\Coh^-(X)$ which is left complete and contains the essential image of $j_* \colon \Coh^-(Z) \to \Coh^-(X)$.
\end{prop}

\begin{proof}
	Since $i \colon U \to X$ is an open immersion, the pullback $i^* \colon \Coh^-(X) \to \Coh^-(U)$ is $t$-exact.
	It follows that $i^*$ commutes with truncations, and therefore that $\cK \coloneqq \ker(i^*)$ is closed under truncations.
	Furthermore, if $X \in \Coh^-(X)$ is such that $\tau_{\ge n}(X) \in \cK$ for every $n \in \mathbb Z$, we obtain:
	\[ \tau_{\ge n}(i^*(X)) \simeq i^*(\tau_{\ge n}(X)) \simeq 0 . \]
	Since the $t$-structure on $\Coh^-(U)$ is left complete, we conclude that $i^*(X) \simeq 0$, i.e.\ $X \in \cK$.
	Thus $\cK$ is left complete.
	
	Let now $\cC$ be the smallest full stable subcategory of $\Coh^-(X)$ which is closed uner truncations, which is left complete and contains the essential image of $j_* \colon \Coh^-(Z) \to \Coh^-(X)$.
	By definition, we have $\cC \subseteq \cK$.
	To prove that $\cK \subseteq \cC$, let us start by proving that $\cK^\heartsuit \subseteq \cC$.
	Let $\cF \in \Coh^\heartsuit(X)$ and suppose that $i^*(\cF) \simeq 0$.
	Let $\cJ$ be the ideal sheaf defining $Z$.
	If $V \subset X$ is an affine open and $X$ is locally noetherian, we see that $\cJ|_V^n \subseteq \mathrm{Ann}(\cF |_V)$.
	Since $X$ is quasi-compact, we can find an integer $m \gg 0$ such that $\cJ^m \subset \mathrm{Ann}(\cF)$.
	Thus $\cF$ is supported on $Z^{(m)} \coloneqq \Spec_X(\cO_X / \cJ^m)$ and hence it belongs to the essential image of
	\[ j_*^{(m)} \colon \Coh^-(Z^{(m)}) \to \Coh^-(X) . \]
	Now observe that for $1 \le l \le m$ we have fiber sequences
	\[ \cJ^{l-1} \cF / \cJ^l \cF \to \cF / \cJ^l \cF \to \cF / \cJ^{l - 1} \cF \]
	in $\Coh^-(X)$.
	Then $\cJ^{l-1} \cF / \cJ^l \cF$ is supported on $Z$, and $\cF / \cJ^{l-1} \cF$ is supported on $Z^{(l-1)}$.
	Thus, starting with $l = 2$ and applying induction, we see that $\cF / \cJ^l \cF$ belongs to the stable image of $j_* \colon \Coh^-(Z) \to \Coh^-(X)$ for every $2 \le l \le m$.
	Taking $l = m$, we conclude that the $\cF$ belongs to the stable image of $j_*$, and hence to $\cC$.
	
	We can now prove that $\cK \subset \cC$. Let $\cF \in \cK$.
	Since the $t$-structure on $\Coh^-(X)$ is right bounded and since $X$ is quasi-compact, we can find an integer $m$ such that for every $m' > m$ one has $\rH^{m'}(\cF) = 0$.
	Moreover, we can assume without loss of generality that $m = 0$.
	Since $\cC$ is left complete by assumption, it is enough to prove that $\tau_{\ge n}(\cF)$ belongs to $\cC$ for every $n \le 0$.
	In other words, we can assume that $\cF$ belongs to $\Cohb(X)$.
	We can therefore proceed by induction on the number $n$ of non-vanishing cohomology sheaves of $\cF$.
	If $n \le 1$, $\cF \in \cK^\heartsuit$ and thus we already proved that $\cF \in \cC$.
	If $n > 1$, we have a fiber sequence
	\[ \tau_{\le -1} \cF \to \cF \to \tau_{\ge 0} \cF . \]
	Now $\tau_{\ge 0} \cF \in \cK^\heartsuit$ and $\tau_{\le -1} \cF \in \cK$ and it has at most $n-1$ non-vanishing cohomology sheaves.
	Therefore, the induction hypothesis implies that $\tau_{\ge 0} \cF$ and $\tau_{\le -1} \cF \in \cC$.
	Since $\cC$ is stable, we conclude that $\cF \in \cC$ as well.
	Thus $\cC = \cK$, and the proof is complete.
\end{proof}

We isolate the following by-product of the proof of the previous proposition:

\begin{cor} \label{cor:internal_characterisation_kernel_localisation}
	Let $X$ be a quasi-compact quasi-separated scheme which is locally Noetherian over $k$.
	Let $i \colon U \to X$ be a quasi-compact open immersion and let $j \colon Z \to X$ be the inclusion of the complementary closed (reduced) subscheme.
	Let $\cK$ be the kernel of $i^* \colon \Coh^-(X) \to \Coh^-(U)$.
	Then $\cK^\heartsuit$ consists of those $\cF \in \Coh^\heartsuit(X)$ such that $\cF$ is annihilated by some power of the ideal of definition of $Z$.
\end{cor}

\begin{thm}  \label{thm:localization_for_coh_minus}
	Let $X$ be a quasi-compact quasi-separated scheme which is locally Noetherian over $k$.
	Let $i \colon U \to X$ be a quasi-compact open immersion.
	Let $\cK$ be the kernel of the pullback $i^* \colon \Coh^-(X) \to \Coh^-(U)$.
	Then the square
	\[ \begin{tikzcd}
	\cK \arrow[hook]{r} \arrow{d} & \Coh^-(X) \arrow{d} \\
	0 \arrow{r} & \Coh^-(U)
	\end{tikzcd} \]
	is a cofiber sequence in $\Catst$.
\end{thm}

\begin{proof}
	Since $i^* \colon \Coh^-(X) \to \Coh^-(U)$ is both left and right $t$-exact, we see that $\cK$ is closed under truncations.
	Thus, the $t$-structure on $\Coh^-(X)$ restricts to a $t$-structure on $\cK$.
	Furthermore, \cref{cor:internal_characterisation_kernel_localisation} shows that we can identify $\cK^\heartsuit$ with the inclusion of the full subcategory of $\Coh^\heartsuit(X)$ spanned by those $M$ such that $\cJ^n \subset \mathrm{Ann}(M)$, where $\cJ$ is the ideal sheaf defining $Z \coloneqq X \smallsetminus U$ (equipped with its reduced scheme structure).
	It follows that $\cK^\heartsuit \hookrightarrow \Coh^\heartsuit(X)$ has a right adjoint, given by the $J$-torsion.
	In particular, the counit map of this adjunction is a monomorphism.
	Thus, the hypotheses of \cref{cor:quotient_t_structure} are satisfied.
	Let $\cD$ be the Verdier quotient $\Coh^-(X) / \cK$.
	Then \cref{cor:quotient_t_structure} shows that $\cD$ admits a $t$-structure such that the localization map $\Coh^-(X) \to \cD$ is $t$-exact.
	
	Let us show that, in addition, the $t$-structure on $\cD$ is left complete.
	For this, it will be enough to prove that $\cK$ is the kernel of $L \colon \Coh^-(X) \to \cD$.
	Let $\widetilde{\cD}$ be the idempotent completion of $\cD$.
	Then $i \colon \cD \to \widetilde{\cD}$ is fully faithful, and therefore
	\[ \ker(L) = \ker(i \circ L) . \]
	Furthermore, $\widetilde{\cD}$ can be identified with the Verdier quotient $\Coh^-(X) / \cK$ computed in $\Catstidem$.
	It follows from \cite[Proposition 1.18]{Robalo_Homotopy_noncommutative_spaces} that $\cK$ is the kernel of the map $i \circ L \colon \Coh^-(X) \to \widetilde{\cD}$.
	Thus $\cK = \ker(L)$, and therefore the conclusion of \cref{cor:quotient_t_structure} applies.
	
	We claim that $\cD$ coincides also with the cofiber of $\cK \hookrightarrow \Coh^-(X)$ when computed in the $\infty$-category $\Catstlc$ of stable $\infty$-categories equipped with a left complete $t$-structure and $t$-exact functors between them.
	Indeed, let $\cE \in \Catstlc$.
	Then, since any functor $F \colon \cD \to \cE$ which is naturally equivalent to a $t$-exact functor is $t$-exact on its own, we conclude that there is a homotopy monomorphism
	\[ u \colon \Map_{\Catstlc}(\cD, \cE) \hookrightarrow \Map_{\Catst}(\cD, \cE) \simeq \Map_{\Catst}^\cK(\Coh^-(X), \cE) , \]
	where $\Map_{\Catst}^\cK(\Coh^-(X), \cE)$ denotes the full subsimplicial set of $\Map_{\Catst}(\Coh^-(X), \cE)$ generated by those functors $F \colon \Coh^-(X) \to \cE$ having the property that for every $\cF \in \cK$, the object $F(\cF)$ is a zero object in $\cE$.
	Similarly, there is a homotopy monomorphism
	\[ v \colon \Map_{\Catstlc}^\cK(\Coh^-(X), \cE) \hookrightarrow \Map_{\Catst}^\cK(\Coh^-(X), \cE) , \]
	and $u$ clearly factors through $v$.
	It will therefore be enough to prove that the induced homotopy monomorphism
	\[ w \colon \Map_{\Catstlc}(\cD, \cE) \hookrightarrow \Map_{\Catstlc}^\cK(\Coh^-(X), \cE) \]
	is surjective on $\pi_0$.
	To see this, fix an $\infty$-functor $F \in \Map_{\Catstlc}^\cK(\Coh^-(X), \cE)$.
	Reviewing $F$ as an $\infty$-functor in $\Catst$ and using the universal property of the cofiber, we can produce an extension
	\[ \widetilde{F} \colon \cD \to \cE \]
	in $\Catst$, and all we are left to check is that $\widetilde{F}$ is $t$-exact.
	Since the localization functor $L \colon \Coh^-(X) \to \cD$ is essentially surjective by \cref{lem:localization_essential_surjective}, we can combine \cref{lem:t_structure_Ind} and \cref{lem:t_structure_Ind_restriction} to conclude that
	\[ \cD^{\le 0} \simeq L((\Coh^-(X))^{\le 0}) . \]
	We now claim that one has also
	\[ \cD^{\ge 0} \simeq L((\Coh^-(X))^{\ge 0}) . \]
	Indeed, let $\cG \in \cD^{\ge 0}$.
	Since $L$ is essentially surjective, we can find $\cF \in \Coh^-(X)$ such that
	\[ L(\cF) \simeq G . \]
	Consider the fiber sequence
	\[ \tau_{\le - 1} \cF \to \cF \to \tau_{\ge 0} \cF . \]
	Since $L$ is $t$-exact, we obtain
	\[ L(\tau_{\le -1} \cF) \simeq \tau_{\le - 1} L(\cF) \simeq \tau_{\le -1}(\cG) \simeq 0 . \]
	Thus
	\[ \cG \simeq L(\cF) \simeq L(\tau_{\ge 0} \cF) . \]
	We can now proceed with the proof of the original claim.
	Let $\cG \in \cD^{\le 0}$ (resp.\ $\cG \in \cD^{\ge 0}$) and write it as $\cG \simeq L(\cF)$, with $\cF \in (\Coh^-(X))^{\le 0}$ (resp.\ $\cF \in (\Coh^-(X))^{\ge 0}$).
	Then
	\[ \widetilde{F}(\cG) \simeq \widetilde{F}(L(\cF)) \simeq F(\cF) . \]
	Since $F$ was assumed to be $t$-exact, we conclude that $\widetilde{F}(\cG) \in \cE^{\le 0}$ (resp.\ $\widetilde{F}(\cG) \in \cE^{\ge 0}$).
	Thus, $\widetilde{F}$ is $t$-exact and the proof of the claim is completed.
	
	Recall from \cite[1.2.1.18]{Lurie_Higher_algebra} that there is an equivalence between $\Catstlc$ and the $\infty$-category $\Catstlb$ of stable $\infty$-categories equipped with a left bounded $t$-structure.
	Under this equivalence, $\Coh^-(X)$ and $\Coh^-(U)$ correspond to $\Cohb(X)$ and $\Cohb(U)$, respectively, and $\cK$ corresponds to the kernel
	\[ \cK^\rb \coloneqq \ker( i^* \colon \Cohb(X) \to \Cohb(U) ) . \]
	For this reason we can reduce the proof of the proposition to the analogous statement for $\Cohb(X)$.
	In this case, combining \cite[Lemma 4.1.1]{Gaitsgory_IndCoh} and the fact that the $\Ind$-construction commutes with localizations of $\infty$-categories, we can deduce that $\Cohb(U)$ is the cofiber of $\cK^\rb \hookrightarrow \Cohb(X)$ in $\Catstidem$.
	We claim that this is also a cofiber sequence in $\Catst$.
	Using the construction of Verdier quotients in $\Catst$ given in \cref{thm:general_Verdier_quotients}, we see that it is enough to prove that the smallest full stable subcategory of $\Cohb(U)$ containing the essential image of $i^* \colon \Cohb(X) \to \Cohb(U)$ is $\Cohb(U)$ itself.
	Denote by $\cA$ this full stable subcategory of $\Cohb(U)$ and let $\cG \in \Cohb(U)$.
	Since $U$ is quasi-compact, $\cG$ lives in a finite cohomological range.
	Let $n(\cG)$ be the cohomological amplitude of $\cG$.
	We prove that $\cG \in \cA$ by induction on $n(\cG)$.
	If $n(\cG) \le 1$, then $\cG \in \Cohh(U)$.
	Therefore, it exists a coherent extension $\cF \in \Cohh(X)$ (i.e.\ the sheaf $\cF$ satisfies $i^*(\cF) \simeq \cG$).
	Suppose now that $n(\cG) \ge 2$.
	We can find a fiber sequence
	\[ \tau_{\le k - 1} \cG \to \cG \to \tau_{\ge k} \cG \]
	such that
	\[ n(\tau_{\le k-1} \cG), n(\tau_{\ge k} \cG) < n(\cG) . \]
	We can thus use the induction hypothesis to deduce that $\tau_{\le k - 1} \cG, \tau_{\ge k} \cG \in \cA$.
	Since $\cA$ is a full stable subcategory of $\Cohb(U)$ (which is closed under equivalences), we conclude that
	\[ \cG \simeq \fib( \tau_{\ge k} \cG \to \tau_{\le k - 1} \cG[-1] ) \in \cA . \]
	This completes the proof of the claim.
	
	It follows now from \cite{Drew_Verdier_quotients_2015} that $\Cohb(U)$ is the localization of $\Cohb(X)$ at the collection of maps
	\[ S_\cK \coloneqq \{ \varphi \colon \cF \to \cG \in \Cohb(X) \mid \cofib(\varphi) \in \cK \} . \]
	Combining this with \cref{lem:localization_essential_surjective}, we conclude that the map $\Cohb(X) \to \Cohb(U)$ is essentially surjective.
	At this point, the same argument used for $\Coh^-(X)$ applies verbatim to show that $\Cohb(U)$ is the cofiber of $\cK^\rb \hookrightarrow \Cohb(X)$ in $\Catstlb$.
	The proof is thus complete.
\end{proof}

\section{Sheaves on the punctured formal neighbourhood}\label{section:sheaves_on_PFN}

In this section we introduce several models for the punctured formal neighbourhood of a closed immersion $Z \hookrightarrow X$.
When we say ``formal neighbourhood'' we simply mean the formal completion $\hZ$ of $X$ along $Z$, which we think geometrically as a tubular neighbourhood of $Z$ inside $X$.
As it has been remarked in the introduction, the ``punctured formal neighbourhood'' does not exist in the realm of formal algebraic geometry, because $\hZ \smallsetminus Z$ is the empty formal scheme.
In this section we introduce three different models for this geometric object:
\begin{enumerate}
	\item at first we construct a functor $\hPsi^\circ_Z \colon X\Et \to \St_k$, where $X\Et$ denotes the big \emph{affine} \'etale site of $X$. The intuition behind $\hPsi^\circ_Z$ is the following: suppose that $X$ is affine, say $X = \Spec(A)$, and let $\widehat{A}$ be the formal completion of $A$ at the ideal defining $Z$; to $\widehat{A}$ we can then associate two different objects: on one side, we have its formal spectrum $\widehat{Z} = \Spf(\widehat{A})$, and on the other side we can consider the affinization of $\widehat{Z}$, i.e.\ $\Spec(\widehat{A})$. Now, $Z$ defines a closed subscheme of $\Spec(\widehat{A})$, and its open complement plays to some extent the role of the punctured formal neighbourhood of $Z$ inside $X$. The operation of taking this open complement is functorial in $X$, and we denote the resulting functor by $\hPsi^\circ_Z$.
	\item Even though $\hPsi^\circ_Z$ is easy to understand and to manipulate, it lacks the most fundamental properties that would allow us to think about it as a sort of geometric object. Surely, the most important issue is that $\hPsi^\circ_Z$ is not continuous with respect to the \'etale topology: it does not take \v{C}ech nerves for the \'etale topology on $\Aff_k$ to colimit diagrams in $\St_k$. For this reason we are lead to introduce variations $\hPhi^\circ_Z$ and $\hrP^\circ_Z$, defined respectively as the composition of $\hPsi^\circ_Z$ with the stacks $\bfCoh^-$ and $\bfPerf$. We think of these functors as ``non-commutative'' incarnations of the punctured formal neighbourhood.
\end{enumerate}
The main result in this section (\cref{thm:etale_descent}) is that the functor $\hPhi^\circ_Z$ satisfies \'etale descent.
From here it is easy to deduce that also $\hrP^\circ_Z$ satisfies \'etale descent (cf.\ \cref{prop:restricted_etale_descent_Perf}).
In turn, this allows us to define the stacks $\bfCoh_{\hZ \smallsetminus Z}$ and $\bfPerf_{\hZ \smallsetminus Z}$, that will play a central role in \cref{sec:formal_glueing}.
It is also important to note that we will actually consider symmetric monoidal variations of these functors.

Before passing to the actual constructions and proofs, we fix the following notations.
We denote by $X$ a fixed Artin stack locally of finite type over $\Spec(k)$, where $k$ is a field.
We fix a closed substack $j \colon Z \to X$ and $\cJ$ will always denote the ideal sheaf of $Z$ inside $X$.
Given a map $U \to X$ of Artin stacks, we denote by $Z_U$ the base change $U \times_X Z$.
We denote the ideal sheaf of $Z_U$ inside $U$ by $\cJ_U$.
The $n$-th formal neighbourhood of $Z_U$ inside $U$ is by definition the scheme
\[ Z_U^{(n)} = \Spec_X( \cO_U / \cJ_U^n ) . \]
We let $\widehat{Z_U}$ be the formal completion of $U$ along $Z_U$, seen as formal stack.
We have
\[ \widehat{Z_U} \simeq \colim Z_U^{(n)} , \]
the colimit being computed either in the category $\PSh(\Aff_k)$ or in the category of formal stacks over $k$.

\subsection{Recollection on formal schemes}

Let $\fX$ be a formal stack over $\Spec(k)$.
We let $\fX_s$ denote the closed fiber of $\fX$ (cf.\ \cite[\S 1]{Berkovich_Vanishing_1994}).
We recall the following well known result:

\begin{lem}[{cf.\ \cite[Lemma 2.1]{Berkovich_Vanishing_1994}}] \label{lem:topological_invariance_etale}
	Let $\fX$ be a formal scheme over $k$.
	The closed fiber functor determines an equivalence of sites
	\[ \fX\et \to (\fX_s)\et . \]
\end{lem}

\begin{lem} \label{lem:formal_morphism_sites}
	The assignment
	\[ X\et \to \St_k \]
	given by
	\[ U \mapsto \widehat{Z_U} \]
	factors through $\hZ{\et}$, the small \emph{formal} affine \'etale site of $\hZ$.
	Furthermore, the induced functor
	\[ X\et \to \hZ\et \]
	is a morphism of sites.
\end{lem}

\begin{proof}
	Using \cref{lem:topological_invariance_etale}, we are reduced to show that the functor $X\et \to \St_k$ given by
	\[ U \mapsto Z \times_X U \]
	factors as a morphism of sites
	\[ X\et \to Z\et . \]
	This is clear, since it is the base change functor along the map $Z \to X$.
\end{proof}

\begin{defin}
	Let $A \in \CAlg_k$ be a commutative algebra.
	We define $\Coh^-(A)$ as the full stable $\infty$-subcategory of $\QCoh(A)$ spanned by those $\cF \in \QCoh(A)$ satisfying the following two conditions:
	\begin{enumerate}
		\item for every integer $i \in \mathbb Z$, the cohomology group $\rH^i(\cF)$ is coherent over $A$;
		\item $\rH^i(\cF) = 0$ for $i \gg 0$.
	\end{enumerate}
\end{defin}

The assignment $\Spec(A) \mapsto \Coh^-(A)$ can be promoted to an $\infty$-functor
\[ \bfCoh^{\otimes,-} \colon \Aff_k\op \to \Catstmon , \]
using the $(-)^*$ functoriality. It comes with a pointwise fully faithful natural transformation $\bfCoh^{\otimes,-} \to \bfQCoh^\otimes$.

Since $\bfQCoh^{\otimes}$ is an hypercomplete sheaf for the \'etale topology, and since being coherent is a local property, the functor $\bfCoh^{\otimes,-}$ is an hypercomplete sheaf for the \'etale topology.
In particular, both $\bfQCoh^\otimes$ and $\bfCoh^{\otimes,-}$ extend uniquely into limit-preserving functors
\[
\bfQCoh^{\otimes},\,\bfCoh^{\otimes,-} \colon \St_k\op \to \Catstmon
\]

\begin{defin}
The category of formal schemes embeds, through the functor of points, into the category $\St_k$. In particular, for any formal scheme $\fX$, we get two stable symmetric monoidal $\infty$-categories $\Coh^-(\fX)$ and $\QCoh(\fX)$.
\end{defin}

\begin{rem}
	If $\fX$ is a formal scheme, we can find a filtered diagram $X \colon I \to \Sch_k$ with the property that for every map $\alpha \to \beta$ in $I$, the transition map $X_{\alpha} \to X_{\beta}$ is an infinitesimal thickening and such that
	\[ \fX \simeq \colim_{\alpha \in I} X_\alpha , \]
	the colimit being computed in $\St_k$.
	It follows that there is a canonical categorical equivalence
	\[ \Coh^-(\fX) \simeq \varprojlim_{\alpha \in I} \Coh^-(X_\alpha) , \]
	the inverse limit being computed in $\Catstmon$ or, equivalently, in $\Cat_\infty$ (cf.\ \cite[1.1.4.4, 4.7.4.5]{Lurie_Higher_algebra}).
\end{rem}

Let $\fX$ be any formal scheme.
\cref{lem:topological_invariance_etale} implies that the formal \'etale topology on $\fX\et$ is subcanonical.
We conclude that the composite $\infty$-functor
\[ \bfCoh^{\otimes,-} \colon \fX\et \to \St_k \to \Catstmon \]
is a sheaf for the formal \'etale topology.
Furthermore, we have:

\begin{lem} \label{lem:basic_formal_GAGA}
	If $\fX = \Spf(A)$, where $A$ is a Noetherian adic ring, then there is a canonical equivalence of symmetric monoidal stable $\infty$-categories:
	\[ \Coh^-(\fX) \simeq \Coh^-(A) . \]
\end{lem}

\begin{proof}
	Choose an ideal of definition $I$ for $A$.
	Then we have an equivalence of symmetric monoidal stable $\infty$-categories
	\[ \Coh^-(\fX) \simeq \varprojlim_n \Coh^-(A / I^n) . \]
	The natural maps $A \to A / I^n$ induce symmetric monoidal functors
	\[ \Coh^-(A) \to \Coh^-(A / I^n) \]
	that induce by the universal property of inverse limits a symmetric monoidal functor
	\[ \Coh^-(A) \to \Coh^-(\fX) . \]
	Since the forgetful functor $\Catstmon \to \Catst$ is conservative, it is enough to prove that the above functor is an equivalence of stable $\infty$-categories.
	This follows directly from \cite[Theorem 5.3.2]{DAG-XII}.
\end{proof}

\subsection{Punctured tubular neighbourhoods} \label{subsec:punctured_formal_neighbourhood}

We fix in this section an Artin stack $X$, which we assume to be locally of finite type over $\Spec(k)$.
Furthermore, we fix two closed substacks $Z, Z' \hookrightarrow X$ and we suppose that $Z \to X$ factors through $Z' \to X$.

Let $X\Et$ be the big affine \'etale site of $X$.
Given $U \in X\Et$, denote by $Z_U$ (resp.\ $Z'_U$) the base changes $U \times_X Z$ (resp.\ $U \times_X Z'$).
These are closed subscheme of $U$ and $Z_U$ is also a closed subscheme of $Z'_U$.
We denote by $\widehat{Z_U}$ the formal completion of $U$ along $Z_U$.

The assignment $U \mapsto \widehat{Z_U}$ defines a functor:
\[ \hPsi_{Z/X} \colon X\Et \to \St_k . \]
Define a new functor
\[ \hPsi_{Z/X}^{\mathrm{aff}} \colon X\Et \to \St_k \]
as the composition
\[ \begin{tikzcd}
X\Et \arrow{r}{\hPsi_{Z/X}} & \St_k \arrow{r}{\Spec \circ \Gamma} & \St_k .
\end{tikzcd} \]
If $U = \Spec(A_U)$, $J_U$ denotes the ideal defining $Z_U$ inside $A_U$ and $\widehat{A_U}$ denotes the formal completion of $A_U$ at $J_U$, then
\[ \hPsi_{Z/X}^{\mathrm{aff}}(U) \coloneqq \Spec(\widehat{A_U}) . \]
To keep the notations consistent, let us denote by $\fW_X$ the functor
\[ X\Et \to \St_k \]
defined by sending $U$ to $U = U \times_X X$.
Observe that there is a natural transformation $\hPsi_{Z/X}^{\mathrm{aff}} \to \fW_X$.
Similarly, the functor
\[ \fW_{Z'/X} \coloneqq - \times_X Z' \colon X\Et \to \St_k \]
is equipped with a natural transformation $\fW_{Z' / X} \to \fW_X$.
We denote by $\hPsi_{Z,Z'/X}$ the pullback
\[ \hPsi^{\mathrm{aff}}_{Z,Z' / X} \coloneqq \fW_{Z'/X} \times_{\fW_X} \hPsi_{Z/X}^{\mathrm{aff}} , \]
and we let
\[ j \colon \hPsi^{\mathrm{aff}}_{Z,Z'/X} \to \hPsi_{Z/X}^{\mathrm{aff}} \]
be the induced natural transformation.
Note that if $J'_U$ denotes the ideal of $Z'_U$ inside $U = \Spec(A_U)$, then
\[ \hPsi^{\mathrm{aff}}_{Z,Z'/X}(U) = \Spec(\widehat{A_U} / J'_U \widehat{A_U}) . \]

\begin{lem}
	For every $U \in X\Et$, $j_U \colon \hPsi^{\mathrm{aff}}_{Z,Z'/X}(U) \to \hPsi_{Z/X}^{\mathrm{aff}}(U)$ is a closed immersion of affine schemes.
	Furthermore, $j$ is a cartesian transformation, in the sense that for every morphism $V \to U$ in $X\Et$, the induced square
	\[ \begin{tikzcd}
	\hPsi^{\mathrm{aff}}_{Z,Z'/X}(V) \arrow{r} \arrow{d}{j_V} & \hPsi^{\mathrm{aff}}_{Z, Z'/X}(U) \arrow{d}{j_U} \\
	\hPsi_{Z/X}^{\mathrm{aff}}(V) \arrow{r} & \hPsi_{Z/X}^{\mathrm{aff}}(U)
	\end{tikzcd} \]
	is a pullback square.
\end{lem}

\begin{proof}
	Set $U = \Spec(A_U)$ and $V = \Spec(A_V)$.
	Let $J_U$ and $J_V$ be the defining ideals of $Z_U$ and $Z_V$ in $U$ and $V$, respectively.
	Similarly, let $J'_U$ and $J'_V$ be the defining ideals of $Z'_U$ and $Z'_V$ in $U$ and $V$, respectively.
	Finally, let $\widehat{A_U}$ and $\widehat{A_V}$ be the formal completions of $A_U$ and $A_V$ at $J_U$ and $J_V$ respectively.
	Then $\fW_{Z/X}(U) = Z_U = \Spec(A_U / J_U)$ and $\hPsi^{\mathrm{aff}}_{Z/X}(U) = \Spec(\widehat{A_U})$.
	By construction, $j_U$ corresponds to the quotient map
	\[ \widehat{A_U} \to \widehat{A_U} / J'_U \widehat{A_U} , \]
	and therefore it is a closed immersion.
	Observe now that the commutative square
	\[ \begin{tikzcd}
	Z'_V \arrow{r} \arrow{d} & Z'_U \arrow{d} \\
	V \arrow{r} & U
	\end{tikzcd} \]
	is a pullback square.
	It follows that $J'_V = J'_U A_V$, and thus that $J'_V \widehat{A_V} = J'_U \widehat{A_V}$.
	In particular, the square
	\[ \begin{tikzcd}
	\Spec(\widehat{A_V} / J'_V \widehat{A_V}) \arrow{r} \arrow{d}{j_V} & \Spec(\widehat{A_U} / J'_U \widehat{A_U}) \arrow{d}{j_U} \\
	\Spec(\widehat{A_V}) \arrow{r} & \Spec(\widehat{A_U})
	\end{tikzcd} \]
	is a pullback as well.
\end{proof}

Since the natural transformation $j \colon \hPsi^{\mathrm{aff}}_{Z,Z'/X} \to \hPsi^{\mathrm{aff}}_{Z/X}$ is cartesian and it is composed of closed immersions, we can define a new functor
\[ \hPsi^\circ_{Z,Z'/X} \colon X\Et \to \St_k \]
by the rule
\[ \hPsi^\circ_{Z,Z'/X}(U) \coloneqq \hPsi^{\mathrm{aff}}_{Z/X}(U) \smallsetminus \hPsi^{\mathrm{aff}}_{Z,Z'/X}(U) . \]
Finally, we introduce $\infty$-functors
\[ \hPhi_{Z/X}^\otimes, \hPhi^{\circ, \otimes}_{Z,Z'/X} \colon X\Et\op \to \Catstmon \]
defined as the compositions
\[ \hPhi_{Z/X}^\otimes \coloneqq \bfCoh^{\otimes,-} \circ \hPsi_{Z/X}, \qquad \hPhi^{\circ, \otimes}_{Z,Z'/X} \coloneqq \bfCoh^{\otimes,-} \circ \hPsi^\circ_{Z,Z'/X} . \]
Note that \cref{lem:basic_formal_GAGA} shows that the natural map $\hPsi_{Z/X} \to \hPsi^{\mathrm{aff}}_{Z/X}$ induces an equivalence
\[ \bfCoh^{\otimes,-} \circ \hPsi_{Z/X} \simeq \bfCoh^{\otimes,-} \circ \hPsi^{\mathrm{aff}}_{Z/X} . \]
There are also perfect variants
\[ \hrP_{Z/X}^\otimes, \hrP^{\circ,\otimes}_{Z,Z'/X} \colon X\Et\op \to \Catstmon , \]
defined as the compositions
\[ \hrP_{Z/X}^\otimes \coloneqq \bfPerf^\otimes \circ \hPsi_{Z/X}, \qquad \hrP^{\circ, \otimes}_{Z/X} \coloneqq \bfPerf^\otimes \circ \hPsi^\circ_{Z,Z'/X} . \]

\begin{notation}\label{notation:frakCoh}
	In all of our applications, we will be interested in the case where $Z = Z'$.
	When this is the case, we write $\hPsi^\circ_{Z/X}$, $\hPhi^{\circ, \otimes}_{Z / X}$ etc.\ instead of $\hPsi^\circ_{Z,Z' / X}$, $\hPhi^{\circ, \otimes}_{Z, Z' / X}$ etc.
	We allow the case where $Z$ is different from $Z'$ because it will play an important role in the proof of the \'etale descent of $\hPhi^{\circ,\otimes}_{Z/X}$.
	
	Furthermore, when $X$ is clear out of the context and there is no possibility of confusion, we will simply write $\fW_Z$, $\hPsi_Z^\circ$, $\hPhi_Z^{\circ,\otimes}$ etc.\ instead of $\fW_{Z/X}$, $\hPsi^\circ_{Z/X}$, $\hPhi^{\circ, \otimes}_{Z/X}$ etc.
\end{notation}

\begin{rem}
	The functor $\hPsi^\circ_{Z/X}$ plays the role of the punctured formal neighbourhood.
	Nevertheless, it is extremely important to remark that $\hPsi^\circ_{Z/X}$ cannot be a geometric object: indeed, it doesn't take \v{C}ech covers for the \'etale topology on $\Aff_k$ to colimit diagrams in $\St_k$.
	On the other side, the main result of this section shows that $\hPhi^{\circ, \otimes}_{Z/X}$ and $\hrP^{\circ, \otimes}_{Z/X}$ are stacks for the \'etale topology.
	For this reason, we think more of $\hPhi^{\circ, \otimes}_{Z/X}$ and of $\hrP^{\circ, \otimes}_{Z/X}$ rather than $\hPsi^\circ_{Z/X}$ as ``geometric models'' of the punctured formal neighbourhood.
	From a certain point of view, $\hPhi^{\circ, \otimes}_{Z/X}$ and $\hrP^{\circ, \otimes}_{Z/X}$ constitute realizations of the punctured formal neighbourhood as a non-commutative space.
\end{rem}

In the next section we will prove that $\hPhi_Z^{\circ, \otimes}$ is a hypercomplete sheaf for the \'etale topology (cf.\ \cref{thm:etale_descent}).
Let us introduce the site $X\Et^{\Geom}$: its objects are morphisms $Y \to X$, where $Y$ is an Artin stack locally of finite type over $k$, and its topology is the usual \'etale topology for geometric stacks.
There is a natural continuous morphism of sites
\[ X\Et \hookrightarrow X\Et^{\Geom} , \]
and since every Artin stack admits an atlas, we see that the conditions of \cite[Proposition 2.22]{Porta_Yu_Higher_analytic_stacks_2014} are satisfied.
In particular, there is an equivalence of $\infty$-topoi
\[ \St(X\Et, \tauet) \simeq \St(X\Et^{\Geom}, \tauet) . \]
Therefore we can associate to $\hPhi^{\circ, \otimes}_Z \in \St(X\Et, \tauet)$ a unique stack on $(X\Et^{\Geom}, \tauet)$ with values in $\Catstmon$, which we still denote $\hPhi^{\circ, \otimes}_Z$.

Let now $(\Aff_k, \tauet)$ be the category of affine schemes over $k$ of finite type, equipped with the \'etale topology.
The canonical morphism $X \to \Spec(k)$ induces by base change a continuous morphism of sites
\[ u \colon (\Aff_k, \tauet) \to (X\Et^{\Geom}, \tauet) . \]

\begin{defin}\label{def:coherentPFN}
	With the notations introduced above, the functor $\bfCoh^{\otimes,-}_{\hZ \smallsetminus Z}$ is defined as the composition
	\[ \hPhi_Z^{\circ, \otimes} \circ u \colon \Aff_k\op \to \Catstmon . \]
	We further define the category $\Coh^-(\hZ \smallsetminus Z)$ to be the category of $k$-points of $\bfCoh^{\otimes,-}_{\hZ \smallsetminus Z}$, i.e.\
	\[ \Coh^-(\hZ \smallsetminus Z) \coloneqq \bfCoh^{\otimes,-}_{\hZ \smallsetminus Z}(\Spec(k)) . \]
\end{defin}

\begin{rem}
Let us give a description of $\Coh^-(\hZ \smallsetminus Z)$. Recall that $X$ is equivalent to the colimit
\[
X \simeq \colim_{U \to X} \Spec A.
\]
where $U$ runs through affine schemes.
As $\hPhi_Z^{\circ, \otimes}$ is a stack, we get:
\[
\Coh^-(\hZ \smallsetminus Z) \simeq \lim_{U \to X} \hPhi_Z^{\circ,\otimes}(U) \simeq \lim_{U \to X} \Coh^-(\hZ_U \smallsetminus Z_U) .
\]
\end{rem}

\begin{prop} \label{prop:etale_descent_coh}
	The functor $\bfCoh^{\otimes,-}_{\hZ \smallsetminus Z} \colon \Aff_k\op \to \Catstmon$ is a stack for the \'etale topology.
\end{prop}

\begin{proof}
	This is a direct consequence of \cref{thm:etale_descent} and of \cite[Lemma 2.13]{Porta_Yu_Higher_analytic_stacks_2014}.
\end{proof}

\subsection{\'Etale descent for almost perfect modules} \label{subsec:etale_descent}

The main goal of this section is to prove the following result:

\begin{thm} \label{thm:etale_descent}
	Let $X$ be an Artin stack locally of finite type over $k$.
	Suppose that $Z \hookrightarrow X$ is a closed substack.
	Then functor $\hPhi_Z^{\circ, \otimes} \colon X\Et\op \to \Catstmon$ is a hypercomplete sheaf with respect to the \'etale topology $\tauet$ on $X\Et$.
\end{thm}

The proof will take us the rest of this section.

\subsubsection{Initial reductions} \label{subsubsec:initial_reduction}

We start with a couple of easy reductions.
First of all, let us remember that the forgetful functor
\[ G \colon \Catstmon = \CAlg(\Cat_\infty^{\mathrm{Ex}, \times}) \to \Catst \]
is conservative and commutes with arbitrary limits in virtue of Lurie-Barr-Beck theorem (cf.\ \cite[4.7.4.5]{Lurie_Higher_algebra}).
It is therefore enough to prove \cref{thm:etale_descent} for the downgraded functor
\[ \hPhi_Z^\circ \coloneqq G \circ \hPhi_Z^{\circ, \otimes} \colon X\Et\op \to \Catst . \]

\begin{notation}
	All the functors valued in $\Catstmon$ introduced in \cref{subsec:punctured_formal_neighbourhood} admit a downgraded version valued in $\Catst$ and obtained by composition with $G \colon \Catstmon \to \Catst$.
	To distinguish between these two versions, we simply drop the symbol $\otimes$ in the notation of the functor.
	So $\hPhi_Z$ is the downgraded version of $\hPhi_Z^\otimes$, $\hPhi_Z^\circ$ is the downgraded version of $\hPhi_Z^{\circ, \otimes}$ etc.
	Thanks to the above observation, from now on we will only consider functors with values in $\Catst$ and not in $\Catstmon$.
\end{notation}

Furthermore we can easily reduce ourselves to the case where $X$ is affine.
Indeed, since $X\Et$ is by definition the affine \'etale site of $X$, it is sufficient to prove that for every affine scheme $U = \Spec(A_U)$ equipped with a morphism $U \to X$ and every \'etale affine hypercover $\{U^\bullet \to U\}$, the canonical morphism
\[ \hPhi_Z^\circ(U) \to \varprojlim_{n \in \mathbf \Delta} \hPhi_Z^\circ(U^n) \]
is an equivalence.
For this, it is enough to restrict ourselves to the case where $X = U$.
We can therefore assume from the very beginning that $X$ is an affine scheme of finite type.
Moreover, it is enough to show that the restriction of $\hPhi_Z^\circ$ to the \emph{small} \'etale site of $X$ is a hypercomplete sheaf.
Committing a slight abuse of notation, we still denote by $\hPhi_Z^\circ$ this restriction, and therefore we consider $\hPhi_Z^\circ$ as a functor
\[ \hPhi^\circ_Z \colon X\et \to \Catst . \]

As usual, we write $X = \Spec(A)$ and we let $J$ be the defining ideal of $Z$ inside $X$.
If $U \to \Spec(A)$ is an \'etale map with affine source, we write $U = \Spec(A_U)$ and let $J_U \coloneqq J A_U$ be the image of the ideal $J$ in $A_U$.
Furthermore, we denote by $\widehat{A_U}$ the formal completion of $A_U$ at $J_U$.

Since $A$ is noetherian, we can choose a finite number of generators $f_1, \ldots, f_n$ for $J$.
For every subset $S$ of $[n] \coloneqq \{1, \ldots, n\}$, we let $J_S$ denote the ideal $(f_i)_{i \in S}$ and by $Z_S \coloneqq V(J_S)$ the closed subscheme determined by $J_S$.
In the previous section we introduced a presheaf
\[ \hPhi^\circ_{Z,Z_S / X} \colon X\et \to \Catst , \]
and $\hPhi^\circ_Z = \hPhi^\circ_{Z,Z_{[n]} / X } = \hPhi^\circ_{Z,Z/X}$.
We will prove that the presheaves $\hPhi^\circ_{Z,Z_S / X}$ are hypercomplete sheaves with respect to $\tauet$ by induction on the cardinality of $S$.

The proof of the base step will occupy us for a while.
Before getting there, let us show how to deal with the induction step.
So, suppose that the result has been proven for all $X = \Spec(A)$, all the ideals $J \subset A$ generated by $n$ elements $f_1, \ldots, f_n$ and all the subsets $S \subset [n]$ of $m-1$ elements.
Choose a subset $S' \subset [n]$ of $m$ elements.
Without loss of generality, we can assume $n = m$ and thus $S' = [n]$.
Let $J_1 \coloneqq (f_1)$ and $J_2 \coloneqq (f_2, \ldots, f_n)$.
Set $Z_1 \coloneqq V(J_1)$, $Z_2 \coloneqq V(J_2)$.
Finally, let $J_{12} \coloneqq (f_1 f_2, \ldots, f_1 f_n)$ and $Z_{12} \coloneqq V(J_{12})$.
Observe that
\[ \hPsi^\circ_{Z,Z_1 / X} \cap \hPsi^\circ_{Z,Z_2 / X} = \hPsi^\circ_{Z,Z_{12} / X} . \]
It is now sufficient to remark that, by Zariski descent, we have
\[ \hPhi^\circ_{Z / X} \simeq \hPhi^\circ_{Z, Z_1 / X} \times_{\hPhi^\circ_{Z, Z_{12} / X}} \hPhi^\circ_{Z, Z_2 / X} . \]
Indeed, the induction hypothesis shows that $\hPhi^\circ_{Z, Z_1, X}$, $\hPhi^\circ_{Z,Z_{12} / X}$ and $\hPhi^\circ_{Z, Z_2 / X}$ are hypercomplete sheaves with respect to $\tauet$.
Thus, it follows that the same goes for $\hPhi^\circ_{Z / X}$.

\subsubsection{Reduction to the torsion-free case} \label{subsubsec:reduction_torsionfree}

To summarize, we are reduced to prove the following statement.
Let $X = \Spec(A)$ and let $J$ be the ideal defining $Z$ inside $X$.
Let $f \in J$ be an element and let $Z' \coloneqq V(f)$.
Then
\[ \hPhi^\circ_{Z,f} \coloneqq \hPhi^\circ_{Z, Z'/X} \colon X\Et \to \Catst \]
is a hypercomplete sheaf for $\tauet$.
Observe furthermore that, given $U = \Spec(A_U)$ in $X\Et$, we have
\[ \hPhi^\circ_{Z,f}(U) = \Coh^-(\widehat{A_U}[f_U\inv]) , \]
where $\widehat{A_U}$ denotes the formal completion of $A_U$ at $J_U$ and $f_U$ denotes the image of $f$ via $A \to A_U$.

The proof of the base step of the induction started in \S \ref{subsubsec:initial_reduction} relies in an essential way on the results of \cite{Bosch_Gortz_Coherent_modules_1998}.
However, in order to apply the main result in loc.\ cit.\ it is important to reduce ourselves to the case where the algebras $\widehat{A_U}$ don't have $f_U$-torsion.
To see why this is possible, we need the following elementary lemma of commutative algebra, which will be handy also in what follows.

\begin{lem} \label{lem:etale_sent_to_flat}
	Let $U \to V$ be a flat morphism in $X\Et$.
	Then the induced map $\hat{g} \colon \widehat{A_V} \to \widehat{A_U}$ is flat.
\end{lem}

\begin{proof}
	We apply \cite[Tag 0523]{stacks-project} taking $R = \widehat{A_V}$, $S = \widehat{A_U}$, $I = J_V \widehat{A_V}$ and $M = S$.
	Since $A_V$ and $A_U$ are noetherian, the same goes for $R$ and $S$ (cf.\ \cite[Theorem 10.26]{Atiyah_Macdonald_Commutative_algebra}).
	Let us show that for every $n$, $S / I^n S$ is flat over $R / I^n$.
	Indeed, $I^n S = J_V^n \widehat{A_U} = J_U^n \widehat{A_U}$.
	Therefore,
	\[ S / I^n S = \widehat{A_U} / J_U^n \widehat{A_U} = A_U / J_U^n = A_U \otimes_{A_V} (A_V / J_V^n) . \]
	Since $A_V \to A_U$ is flat by assumption, the same goes for $A_V / J_V^n \to A_U / J_U^n$, and hence for $\widehat{A_V} / J_V^n \widehat{A_V} \to \widehat{A_U} / J_U^n \widehat{A_U}$.
	Therefore, the hypotheses of \cite[Tag 0523]{stacks-project} are satisfied. The conclusion now follows from \cite[Theorem 8.14]{Matsumura_Commutative_ring}, which implies that $J_U \widehat{A_U}$ is contained in the Jacobson radical of $\widehat{A_U}$.
\end{proof}

\begin{cor}
	Let $U \to V$ be an \'etale morphism in $X\Et$.
	Let $I_{f,V}$ be the $f_V$-torsion in $\widehat{A_V}$ and similarly let $I_{f,U}$ be the $f_U$-torsion in $\widehat{A_U}$.
	Then $I_{f,U} = I_{f,V} \widehat{A_U}$.
\end{cor}

\begin{proof}
	Observe that $I_{f,V}$ is the kernel of the map $m_{f_V} \colon \widehat{A_V} \to \widehat{A_V}$ given by multiplication by $f_V$.
	Since the map $\widehat{A_V} \to \widehat{A_U}$ is flat by \cref{lem:etale_sent_to_flat}, we conclude that
	\[ I_{f,V} \widehat{A_U} = I_{f,V} \otimes_{\widehat{A_V}} \widehat{A_U} \]
	is the kernel of the map $m_{f_V} \otimes \widehat{A_U} = m_{f_U}$.
	In other words, $I_{f,V} \widehat{A_U} = I_{f,U}$.
\end{proof}

Keeping the notations of the above corollary, we introduce the ring
\[ B_U \coloneqq \widehat{A_U} / I_{f,U} . \]
Observe that if $U \to V$ is an \'etale map in $X\et$, the above corollary guarantees that
\[ B_U \simeq B_V \otimes_{\widehat{A_V}} \widehat{A_U} . \]
Combining this remark with \cref{lem:etale_sent_to_flat}, we conclude that the map $B_V \to B_U$ is again flat.

Committing a slight abuse of notation, we still denote by $f_U$ the image of $f_U$ under $\widehat{A_U} \to B_U$.
Then there is a canonical isomorphism
\[ B_U[f_U\inv] \simeq \widehat{A_U}[f_U\inv] . \]
As consequence, we can factor the functor $\hPhi^\circ_{Z,f} \colon X\et \to \Catst$ through
\[ \Spf(B)\et \to \Catst . \]
In other words, we can assume out of the very beginning that $\hA$ doesn't have $f$-torsion.

\subsubsection{The base step for the heart}

We keep working with the notations introduced in the previous paragraphs.
Observe that for every $U \in X\et$, the stable $\infty$-category
\[ \hPhi^\circ_{Z,f}(U) = \Coh^-(\widehat{A_U}[f_U\inv]) \]
has a canonical $t$-structure.
Furthermore, \cref{lem:etale_sent_to_flat} implies that given a map $i \colon U \to V$ in $X\et$, the induced morphism
\[ \widehat{A_V} \to \widehat{A_U} \]
is flat.
As consequence, the morphism
\[ \widehat{A_V}[f_V\inv] \to \widehat{A_U}[f_U\inv] \]
is flat as well.
It follows that $\hPhi^\circ_{Z,f}(i) \colon \hPhi^\circ_{Z,f}(V) \to \hPhi^\circ_{Z,f}(U)$ is $t$-exact.
This allows us to promote $\hPhi^\circ_{Z,f}$ to an $\infty$-functor
\[ \hPhi^\circ_{Z,f} \colon X\et \to \Catstlc . \]
Composing with
\[ (-)^\heartsuit \colon \Catstlc \to \AbCat , \]
we obtain a new functor $(\hPhi^\circ_{Z,f})^\heartsuit \colon X\et \to \AbCat$.
The reductions performed in \S \ref{subsubsec:initial_reduction} and in \S \ref{subsubsec:reduction_torsionfree} allow us finally to invoke \cite{Bosch_Gortz_Coherent_modules_1998} to conclude that $(\hPhi^\circ_{Z,f})^\heartsuit$ has \'etale descent.

Let us explain why this is true.
Recall that at the beginning of \S \ref{subsubsec:initial_reduction} we fixed a finite number of generators $f_1, \ldots, f_n$ for $J$.
Furthermore, we can always assume that $f$ is among these generators.
To simplify notations, we assume $f_1 = f$.
Consider the map $k[T_1, \ldots, T_n] \to A$ defined by sending $T_i$ to $f_i$.
Passing to formal completions at $(T_1, \ldots, T_n)$ and at $J$ respectively, we obtain a map
\[ R \coloneqq k \llb T_1, \ldots, T_n \rrb  \to \widehat{A} . \]
Observe that $(T_1, \ldots, T_n) \hA = J \hA$ is an ideal of definition for $\hA$.
Furthermore, $\hA / I \hA \simeq A / J$ is finitely generated over $k$, and hence over $k\llb T_1, \ldots, T_n\rrb$.
Thus it follows that $\hA$ is topologically of finite generation over $R$ (cf.\ \cite[Chapter 0, 7.5.3]{EGA1}).
Since $\hA$ is noetherian, it also follows that it is topologically of finite presentation.

We consider $R$ as a ring of type (N) in the sense of \cite[\S 1]{Bosch_Gortz_Coherent_modules_1998}.
In order to do this, we have to specify an ideal $I$ inside $R$.
We let $I = (T_1)$.
Observe that the reduction performed in \S \ref{subsubsec:reduction_torsionfree} allows us to assume that $\widehat{A}$ does not have $I$-torsion.
In other words, $\widehat{Z_U} \simeq \Spf(\widehat{A_U})$ is an admissible formal $R$-scheme.
We let $\widehat{Z_U}_{\mathrm{rig}}$ denote the rigid space associated to $\fX$ introduced in \cite[\S 1]{Bosch_Gortz_Coherent_modules_1998}.
The very definition of fpqc maps of rigid spaces given in \cite[\S 3]{Bosch_Gortz_Coherent_modules_1998} shows that the assignment
\[ \fU \mapsto \fU_{\mathrm{rig}} \]
defines a morphism of sites
\[ \hZ\et \to (\hZ_{\mathrm{rig}})_{\mathrm{fpqc}} . \]
Furthermore, we have:

\begin{lem}
	For every $U \in X\et$, there is a canonical equivalence
	\[ (\hPhi^\circ_{Z,f}(U))^\heartsuit \simeq \Coh^\heartsuit(\widehat{Z_U}_{\mathrm{rig}}) . \]
\end{lem}

\begin{proof}
	This is precisely the content of \cite[Corollary 1.8]{Bosch_Gortz_Coherent_modules_1998}.
\end{proof}

At this point, \cite[Theorem 3.1]{Bosch_Gortz_Coherent_modules_1998} implies directly:

\begin{prop} \label{prop:etale_descent_heart}
	The functor $(\hPhi^\circ_{Z,f})^\heartsuit \colon X\et \to \AbCat$ is a $\tauet$-sheaf.
\end{prop}

\begin{cor} \label{cor:etale_hyperdescent_heart}
	The functor $(\hPhi^\circ_{Z,f})^\heartsuit \colon X\et \to \AbCat$ is a hypercomplete $\tauet$-sheaf.
\end{cor}

\begin{proof}
	Since $\AbCat$ is a $1$-category, this statement is a consequence of \cref{prop:etale_descent_heart} and of, for example, \cite[Proposition A.1]{Porta_Comparison_2015}.
\end{proof}

\subsubsection{The base step in general}

We now turn to the proof that the $\infty$-functor $\hPhi^\circ_{Z,f} \colon X\et \to \Catst$ satisfies $\tauet$-hyperdescent.
This will complete the proof of the base step of the induction started in \S \ref{subsubsec:initial_reduction} and as consequence also the proof of \cref{thm:etale_descent}.

We need the following result on the behaviour of $t$-structures in families:

\begin{lem} \label{lem:flat_cosimplicial_t_structure}
	Let $p \colon \cX \to S$ be a stable fibration and let $F \colon S\op \to \Catst$ be the associated $\infty$-functor.
	Suppose furthermore that:
	\begin{enumerate}
		\item for every $s \in S$, we are given a $t$-structure $(\cX_s^{\le 0}, \cX_s^{\ge 0})$ on the fiber $\cX_s$;
		\item for every edge $f \colon s \to s'$ the induced functor $f^* \colon \cX_{s'} \to \cX_s$ is $t$-exact.
	\end{enumerate}
	Then the stable $\infty$-category $\varprojlim F$ has a (unique) $t$-structure characterised by the requirement that for every $s \in S$ the induced functor
	\[ e_s \colon \varprojlim F \to \cX_s \]
	is $t$-exact.
\end{lem}

\begin{proof}
	Using \cite[3.3.3.2]{HTT} we can represent $X$ as the $\infty$-category of cartesian sections of $p \colon \cX \to S$:
	\[ \varprojlim F \simeq \Map_S^{\flat}(S^\sharp, \cX^\natural) . \]
	In order to define a $t$-structure on $\varprojlim F$ it is enough to define a $t$-structure on
	\[ \cC \coloneqq \Map_{S}^{\flat}(S^\sharp, \cX^\natural) . \]
	Define $\cC^{\le 0}$ (resp.\ $\cC^{\ge 0}$) as the full subcategory of $\cC$ spanned by those $x \in \cC$ such that
	\[ x(s) \in \cX_s^{\le 0} \qquad (\mathrm{resp.\ } x(s) \in \cX_s^{\ge 0}) . \]
	for every $s \in S$.
	We claim that this defines a $t$-structure on $\cC$.
	Let $x \in \cC^{\le 0}$ and $y \in \cC^{\ge 1}$.
	Then we have
	\[ \Map_{\cC}(x,y) = \varprojlim_{s \in S} \Map_{\cX_s}(x(s), y(s)) = 0 . \]
	This proves the orthogonality condition.
	
	We next prove that $\cC^{\le 0}[-1] \subset \cC^{\le 1}$ and $\cC^{\ge 0}[1] \subset \cC^{\ge -1}$.
	Since $p \colon \cX \to S$ is a stable fibration, we know that for every edge $f \colon s \to s'$ in $S$ the induced functor $f^* \colon \cX_{s'} \to \cX_s$ is an exact functor between stable $\infty$-categories.
	This implies that finite limits and finite colimits in $\cC$ can be computed objectwise.
	Therefore the previous conditions can be checked sectionwise and thus they follow from the fact that for every $s \in S$, $(\cX_s^{\le 0}, \cX_s^{\ge 0})$ forms a $t$-structure on $\cX_s$.
	
	We finally prove that for every $x \in \cC$ there exists a fiber sequence
	\[ x' \to x \to x'' , \]
	where $x' \in \cC^{\le 0}$ and $x'' \in \cC^{\ge 1}$.
	Let $\cX^{\le 0}$ be the full subcategory of $\cX$ spanned by those objects $x \in \cX$ such that $x \in \cX_{p(x)}^{\le 0}$.
	We denote by $j \colon \cX^{\le 0} \hookrightarrow \cX$ the inclusion.
	The composite functor $p\circ j$ is again a Cartesian fibration, and the inclusion $j \colon \cX^{\le 0} \hookrightarrow \cX$ preserves cartesian edges.
	Furthermore, it follows from \cite[1.2.1.5]{Lurie_Higher_algebra} that $j$ admits fiberwise a right adjoint.
	Therefore, the hypotheses of \cite[7.3.2.6]{Lurie_Higher_algebra} are satisfied.
	Thus, $j$ admits a right adjoint $\tau_{\le 0} \colon \cX \to \cX^{\le 0}$ relative to $S$.
	We claim that $\tau_{\le 0}$ preserves Cartesian edges as well.
	This amounts to show that for every edge $f \colon s \to s'$ in $S$, the induced square
	\[ \begin{tikzcd}
	\cX_{s'} \arrow{r}{f^*} \arrow{d}{\tau_{\le 0}} & \cX_s \arrow{d}{\tau_{\le 0}} \\
	\cX_{s'}^{\le 0} \arrow{r}{f^*} & \cX_s^{\le 0}
	\end{tikzcd} \]
	is commutative.
	Since $f^*$ is $t$-exact by construction, it commutes with truncations.
	Therefore the assertion is proved, and $\tau_{\le 0} \colon \cX \to \cX^{\le 0}$ preserves Cartesian edges.
	
	In particular, composition with $\tau_{\le 0}$ preserves Cartesian sections, and therefore we obtain a well defined functor
	\[ \tau_{\le 0}^* \colon \cC \to \cC^{\le 0} \]
	formally defined by $\tau_{\le 0}^*(x) \coloneqq \tau_{\le 0} \circ x$.
	Since $\tau_{\le 0}$ is right adjoint to $j$, we see that $\tau_{\le 0}^*$ is right adjoint to the inclusion $j^* \colon \cC^{\le 0} \to \cC$.
	In particular, for every $x \in \cC$ there is a natural transformation
	\[ \varepsilon_x \colon \tau_{\le 0} \circ x \to x . \]
	Set $x'' \coloneqq \cofib( \varepsilon_x )$.
	Since finite colimits of cartesian sections can be computed objectwise, we see indeed that $x''(s) \in \cC^{\ge 1}$ for every $s \in S$.
	This completes the proof that $(\cC^{\le 0}, \cC^{\ge 0})$ forms a $t$-structure on $\cC$, and therefore the proof of the lemma is achieved.
\end{proof}

We are finally ready to prove the main result.
Let us start by fixing a couple of notations.
Let $u^\bullet \colon U^\bullet \to X$ be an \'etale hypercover in $X\et$.
To ease the notation, we simply write $B_n$ for $A_{U^n}$ and $f_n$ for $f_{U^n}$.
Furthermore, we write
\[ v^\bullet \colon \widehat{A}[f\inv] \to \widehat{B_n}[f_n\inv] \]
be the induced morphisms.
In this way, we get a natural functor
\[ v^{\bullet *} \colon \hPsi^\circ_{Z,f}(X) \to \varprojlim_{n \in \mathbf \Delta} \hPsi^\circ_{Z,f}(U^n) . \]
Unravelling the definitions, we can rewrite the above map as
\[ v^{\bullet *} \colon \Coh^-(\widehat{A}[f\inv]) \to \varprojlim_{n \in \mathbf \Delta} \Coh^-(\widehat{B_n}[f_n\inv]) . \]
Committing a slight abuse of notation, we still denote by $v^{\bullet *}$ the induced map
\[ v^{\bullet *} \colon \QCoh(\widehat{A}[f\inv]) \to \varprojlim_{n \in \mathbf \Delta} \QCoh(\widehat{B_n}[f_n\inv]) . \]
Since the categories $\QCoh(\widehat{A}[f\inv])$ and $\QCoh(\widehat{B_n}[f_n\inv])$ are presentable, we see that $v^{\bullet *}$ has a right adjoint, which we denote $v^\bullet_*$.
It follows from \cite[Corollary 8.6]{Porta_Yu_Higher_analytic_stacks_2014} that $v^\bullet_*$ can be identified with the functor given by the informal assignment
\[ \{\cF^n\} \mapsto \varprojlim v^n_* \cF^n . \]

As direct consequence of \cref{prop:etale_descent_heart} we obtain the following:

\begin{lem}[{cf.\ \cite[Step 1]{Conrad_Descent_for_coherent_2003}}] \label{lem:right_adjointable_key_lemma}
	Let $u^\bullet \colon U^\bullet \to X$ be an \'etale hypercover in $X\et$.
	Let
	\[ \{\cF^n\} \in \varprojlim_{n \in \mathbf \Delta} \Coh^\heartsuit(\widehat{B_n}[f_n\inv]) . \]
	Let $\cF \simeq \mathrm{eq}( \cF^0 \rightrightarrows \cF^1 ) \in \Coh^\heartsuit(\widehat{A}[f\inv])$ be the coherent $\widehat{A}[f\inv]$-module obtained by descent.
	Then there is a canonical equivalence
	\[ \cF \simeq v^\bullet_*( \{\cF^n\} ) . \]
	In particular, the square
	\[ \begin{tikzcd} \label{eq:right_adjointable_for_devissage}
		\Cohh(\widehat{A}[f\inv]) \arrow{d}{i}  \arrow{r}{v^{\bullet *}} & \varprojlim \Cohh(\widehat{B_n}[f_n\inv]) \arrow{d}{i^\bullet} \\
		\QCoh(\widehat{A}[f\inv]) \arrow{r}{v^{\bullet *}} & \varprojlim \QCoh(\widehat{B_n}[f_n\inv])
	\end{tikzcd} \]
	is right adjointable.
\end{lem}

\begin{proof}
	To start with, we consider the commutative square
	\[ \begin{tikzcd}
		\QCoh(\widehat{A}) \arrow{r}{\hat{u}^{\bullet *}} \arrow{d}{i^*} & \varprojlim \QCoh(\widehat{B_n}) \arrow{d}{i^{\bullet *}} \\
		\QCoh(\widehat{A}[f\inv]) \arrow{r}{v^{\bullet *}} & \varprojlim \QCoh(\widehat{B_n}[f_n\inv]) .
	\end{tikzcd} \]
	Since both $\hat{u}^{\bullet *}$ and $v^{\bullet *}$ have right adjoints, we obtain the Beck-Chevalley transformation.
	\[ \varphi \colon i^* \circ \hat{u}^\bullet_* \to v^\bullet_* \circ i^{\bullet *} . \]
	Moreover, using Lemmas \ref{lem:etale_sent_to_flat} and \ref{lem:flat_cosimplicial_t_structure} we can endow the categories
	\[ \varprojlim \QCoh(\widehat{B_n}), \quad \varprojlim \QCoh(\widehat{B_n}[f_n\inv]) \]
	with $t$-structures such that $i^{\bullet *}$ becomes a $t$-exact functor.
	
	Let $\{\cF^n\} \in \varprojlim \QCoh(\widehat{B_n})$ and suppose that for every $n \in \mathbf \Delta$, $\cF^n \in \QCoh^{\heartsuit}(\widehat{B_n})$.
	We claim that in this case $\varphi_{\{\cF^n\}}$ is an equivalence.
	To see this, we remark that, since $\QCoh(\widehat{A}[f\inv])$ is canonically equipped with a complete $t$-structure, it is enough to prove that $\rH^i(\varphi_{\{\cF^n\}})$ is an equivalence for every $i \in \mathbb Z$.
	Observe that $\hat{u}^{\bullet}_*$ and $v^\bullet_*$ are right $t$-exact.
	On the other side, $i^*$ and $i^{\bullet *}$ are $t$-exact.
	Thus, we deduce that
	\[ \rH^i(i^* \hat{u}^\bullet_*(\{\cF^n\})) \simeq 0 \simeq \rH^i(v^\bullet_* i^{\bullet *}(\{\cF^n\})) \]
	for every $i < 0$.
	We now let $i \ge 0$.
	In this case, we first remark that
	\[ \rH^i (i^* \hat{u}^\bullet_*( \{\cF^n\} )) \simeq i^* \rH^i(\hat{u}^\bullet_*( \{\cF^n\} )) . \]
	Next, we observe that $\rH^i(\hat{u}^\bullet_*(\{\cF^n\}))$ can be computed by a finite limit inside $\QCoh^\heartsuit(\widehat{A})$.
	Since $i^* \colon \QCoh(\widehat{A}) \to \QCoh(\widehat{A}[f\inv])$ is $t$-exact, we deduce that $i^*$ commutes with this particular limit.
	It follows that $i^* \rH^i(\hat{u}^\bullet_*( \{\cF^n\} ))$ can be canonically identified with $\rH^i( v^\bullet_*( \{i^{n*} cF^n\} ) )$.
	This completes the proof of the claim.
	
	We can now complete the proof of the lemma.
	Fix therefore $\{\cF^n\} \in \varprojlim \Coh^\heartsuit(\widehat{B_n}[f_n\inv])$.
	Using the effectivity of the descent proven in \cref{cor:etale_hyperdescent_heart} we can rewrite $\{\cF^n\}$ as $v^{\bullet *}(\cF)$, where
	\[ \cF \coloneqq \mathrm{eq}( \cF^0 \rightrightarrows \cF^1 ) . \]
	Furthermore, the localization map
	\[ i^* \colon \Coh^\heartsuit(\widehat{A}) \to \Coh^\heartsuit(\widehat{A}[f\inv]) \]
	is essentially surjective.
	We can therefore choose $\cG \in \Coh^\heartsuit(\hA)$ such that $i^*(\cG) \simeq \cF$.
	Since the canonical map
	\[ \hat{u}^{\bullet *} \colon \Coh^-(\widehat{A}) \to \varprojlim_{n \in \mathbf \Delta} \Coh^-(\widehat{B_n}) \]
	is an equivalence of stable $\infty$-categories, we see that the canonical morphism
	\[ \cG \to \hat{u}^{\bullet}_* \hat{u}^{\bullet *} \cG \]
	is an equivalence.
	It is now sufficient to remark that, since $\varphi_{\{\cF^n\}}$ is an equivalence, it follows that $i^*$ takes this map to the unit
	\[ \cF \simeq i^*(\cG) \to v^\bullet_* v^{\bullet *} (i^*(\cG)) \simeq v^\bullet( \{\cF^n\} ) . \]
	The proof is thus complete.
\end{proof}

Consider now the following diagram
\[ \begin{tikzcd}
\Coh^-(\widehat{A}[f\inv]) \arrow{r}{v^{\bullet *}} \arrow{d}{i_-} & \varprojlim \Coh^-(\widehat{B_n}[f_n\inv]) \arrow{d}{i^\bullet_-} \\
\QCoh(\widehat{A}[f\inv]) \arrow{r}{v^{\bullet *}} & \varprojlim \QCoh(\widehat{B_n}[f_n\inv]) .
\end{tikzcd} \]
It follows from Lemmas \ref{lem:etale_sent_to_flat} and \ref{lem:flat_cosimplicial_t_structure} that both
\[ \varprojlim \Coh^-(\widehat{B_n}[f_n\inv]), \quad \varprojlim \QCoh(\widehat{B_n}[f_n\inv]) \]
have $t$-structures and that all the functors $i_-$, $i_-^\bullet$ and $v^{\bullet *}$ are $t$-exact.

\begin{lem}
	The induced functor
	\[ \begin{tikzcd}
	\varprojlim \Coh^-(\widehat{B_n}[f_n\inv]) \arrow{r}{i^\bullet_-} & \varprojlim \QCoh(\widehat{B_n}[f_n\inv]) \arrow{r}{v^\bullet_*} & \QCoh(\widehat{A}[f\inv])
	\end{tikzcd} \]
	is $t$-exact.
\end{lem}

\begin{proof}
	It is easily seen that this composition is right $t$-exact: since $v^{\bullet *}$ is $t$-exact, it follows from adjunction that $v^\bullet_*$ is right $t$-exact.
	Therefore $v^\bullet_* \circ i^\bullet_-$ is right $t$-exact as the composition of two functors with this property.
	Let us show that $v^\bullet_* \circ i^\bullet_-$ is also left $t$-exact.
	Let $\{\cF^n\} \in \varprojlim \Coh^-(\widehat{B_n}[f_n\inv])$ and suppose that
	\[ \cF^n \in \Coh^{\le 0}(\widehat{B_n}[f_n\inv]) \]
	for every $n \ge 0$.
	In this case, the double complex spectral sequence reads:
	\[ E_1^{p,q} = \rH^p( v^q_*(\cF^q) ) \Rightarrow \rH^{p+q}( \varprojlim v^n_* \cF^n ) . \]
	Since the maps $v^q$ are affine, we can rewrite the $E_1$-page as
	\[ E_1^{p,q} = v^q_*( \rH^p( \cF^q ) ) . \]
	Recall that the sheaves $\rH^p( \cF^q )$ are coherent.
	Combining \cref{cor:etale_hyperdescent_heart} with right adjointability of the square \eqref{eq:right_adjointable_for_devissage}, we conclude that the cosimplicial objects
	\[ \{v^q_*( \rH^p(\cF^q) )\}_{q \in \Delta} \]
	are acyclic for every $q \ge 1$.
	Therefore, the spectral sequence degenerates at the second page.
	Since $\cF^n \in \Coh^{\le 0}(\widehat{B_n}[f_n\inv])$, we conclude that
	\[ v^\bullet_*( i^\bullet_- \{\cF^n\} ) \in \QCoh^{\le 0}(\widehat{A}[f\inv]) . \]
\end{proof}

Observe furthermore that the spectral sequence argument given in the above lemma shows that if $\{\cF^n\} \in \varprojlim \Coh^-(\widehat{B_n}[f_n\inv])$, then the cohomologies of $v^\bullet_*( \{\cF^n\} )$ are coherent.
As consequence, the functor $u^\bullet_* \circ i^\bullet_-$ factors as a $t$-exact functor
\[ R \colon \varprojlim \Coh^-(\widehat{B_n}[f_n\inv]) \to \Coh^-(\widehat{A}[f\inv]) . \]
Since the functors $i_-$ and $i^\bullet_-$ are fully faithful, we conclude that $R$ is right adjoint to $u^{\bullet *}$.

Finally, we can complete the proof of \cref{thm:etale_descent}.
All the preparations made until this point show that it is enough to prove that the map
\[ v^{\bullet *} \colon \Coh^-(\widehat{A}[f\inv]) \to \varprojlim \Coh^-(\widehat{B_n}[f_n\inv]) \]
is an equivalence.
Since $v^{\bullet *}$ has a right adjoint $R$ it is enough to prove that both the unit $\eta$ and counit $\varepsilon$ are equivalences.
Since both $v^{\bullet *}$ and $R$ are $t$-exact, it is enough to prove that $\rH^i(\eta)$ and $\rH^i(\varepsilon)$ are equivalences.
But this is precisely the content of \cref{cor:etale_hyperdescent_heart}.

\subsection{\'Etale descent for perfect modules} \label{subsec:perfect_complexes}

In this section we show how to obtain an analogous for perfect modules of the stack $\bfCoh^{\otimes,-}_{\hZ \smallsetminus Z}$ we constructed in \cref{subsec:punctured_formal_neighbourhood}.
As usual, $X$ denotes a fixed Artin stack locally of finite presentation over $k$.

From \cref{thm:etale_descent} we can deduce that $\hrP_Z^\circ$ is a hypercomplete sheaf:

\begin{prop} \label{prop:restricted_etale_descent_Perf}
	The functor $\hrP_Z^{\circ, \otimes} \colon X\Et\op \to \Catstmon$ introduced in \cref{subsec:punctured_formal_neighbourhood} is a hypercomplete sheaf.
\end{prop}

\begin{proof}
	As in the case of $\hPhi_Z^{\circ, \otimes}$, it is straightforward to reduce to prove the same statement for $\hrP_Z^\circ$ and to the case where $X$ is affine.
	
	It follows from \cref{thm:etale_descent} that $\hPhi_Z^\circ \colon X\et\op \to \Catst$ is a sheaf.
	Since $\hrP_Z^\circ$ is a fully faithful sub-presheaf of $\hPhi^\circ_Z$, it is enough to check that the property of belonging to $\hrP_Z^\circ$ is local for the \'etale topology.
	Let therefore $\{U_i \to U\}$ be an (affine) \'etale cover in $X\et$.
	Set $U^0 \coloneqq \coprod U_i$ and let $u \colon U^0 \to U$ be the induced map.
	Observe that the diagram
	\[ \begin{tikzcd}
		\hPsi^\circ_Z(U^0) \arrow{r}{v} \arrow{d}{j_{U^0}} & \hPsi^\circ_Z(U) \arrow{d} \\
		\Spec( \widehat{A_{U^0}} ) \arrow{r}{\hat{u}} & \Spec( \widehat{A_U} )
	\end{tikzcd} \]
	is a pullback square of $k$-schemes, where $v = \hPsi^\circ_Z(u)$.
	Using \cref{lem:etale_sent_to_flat}, we see that the map $\hat{u}$ is faithfully flat.
	Thus, we deduce that $v$ is faithfully flat as well.
	In particular, an almost perfect complex $\cF$ on $\hPsi^\circ_Z(U)$ is perfect if and only if $v^*(\cF)$ is perfect.
	The proof is therefore complete.
\end{proof}

Using again the equivalence $\St(X\Et, \tauet) \simeq \St(X\Et^{\Geom}, \tauet)$ as at the end of \cref{subsec:punctured_formal_neighbourhood}, we can extend $\hrP_Z^{\circ, \otimes}$ to an object in $\St(X\Et^{\Geom}, \tauet)$, which we still denote $\hrP_Z^{\circ, \otimes}$.
Finally, using again the continuous morphism of sites
\[ u \colon (\Aff_k, \tauet) \to (X\Et^{\Geom}, \tauet) \]
given by base change along the structural morphism $X \to \Spec(k)$, we can give the following definition:

\begin{defin} \label{def:perfect_PFN}
	We define the functor $\bfPerf^\otimes_{\hZ \smallsetminus Z}$ to be the composition
	\[ \bfPerf^\otimes_{\hZ \smallsetminus Z} \coloneqq \hrP^{\circ, \otimes}_Z \circ u \colon \Aff_k\op \to \Catstmon . \]
\end{defin}

\begin{prop} \label{prop:etale_descent_Perf}
	The functor $\bfPerf^\otimes_{\hZ \smallsetminus Z} \colon \Aff_k\op \to \Catst$ is a stack with respect to the \'etale topology.
\end{prop}

\begin{proof}
	This is a direct consequence of \cref{prop:restricted_etale_descent_Perf} and \cite[Lemma 2.13]{Porta_Yu_Higher_analytic_stacks_2014}.
\end{proof}

\section{Formal glueing and flag decomposition for (almost) perfect modules} \label{sec:formal_glueing}

Let $X$ be an Artin stack locally of finite presentation over $k$, and let $Z \hookrightarrow X$ be a closed substack.
In this section we prove our main formal decomposition results.
The first of these results (see \cref{cor:one_step_flag_decomposition_cohperf} below) is a ``formal glueing'' type of decomposition: it describes the stack (and in particular the category) of (almost) perfect modules on $X$ as the fiber product of the one on the open complementary $V \coloneqq X \smallsetminus Z$, the one on the formal completion $\hZ$, over the one on the formal punctured neighbourhood:
\[ \bfCoh^{\otimes,-}_X \simeq \bfCoh^{\otimes,-}_V \times_{\bfCoh^{\otimes,-}_{\hZ \smallsetminus Z}} \bfCoh^{\otimes,-}_{\hZ} , \qquad \bfPerf^\otimes_X \simeq \bfPerf^\otimes_V \times_{\bfPerf^\otimes_{\hZ \smallsetminus Z}} \bfPerf^\otimes_{\hZ} . \]
The key step in the proof is \cref{thm:etale_descent}, that allows to check this equivalence locally on $X$.
In this way, we can reduce to the situation where $X$ is affine, and therefore invoke the more classical formal decomposition results of, for example, \cite{Bhatt_algebraization_2014}.
It is also important to remark that when both $X$ and $Z$ are $k$-varieties, this result (combined with \cref{cor:comparison}) implies the main theorem Ben-Bassat and Temkin obtained in \cite{Ben-Bassat_Temkin_Tubular_2013}.

The second formal decomposition result that can be found in this section is the decomposition along a \emph{flag} $(Z_n \hookrightarrow Z_{n-1} \hookrightarrow \cdots Z_1 \hookrightarrow X)$ of closed substacks in $X$ (see Corollaries \ref{cor:iterated_flag_decomposition_stacky} and \ref{cor:iterated_flag_decomposition_categorical} below).
The proof of this result relies on the previous formal glueing decomposition, combined with a careful handling of the iteration.
In particular, we are able to prove these formal decomposition results only because our formal glueing theorem does not require neither $X$ nor $Z$ to be reduced.

\subsection{Formal glueing}\label{formalglueing}
Let $X$ be an Artin stack locally of finite presentation over $k$ and let $Z \hookrightarrow X$ be a closed substack.
Let $V \coloneqq X \smallsetminus Z$ and introduce the functor
\[ \fW_V \coloneqq - \times_X V \colon X\Et \to \St_k . \]
We denote by $\fCoh^\otimes_V$ the composition $\fCoh^\otimes_V \coloneqq \bfCoh^{\otimes,-} \circ \fW_V$ and by $\fPerf^\otimes_V$ the composition $\fPerf^\otimes_V \coloneqq \bfPerf^\otimes \circ \fW_V$.

Coherently with this notation, we let $\fW_X$ denote the inclusion functor:
\[ \fW_X \coloneqq - \times_X X \colon X\Et \to \St_k . \]
Observe that there is a natural transformation
\[ \alpha \colon \hPsi_Z \to \fW_X , \]
and since $\fW_X$ factors through $\Aff_k \hookrightarrow \St_k$, we conclude that $\alpha$ induces a natural transformation
\[ \alpha^{\mathrm{aff}} \colon \hPsi_Z^{\mathrm{aff}} \to \fW_X . \]
Observe also that for every $U \in X\Et$, the induced square
\[ \begin{tikzcd}
	Z_U \arrow[-, double equal sign distance]{r} \arrow[hook]{d} & Z_U \arrow[hook]{d} \\
	\widehat{Z_U} \arrow{r} & U
\end{tikzcd} \]
is a pullback.
This implies that there is an induced commutative square
\[ \begin{tikzcd}
	\hPsi^\circ_Z(U) \arrow{r}{\beta_U} \arrow[hook]{d}{i_U} & \fW_V(U) \arrow[hook]{d} \\
	\hPsi_Z^{\mathrm{aff}}(U) \arrow{r} & \fW_X(U) .
\end{tikzcd} \]
Since the natural transformation $\fW_V \hookrightarrow \fW_X$ is an open immersion, we conclude that the transformations $\beta_U$ assemble into a natural transformation
\[ \beta \colon \hPsi^\circ_Z \to \fW_V . \]
Recall the functors $\hPhi_Z^{\circ, \otimes} = \bfCoh^{\otimes, -} \circ \hPsi^\circ_Z$ and $\hPhi^\otimes_Z = \bfCoh^{\otimes,-} \circ \hPsi_Z^{\mathrm{aff}}$.
Composing the natural transformation $\beta$ (respectively $i \colon \hPsi_Z^\circ \to \hPsi_Z^{\mathrm{aff}}$) with $\bfCoh^{\otimes,-}$, we obtain well defined restriction maps
\[ \beta^* \colon \fCoh^\otimes_V \to \hPhi_Z^{\circ, \otimes} \leftarrow \hPhi^\otimes_Z \colon i^*. \]

\begin{thm}\label{prop:one_step_flag_decomposition_coh}
	Let $X$ be an Artin stack locally of finite type over $k$ and let $Z \hookrightarrow X$ be a closed substack, and $V \coloneqq X \smallsetminus Z$.
	The natural transformations $\beta^* \colon \fCoh^\otimes_V \to \hPhi^{\circ, \otimes}_Z$ and $i^* \colon \hPhi^\otimes_Z \to \hPhi^{\circ, \otimes}_Z$ induce an equivalence in $\Sh_{\Catstmon}(X\Et, \tauet)$:
	\[ \fCoh_X^\otimes \simeq \fCoh^\otimes_V \times_{\hPhi^{\circ, \otimes}_Z} \hPhi^\otimes_Z . \]
\end{thm}

\begin{proof}
	First of all, since the forgetful functor $\Catstmon = \CAlg(\Cat_\infty^{\mathrm{Ex}, \otimes}) \to \Catst$ is conservative and commutes with arbitrary limits, we are immediately reduced to prove the analogous statement for the downgraded functors with values in $\Catst$.
	Now, it is clear that $\fCoh_V$ and $\hPhi_Z$ define sheaves on $X\Et$.
	Furthermore, \cref{thm:etale_descent} shows that also $\hPhi^\circ_Z$ is a sheaf on $X\Et$.
	Since limits in $\Sh_{\Catst}(X\Et, \tauet)$ can be computed in $\PSh_{\Catst}(X\Et)$, we see that the question is local on $X$.
	We can therefore suppose $X$ to be affine.
	In this case, the result is a direct consequence of \cite[Proposition 5.6]{Bhatt_algebraization_2014}.
\end{proof}

A similar result holds for perfect modules.
Indeed, observe that the natural transformation $i^* \colon \hPhi_Z^\otimes \to \hPhi^{\circ,\otimes}_Z$ restricts by definition to a natural transformation
\[ i^* \colon \hrP^\otimes_Z \to \hrP^{\circ,\otimes}_Z . \]
(cf.\ \cref{subsec:perfect_complexes} for the notations used here).
Furthermore, since $\hrP^\circ_Z$ is a full substack of $\hPhi^{\circ, \otimes}_Z$ and perfect modules are stable under base change, we see that also $\beta^* \colon \fCoh^\otimes_V \to \hPhi^{\circ, \otimes}_Z$ restricts to a natural transformation
\[ \beta^* \colon \fPerf^\otimes_V \to \hrP^{\circ, \otimes}_Z . \]

\begin{thm} \label{thm:one_step_flag_decomposition_perf}
	Let $X$ be an Artin stack locally of finite type over $k$ and let $Z \hookrightarrow X$ be a closed substack, and $V \coloneqq X \smallsetminus Z$.
	The natural transformations $\beta^* \colon \fPerf\otimes_V \to \hrP^{\circ,\otimes}_Z$ and $i^* \colon \hrP^{\otimes}_Z \to \hrP^{\circ, \otimes}_Z$ induce an equivalence in $\Sh_{\Catstmon}(X\Et, \tauet)$:
	\[ \fPerf^\otimes_X \simeq \fPerf^\otimes_V \times_{\hrP^{\circ,\otimes}_Z} \hrP^\otimes_Z . \]
\end{thm}

\begin{proof}
	As usual, it is enough to prove the analogous statement for the downgraded functors with values in $\Catst$ instead of $\Catstmon$.
	Using \cref{prop:restricted_etale_descent_Perf} we see that $\hrP^\circ_Z$ is a sheaf.
	We can therefore reduce to the case where $X$ is local, and the assertion is a direct consequence of \cite[Proposition 5.6]{Bhatt_algebraization_2014}.
\end{proof}

\begin{cor}\label{cor:one_step_flag_decomposition_cohperf}
	Let $X$ be an Artin stack locally of finite type over $k$ and let $Z \hookrightarrow X$ be a closed substack.
	Let $V \coloneqq X \smallsetminus Z$ be the open complementary stack.
	Then there are canonical decompositions of stacks
	\[ \bfCoh^{\otimes,-}_X \simeq \bfCoh^{\otimes,-}_V \times_{\bfCoh^{\otimes,-}_{\hZ \smallsetminus Z}} \bfCoh^{\otimes,-}_{\hZ} \]
	and
	\[ \bfPerf^\otimes_X \simeq \bfPerf^\otimes_V \times_{\bfPerf^\otimes_{\hZ \smallsetminus Z}} \bfPerf^\otimes_{\hZ} . \]
	In particular, by taking $k$-points, we get equivalences of symmetric monoidal stable $\infty$-categories
	\[ \Coh^-(X) \simeq \Coh^-(V) \times_{\Coh^-(\hZ \smallsetminus Z)} \Coh^-(\hZ) \]
	and
	\[ \Perf(X) \simeq \Perf(V) \times_{\Perf(\hZ \smallsetminus Z)} \Perf(\hZ) . \]
\end{cor}

\begin{proof}
	As before, it is enough to prove the analogous statements for the downgraded functors with values in $\Catst$.
	Now, recall from Definitions \ref{def:coherentPFN} and \ref{def:perfect_PFN} that the stacks $\bfCoh^-_{\hZ \smallsetminus Z}$ and $\bfPerf_{\hZ \smallsetminus Z}$ are defined as the compositions
	\[ \bfCoh^-_{\hZ \smallsetminus Z} \coloneqq \hPhi^\circ_{Z} \circ u \quad \text{and} \quad \bfPerf_{\hZ \smallsetminus Z} \coloneqq \hrP^\circ_{Z} \circ u , \]
	where $u \colon (\Aff_k, \tauet) \to (X\et^{\Geom}, \tauet))$ is the morphism of sites induced by the canonical map $X \to \Spec(k)$.
	We now remark that
	\begin{gather*}
		\bfCoh^-_X \simeq \fCoh_X \circ u, \quad \bfCoh^-_V \simeq \fCoh_V \circ u , \\
		\bfPerf_X \simeq \fPerf_X \circ u, \quad \bfPerf_V \simeq \fPerf_V \circ u .
	\end{gather*}
	Finally, we claim that
	\[ \bfCoh^-_{\hZ} \simeq \hPhi_Z \circ u \quad \text{and} \quad \bfPerf_{\hZ} \simeq \hrP_Z \circ u . \]
	Indeed, we can write
	\[ \hZ \simeq \colim_{n \in \mathbb N} Z^{(n)} \]
	in $\St_k$, and since colimits are universal in $\St_k$, we obtain that for every $Y \in \Aff_k$, we can identify $Y \times \hZ$ with the formal completion of $Y \times Z$ inside $Y \times X$.
	In particular,
	\[ \bfCoh^-_{\hZ}(Y) = \Coh^-(Y \times \hZ) \simeq \Coh^-(\widehat{Z_{Y \times X}}) \simeq \hPhi_Z(Y \times X) \simeq \hPhi_Z(u(Y)) .  \]
	A completely analogous argument shows that $\bfPerf_{\hZ} \simeq \hrP_Z \circ u$.
	
	At this point, the conclusion follows from the fact that the functor
	\[ u^s \colon \Sh( X\Et^{\Geom}, \tauet ) \to \Sh(\Aff_k, \tauet) = \St_k , \]
	which is given by $F \mapsto F \circ u$, is a right adjoint (cf.\ in virtue of \cite[Lemmas 2.13 and 2.14]{Porta_Yu_Higher_analytic_stacks_2014}).
	In particular, it commutes with fiber products, and therefore we see that the corollary is a direct consequence of Theorems \ref{prop:one_step_flag_decomposition_coh} and \ref{thm:one_step_flag_decomposition_perf}.
\end{proof}

\subsection{Decomposition along a flag}\label{subsection:flagdec}

We are now able to give decomposition theorems for the stacks $\bfCoh_X^{\otimes,-}$ and $\bfPerf^\otimes_X$ \emph{along a non-linear flag} in $X$, and deduce the corresponding decomposition results for the derived $\infty$-categories $\Coh^{-}(X)$ and $\Perf(X)$. \\
We start by recalling the notion of non-linear flag on $X$.

\begin{defin}\label{definitionofflag}
	Let $X$ be an Artin stack locally of finite type over $k$.
	A \emph{flag of length $n$} on $X$ is a sequence $\cZ \coloneqq (Z_0, Z_1, \ldots, Z_n)$ of closed substacks of $X$ such that:
	\begin{enumerate}
		\item $Z_0 = X$;
		\item the inclusion $Z_{i+1} \hookrightarrow X$ factors through $Z_i \hookrightarrow X$;
	\end{enumerate}
\end{defin}

Let $X$ be an Artin stack locally of finite type over $k$.
Fix a flag $\cZ = (Z_0, Z_1, \ldots, Z_n)$ on $X$.
For $i$ in $1,\dots,n$, we define stacks
\[ \hPhi^{\circ,\otimes}_i,\,\hrP^{\circ,\otimes}_i \colon X\Et\op \to \Catstmon \]
as follows:
\begin{align*}
&i=1 & \hPhi^{\circ,\otimes}_1 &\coloneqq \hPhi^{\circ,\otimes}_{Z_1/X}, & \hrP^{\circ,\otimes}_1 &\coloneqq \hrP^{\circ,\otimes}_{Z_1/X} \\
&i\geq 2 & \hPhi^{\circ,\otimes}_{i} &\coloneqq \varprojlim_m \hPhi^{\circ,\otimes}_{Z_{i}/Z_{i-1}^{(m)}}, & \hrP^{\circ,\otimes}_{i} &\coloneqq \varprojlim_m \hrP^{\circ,\otimes}_{Z_{i}/Z_{i-1}^{(m)}}
\end{align*}
where $Z_{i-1}^{(m)}$ is the $m^\mathrm{th}$ infinitesimal neighborhood of $Z_{i-1}$ in $X$. Note that the definition for $i \geq 2$ also holds for $i = 1$ as $Z_0^{(m)} = X$ for all $m$.
We also define $\hPhi_i^\otimes \,\hrP_i^\otimes \colon X\Et\op \to \Catst$ by
\begin{align*}
&i=0 & \hPhi_0^\otimes &\coloneqq \fCoh^\otimes_X, & \hrP^\otimes_0 &\coloneqq \fPerf^\otimes_X \\
&i\geq 1 & \hPhi^\otimes_{i} &\coloneqq \varprojlim_m \hPhi^\otimes_{Z_{i}/Z_{i-1}^{(m)}}, & \hrP^\otimes_{i} &\coloneqq \varprojlim_m \hrP^\otimes_{Z_{i}/Z_{i-1}^{(m)}}.
\end{align*}
Note that for any $i \geq 1$, we have
\[\hPhi^\otimes_{i} \simeq \varprojlim_{m,n} \hPhi^\otimes_{Z_{i}^{(n)} / Z_{i-1}^{(m)}} \simeq \hPhi^\otimes_{Z_{i} / X} \hspace{1cm} \text{and} \hspace{1cm} \hrP^\otimes_{i} \simeq \varprojlim_{m,n} \hrP^\otimes_{Z_{i}^{(n)} / Z_{i-1}^{(m)}} \simeq \hrP^\otimes_{Z_{i} / X}.
\]
We also have $\hPhi^\otimes_{0} \simeq \fCoh^\otimes_X \simeq \hPhi^\otimes_{Z_0/X}$ and $\hrP^\otimes_{0} \simeq \fPerf^\otimes_X \simeq \hrP^\otimes_{Z_0/X}$.
Finally, we let $V_{i+1}^{(m)}$ be the open complement of $Z_{i+1}$ inside $Z_i^{(m)}$, and we set
\[ \fCoh^\otimes_{V,i+1} \coloneqq \varprojlim_m \fCoh^\otimes_{V_{i+1}^{(m)}} \hspace{1cm} \text{and} \hspace{1cm} \fPerf^\otimes_{V,i+1} \coloneqq \varprojlim_m \fPerf^\otimes_{V_{i+1}^{(m)}}. \]
Note that the sequence $\{V_{i+1}^{(m)}\}_m$ defines a formal open subscheme $\fV_{i+1}$ of
\[ \widehat{Z_i} = \colim_m Z_i^{(m)} . \]
Note that $\fV_1$ simply coincides with $V = X \smallsetminus Z$.
With these notations, we have $\fCoh^\otimes_{V,i+1}(X) = \Coh^-(\fV_{i+1})$ and $\fPerf^\otimes_{V,i+1}(X) = \Perf^-(\fV_{i+1})$.

\begin{prop} \label{prop:flag_decomposition_induction_step}
	Let $X$ be an Artin stack locally of finite presentation over $k$ and let $\cZ = (Z_0, Z_1, \ldots, Z_n)$ be a flag on $X$.
	Then for every $0 \le i \le n-1$ there are natural equivalences in $\Sh_{\Catst}(X\Et, \tauet)$:
	\[ \hPhi^\otimes_i \simeq \fCoh^\otimes_{V,i+1} \times_{\hPhi^{\circ,\otimes}_{i+1}} \hPhi^\otimes_{i+1} \]
	and
	\[ \hrP^\otimes_i \simeq \fPerf^\otimes_{V, i+1} \times_{\hrP^{\circ, \otimes}_{i+1}} \hrP^{\otimes}_{i+1} . \]
\end{prop}

\begin{proof}
	Unravelling the definitions and using the fact that limits commute with limits, we see that this proposition follows by applying iteratively Theorems \ref{prop:one_step_flag_decomposition_coh} and \ref{thm:one_step_flag_decomposition_perf}.
\end{proof}

\begin{cor} \label{cor:iterated_flag_decomposition_stacky}
	Let $X$ be an Artin stack locally of finite type over $k$ and let $\cZ \coloneqq (Z_0, Z_1, \ldots, Z_n)$ be a flag on $X$.
	Then for every $0 \le i \le n-1$, there are natural equivalences in $\Sh_{\Catst}(\Aff_k, \tauet)$:
	\[ \bfCoh^{\otimes,-}_{\widehat{Z_i}} \simeq \bfCoh^{\otimes,-}_{\fV_{i+1}} \times_{\bfCoh^{\otimes,-}_{\widehat{Z_{i+1}} \smallsetminus Z_{i+1}}} \bfCoh^{\otimes,-}_{\widehat{Z_{i+1}}} \]
	and
	\[ \bfPerf^\otimes_{\widehat{Z_i}} \simeq \bfPerf^\otimes_{\fV_{i+1}} \times_{\bfPerf^\otimes_{\widehat{Z_{i+1}} \smallsetminus Z_{i+1}}} \bfPerf^\otimes_{\widehat{Z_{i+1}}} . \]
\end{cor}

By taking $k$-points of the stacks in the previous corollary, we get the following decomposition for derived $\infty$-categories of almost perfect and perfect modules on $X$:

\begin{cor} \label{cor:iterated_flag_decomposition_categorical}
	Let $X$ be an Artin stack locally of finite type over $k$ and let $\cZ \coloneqq (Z_0, Z_1, \ldots, Z_n)$ be a flag on $X$.
	Then for every $0 \le i \le n-1$, there are equivalences of symmetric monoidal stable $\infty$-categories
	\[ \Coh^-(\widehat{Z_i}) \simeq \Coh^-(\fV_{i+1}) \times_{\Coh^-(\widehat{Z_{i+1}} \smallsetminus Z_{i+1})} \Coh^-(\widehat{Z_{i+1}}) \]
	and
	\[ \Perf(\widehat{Z_i}) \simeq \Perf(\fV_{i+1}) \times_{\Perf(\widehat{Z_{i+1}} \smallsetminus Z_{i+1})} \Perf(\widehat{Z_{i+1}}) . \]
\end{cor}

\begin{warning}
	The notation we chose is slightly ambiguous: note that while the category $\Coh^-(\widehat{Z_{i+1}})$ only depends on the embedding of $Z_{i+1}$ in $X$, the categories $\Coh^-(\fV_{i+1})$ and $\Coh^-(\widehat{Z_{i+1}} \smallsetminus Z_{i+1})$ depend on the embedding $Z_{i+1} \hookrightarrow Z_i$.
	Therefore, the right hand side in the above decomposition depends on the given flag and not only on the closed subscheme $Z_{i+1}$.
\end{warning}

Iterating the flag decomposition provided by Corollary \ref{cor:iterated_flag_decomposition_categorical}, we get a \emph{full flag decomposition}:

\begin{cor} \label{cor:full_flag_decomposition}
	Let $X$ be an Artin stack locally of finite type over $k$ and let $\cZ \coloneqq (Z_0, Z_1, \ldots, Z_n)$ be a flag on $X$.
	Then there are equivalences in $\Sh_{\Catst}(\Aff_k, \tauet)$:
	\[ \bfCoh^{\otimes,-}_X \simeq \bfCoh^{\otimes,-}_{\fV_1} \times_{\bfCoh^{\otimes,-}_{\widehat{Z_1} \smallsetminus Z_1}} \left( \bfCoh^{\otimes,-}_{\fV_2} \times_{\bfCoh^{\otimes,-}_{\widehat{Z_2} \smallsetminus Z_2}} \left( \cdots \left( \bfCoh^{\otimes,-}_{\fV_n} \times_{\bfCoh^{\otimes,-}_{\widehat{Z_n} \smallsetminus Z_n}} \bfCoh^{\otimes,-}_{\widehat{Z_n}} \right) \cdots \right) \right) \]
	and
	\[ \bfPerf^\otimes_X \simeq \bfPerf^\otimes_{\fV_1} \times_{\bfPerf^\otimes_{\widehat{Z_1} \smallsetminus Z_1}} \left( \bfPerf^\otimes_{\fV_2} \times_{\bfPerf^\otimes_{\widehat{Z_2} \smallsetminus Z_2}} \left( \cdots \left( \bfPerf^\otimes_{\fV_n} \times_{\bfPerf^\otimes_{\widehat{Z_n} \smallsetminus Z_n}} \bfPerf^\otimes_{\widehat{Z_n}} \right) \cdots \right) \right) . \]
	Furthermore, taking $k$-points in the above equivalences, we obtain the following equivalences of symmetric monoidal stable $\infty$-categories:
	\[ \Coh^-(X) \simeq \Coh^-(\fV_1) \times_{\Coh^-(\widehat{Z_1} \smallsetminus Z_1)} \left( \Coh^-(\fV_2) \times_{\Coh^-(\widehat{Z_2} \smallsetminus Z_2)} \left( \cdots \left( \Coh^-(\fV_n) \times_{\Coh^-(\widehat{Z_n} \smallsetminus Z_n)} \Coh^-(\widehat{Z_n}) \right) \cdots \right) \right) \]
	and
	\[ \Perf(X) \simeq \Perf(\fV_1) \times_{\Perf(\widehat{Z_1} \smallsetminus Z_1)} \left( \Perf(\fV_2) \times_{\Perf(\widehat{Z_2} \smallsetminus Z_2)} \left( \cdots \left( \Perf(\fV_n) \times_{\Perf(\widehat{Z_n} \smallsetminus Z_n)} \Perf(\widehat{Z_n}) \right) \cdots \right) \right) \]
\end{cor}

\begin{warning}
	Since the stacks $\bfCoh^{\otimes,-}(\widehat{Z_i} \smallsetminus Z_i)$ depends on the given embedding $Z_i \hookrightarrow Z_{i-1}$, we see that we cannot remove the parentheses in the equivalence of \cref{cor:full_flag_decomposition}.
	A similar caution has to be applied to all the other decompositions proved above.
\end{warning}

In the case $X=S$ is a surface, we get a particularly readable expression of the flag decomposition. There is a stacky version (like in Corollary \ref{cor:iterated_flag_decomposition_stacky}), that we omit, from which we get the following:

\begin{cor} \label{cor:full_flag_decomposition_surface}
	Let $S$ be a surface and let $(S, C, x)$ be a flag in $S$.
	Then there are natural equivalences of symmetric monoidal stable $\infty$-categories
	\[ \Coh^-(S) \simeq \Coh^-(S \smallsetminus C) \times_{\Coh^-(\widehat{C} \smallsetminus C)} \left( \Coh^-(\widehat{C} \smallsetminus x) \times_{\Coh^-(\hat{x} \smallsetminus x)} \Coh^-(\hat{x}) \right) \]
	and
	\[ \Perf(S) \simeq \Perf(S \smallsetminus C) \times_{\Perf(\widehat{C} \smallsetminus C)} \left( \Perf(\widehat{C} \smallsetminus x) \times_{\Perf(\hat{x} \smallsetminus x)} \Perf(\hat{x}) \right) \]
\end{cor}

\section{Formal glueing and flag decomposition for $G$-bundles} \label{section:gbundles}

Let $X$ be an Artin stack locally of finite presentation over $k$, $Z\hookrightarrow X$ a closed substack. In this section, we will prove analogues of the formal glueing decomposition (Section \ref{formalglueing}) and of the flag decomposition (Section \ref{subsection:flagdec}) for the stack of $G$-bundles on $X$, where $G$ is an affine algebraic $k$-group scheme (see \cref{thm:gbundles}). These results are consequences of the corresponding results for perfect modules, once we prove  the key property that the classifying stack $\rB G$ is $\Perf$-local (see \cref{lem:bgperflocal}).

\begin{defin}
Let $G$ denote a smooth, affine $k$-algebraic group scheme.
Let us recall the functor $\hPsi^\circ_Z \colon X\Et \to \St_k$.
Let us denote by $\Bunhat^\circ_G$ the composite of $\hPsi^\circ_Z$ with the functor $\Bun_G$ mapping a stack $Y$ to its groupoid of $G$-bundles:
\[
\Bunhat^\circ_G = \Bun_G \circ \hPsi^\circ_Z \colon X\Et \to \St_k \to \cS^\mathrm{op}
\]
Informally, the functor $\Bunhat^\circ_G$ maps a affine scheme $U = \Spec A$ equipped with an étale map $f \colon U \to X$ to the groupoid of $G$-bundles on the scheme $\Spec(\hat A_I) \smallsetminus Z_U$, where $I$ is the ideal defining $Z_U$ in $U$.
We define similarly 
\begin{align*}
\Bunhat_G = \Bun_G \circ \hPsi_Z &\colon U \mapsto \Bun_G(\Spec(\hat A_I)) \\
\bfBun_G^\circ = \Bun_G \circ \fW_V &\colon U \mapsto \Bun_G(U \smallsetminus Z_U)  \\
\bfBun_G^X = \Bun_G \circ \fW_X &\colon U \mapsto \Bun_G(U) .
\end{align*}
\end{defin}
\begin{rem}
The functor $\bfBun_G^X$ encodes the stack of $G$-bundles on $X$, while the functor $\bfBun_G^\circ$ encodes the stack of $G$-bundles on $X \smallsetminus Z$ (see \cref{def:stackofbundles} below).
\end{rem}

\begin{thm}\label{thm:gbundles}
The square of natural transformations
\[
\xymatrix{
\bfBun_G^X \ar[r] \ar[d] & \bfBun_G^\circ \ar[d] \\
\Bunhat_G \ar[r] & \Bunhat^\circ_G
}
\]
is cartesian.
\end{thm}

\begin{rem}
We can state the above theorem informally, by saying that for any closed subscheme $Z \subset X$, a $G$-bundle on $X$ amounts to a $G$-bundle on $X \smallsetminus Z$ and a $G$-bundle on the formal neighbourhood of $Z$ in $X$, together with some glueing data on the punctured formal neighbourhood.
Let us emphasize that no codimension requirement is made on $Z$.
\end{rem}

The proof of \cref{thm:gbundles} is based on a key lemma (see \cref{lem:bgperflocal} below). Let us record another important consequence of that lemma:

\begin{prop}\label{prop:bunGpunctdescent}
The functors $\Bunhat_G$ and $\Bunhat^\circ_G \colon X\Et \to \cS^\mathrm{op}$ is an hypercomplete sheaf for the \'etale topology.
\end{prop}

\begin{lem}\label{lem:bgperflocal}
Let $f \colon Y_1 \to Y_2$ be a map of stacks.
Assume that the pullback functor $f^* \colon \Perf(Y_2) \to \Perf(Y_1)$ is an equivalence, then the pullback functor
\[
f^* \colon \Bun_G(Y_2) \to \Bun_G(Y_1)
\]
is an equivalence.
\end{lem}

\begin{proof}
To prove this lemma, we will use Tannaka duality as proven for instance in \cite[Theorem 3.4.2]{Lurie_Tannaka_duality}. It gives us a description of the groupoid of $G$-bundles in terms of monoidal functors. Namely, for any stack $Y$, the map
\[
\alpha_Y \colon \Bun_G(Y) \simeq \Map(Y, \rB G) \to \Fun^{\otimes}(\QCoh(\rB G),\QCoh(Y))
\]
is fully faithful. Note also that we have a canonical embedding
\[
\Fun^{\otimes}(\Perf(\rB G),\Perf(Y)) \to \Fun^{\otimes}(\QCoh(\rB G),\QCoh(Y))
\]
given by the inclusion $\Perf(Y) \to \QCoh(Y)$ and the left Kan extension functor along $\Perf(\rB G) \to \QCoh(\rB G)$.
The functor $\alpha_Y$ actually factors through $\Fun^{\otimes}(\Perf(\rB G),\Perf(Y))$.
Let now $f \colon Y_1 \to Y_2$ be such that $f^* \colon \Perf(Y_2) \to \Perf(Y_1)$ is an equivalence.
We get a commutative diagram
\[
\xymatrix{
\Bun_G(Y_2) \ar[r] \ar[d]^{F_{\Bun}} & \Fun^{\otimes}(\Perf(\rB G),\Perf(Y_2)) \ar[r] \ar[d]^\simeq_{F_\Perf} & \Fun^{\otimes}(\QCoh(\rB G),\QCoh(Y_2)) \ar[d]^{F_\QCoh} \\
\Bun_G(Y_1) \ar[r] & \Fun^{\otimes}(\Perf(\rB G),\Perf(Y_1)) \ar[r] & \Fun^{\otimes}(\QCoh(\rB G),\QCoh(Y_1))
}
\]
From what precedes, we get that $F_{\Bun}$ is fully faithful. 
Recall the description of the essential images of the functors $\alpha_{Y_i}$ given in \cite[Theorem 3.4.2]{Lurie_Tannaka_duality}.
A monoidal functor $\beta \colon \Perf(\rB G) \to \Perf(Y_1)$ then lies in the essential image of $\alpha_{Y_1}$ if and only if its preimage $F_\Perf^{-1}(\beta)$ lies in the essential image of $\alpha_{Y_2}$.
It follows that $F_{\Bun}$ is an equivalence.
\end{proof}

\begin{rem}
In the above lemma, one could replace the classifying stack of $G$-bundles $\rB G$ by any geometric stack with affine diagonal (so that \cite[Theorem 3.4.2]{Lurie_Tannaka_duality} applies).
Stacks satisfying the conclusion of \cref{lem:bgperflocal} are usually called \emph{$\Perf$-local stacks}.
\end{rem}

\begin{proof}[Proof (of {\cref{thm:gbundles}})]
Using \cref{prop:flag_decomposition_induction_step}, we apply \cref{lem:bgperflocal} to the natural transformation
\[
\fW_V \amalg_{\hPsi^\circ_Z} \hPsi_Z \to \fW_X.
\]
\end{proof}

\begin{proof}[Proof (of \cref{prop:bunGpunctdescent})]
Let $U_\bullet \to U$ be an hypercovering in $X\Et$. Using \cref{prop:restricted_etale_descent_Perf}, we see that the map
\[
\colim \hPsi^\circ_Z(U_\bullet) \to \hPsi^\circ_Z(U)
\]
is sent to an equivalence by the functor $\Perf$. The same holds with $\hPsi_Z$ instead of $\hPsi^\circ_Z$. The result then follows from \cref{lem:bgperflocal}.
\end{proof}

\begin{defin}\label{def:stackofbundles}
Recall from right before \cref{def:coherentPFN} the morphism of sites $u \colon (\mathrm{Aff}_k, \tauet) \to (X\Et^{\Geom}, \tauet)$ mapping an affine scheme $T$ over $k$ to $X \times T$. For any closed embedding $Z \subset X$ with complement $V = X \smallsetminus Z$, we define the stacks $\mathrm{Aff}_k^\mathrm{op} \to \cS$
\begin{align*}
&\bfBun_G(X) = \bfBun_G^X \circ u \,\, \text{ of }G\text{-bundles on }X \\
&\bfBun_G(V) = \bfBun_G^\circ \circ u \,\, \text{ of }G\text{-bundles on } \fV = X \smallsetminus Z \\
&\bfBun_G(\hat{Z}) = \Bunhat_G \circ u \,\, \text{ of }G\text{-bundles on the formal neighbourhood }\hat{Z} \\
&\bfBun_G(\hat{Z} \smallsetminus Z) = \Bunhat^\circ_G \circ u \,\, \text{ of }G\text{-bundles on the punctured formal neighbourhood}
\end{align*}
As in the case of (almost) perfect complexes, we extend those definitions to the case where $X$ is a formal ind-stack (i.e.\ a filtered colimit of Artin stacks locally of finite presentation whose reduced part is constant) -- see shortly after \cref{definitionofflag}.
\end{defin}

\begin{rem}
With the above definition, if $Z \subset X$ is of codimension $1$ then we get a definition of the affine Grassmanian along $Z$ over $X$ as the fiber
\[
\bfBun_G(\hat{Z} \smallsetminus Z) \times_{\bfBun_G(\hat{Z})} *.
\]
In the case where $Z \subset X$ is the inclusion $\{0\} \subset \mathbb{A}^1$, this fiber is equivalent to the affine Grassmanian $G(\!(t)\!)/G[\![t]\!]$.
\end{rem}

\begin{cor}\label{flagbung}
Let $X$ be an Artin stack locally of finite type over $k$ and let $\cZ \coloneqq (Z_0, Z_1, \ldots, Z_n)$ be a flag on $X$ (cf.\ \cref{definitionofflag}). The stack $\bfBun_G(X)$ is equivalent to
\[
\bfBun_G(\fV_1) \times_{\bfBun_G(\widehat{Z_1} \smallsetminus Z_1)} \left( \bfBun_G(\fV_2) \times_{\bfBun_G(\widehat{Z_2} \smallsetminus Z_2)} \left( \cdots \left( \bfBun_G(\fV_n) \times_{\bfBun_G(\widehat{Z_n} \smallsetminus Z_n)} \bfBun_G(\widehat{Z_n}) \right) \cdots \right) \right) .
\]
where, as before $\fV_{i+1}$, is the formal ind-stack $\widehat{Z_{i}} \smallsetminus Z_{i+1}$.
In the particular case where $X$ is a surface $S$ and the flag consists of a point $x$ in a curve $C$ in $S$, we get:
\[
\bfBun_G(S) \simeq \bfBun_G(S \smallsetminus C) \times_{\bfBun_G(\widehat{C} \smallsetminus C)} \left( \bfBun_G(\widehat{C} \smallsetminus x) \times_{\bfBun_G(\hat{x} \smallsetminus x)} \bfBun_G(\hat{x}) \right) 
\].
\end{cor}

\section{Generalization to derived stacks} \label{sec:DAG}

In this section, we will generalise all the results of the previous sections to the context of derived algebraic geometry. Note that the content of this section is stronger than what precedes, even in the case of $X$ a variety.
Indeed \cref{cor:one_step_flag_decomposition_cohperf_derived} and \cref{thm:flag_bung_derived} establish that the \emph{derived} structures of the stacks of (almost) perfect complexes and of $G$-bundles can be formally glued as well.
For a comprehensive review of derived algebraic geometry, we refer the reader to \cite{Toen_Derived_2014}.

Let us nevertheless recall a few definitions, and fix our notations.

\begin{defin}
We will denote by $\mathbf{sCAlg}_k$ the $(\infty,1)$-category of simplicial commutative $k$-algebras up to homotopy equivalence.

A derived prestack is then a functor $\mathbf{sCAlg}_k \to \cS$ with values in the category of $\infty$-groupoids. A derived stack is a derived prestack satisfying the étale hyperdescent condition. We will denote by $\dSt_k \coloneqq \St(\mathrm{dAff}_k, \tauet)$ the $\infty$-category of derived stacks.

For any $A \in \mathbf{sCAlg}_k$, we will denote by $\pi_i A$ its $i^\mathrm{th}$ homotopy group. Recall that $\pi_0 A$ is endowed with a commutative $k$-algebra structure, and that for any $i$, the abelian group $\pi_i A$ is endowed with a structure of $\pi_0 A$-module.

For any derived stack $\mathcal X$, we will denote by $\trunc \mathcal X$ its restriction to (discrete) commutative $k$-algebras. The functor $\trunc \mathcal X$ is then a (non-derived) stack.
\end{defin}

\begin{defin}
Let $\mathcal X$ be a derived stack. We define its de Rham stack $\mathcal X_\mathrm{dR}$ as the derived prestack whose $A$-points are the reduced $A$-points of $\mathcal X$:
\[
\mathcal X_\mathrm{dR}(A) = \mathcal X(A_\mathrm{red})
\]
where $A_\mathrm{red}:= (\pi_0 (A))_\mathrm{red}$. It is endowed with a canonical morphism $\mathcal X \to \mathcal X_\mathrm{dR}$.

Let now $\mathcal Z \to \mathcal X$ be a map of derived stacks. The (\emph{derived}) \emph{formal completion} of $\mathcal Z$ in $\mathcal X$ (or of $\mathcal X$ along $\mathcal Z$) is defined as the pullback
\[
\widehat{\mathcal X}_\mathcal Z = \mathcal X \times^{\mathrm{d}}_{\mathcal X_\mathrm{dR}} \mathcal Z_\mathrm{dR} \in \dSt_k
\]
\end{defin}

\begin{rem}
Let us fix $\mathcal X$. The formal completion of $\mathcal X$ along $\mathcal Z$ only depends on the reduced part of $\mathcal Z$. In particular, it does not depend on the derived structure on $\mathcal Z$.
\end{rem}

\begin{thm} \label{thm:dag_completion_truncation}
Let $\mathcal Z \to \mathcal X$ be a map of derived stacks. Assume that $\mathcal X$ is a derived scheme and $\mathcal Z$ is a closed subscheme. Finally, let us assume that the ideal defining $\trunc(\mathcal Z)$ in $\trunc(\mathcal X)$ is locally finitely generated. The following assertions hold.
\begin{enumerate}
\item The formal completion $\widehat{\mathcal X}_\mathcal Z$ is a stack and is representable by a derived ind-scheme.\label{formalcompletion:indscheme}
\item For any map $\mathcal Y \to \mathcal X$, the formal completion of $\mathcal Y$ along $\mathcal Y \times^{\mathrm{d}}_{\mathcal X} \mathcal Z$ is canonically equivalent to the fiber product $\mathcal Y \times^{\mathrm{d}}_{\mathcal X} \widehat{\mathcal X}_{\mathcal Z}$.\label{formalcompletion:basechange}
\item The truncation $\trunc(\widehat{\mathcal X}_\mathcal Z)$ is canonically isomorphic to the formal completion of $\trunc(\mathcal X)$ along $\trunc(\mathcal Z)$. \label{formalcompletion:truncation}
\end{enumerate}
\end{thm}

\begin{proof}
Assertion (\ref{formalcompletion:indscheme}) is \cite[Proposition 6.3.1]{Gaitsgory-Rozenblyum}.
Assertion (\ref{formalcompletion:basechange}) follows from the fact that the functor $\mathcal X \mapsto \mathcal X_\mathrm{dR}$ preserves fiber products.
Let us finally prove assertion (\ref{formalcompletion:truncation}). Using assertion (\ref{formalcompletion:basechange}), we can assume that $\mathcal X$ is affine. Let $A$ denote the algebra of functions on $\trunc(\mathcal X)$ and let $I \subset A$ be an ideal defining $\trunc(\mathcal Z)$ in $\trunc(\mathcal X)$. We can assume that $I$ is finitely generated.
The truncation functor $\trunc$ preserves fiber products. It follows that a $B$-point of $\trunc(\widehat{\mathcal X}_\mathcal Z)$ is a map $f \colon A \to B$ such that $f(I)$ is nilpotent in $B$. The functor $\trunc(\widehat{\mathcal X}_\mathcal Z)$ is therefore represented by the ind-scheme $\colim \Spec(A/I^p)$.
\end{proof}

\begin{defin}\label{defin:Ap}
Let $A$ be any almost finitely presented simplicial commutative $k$-algebra and let $I$ be a finitely generated ideal in $\pi_0 A$.
Let us chose vertices $f_1,\dots,f_n$ in $A$ whose images in $\pi_0 A$ generates $I$.
Let us fix $p \in \mathbb N$. The elements $f_i^p$ determine a morphism $f^p \colon k[t_1,\dots,t_n] \to A$ mapping $t_i$ to $f_i^p$. Let us denote by $A_p$ the derived tensor product $A \otimes^\mathbb{L}_{k[t_1,\dots,t_n]} k$, where the map $k[t_1,\dots,t_n] \to k$ maps the variables to $0$.\\

For any $p$, we have $A_p \otimes_A^{\mathbb L} A_{p+1} \simeq A_p \otimes \Sym_k(k^n[1])$. The canonical augmentation map $\Sym_k(k^n[1]) \to k$ induces a map $A_p \otimes_A^{\mathbb L} A_{p+1} \to A_p$. Using the canonical morphism $A_{p+1} \to A_p \otimes_A^{\mathbb L} A_{p+1}$, we get a canonical map $A_{p+1} \to A_p$.
\end{defin}

\begin{prop}[Gaitsgory-Rozenblyum, see {\cite[Proposition 6.7.4]{Gaitsgory-Rozenblyum}}]\label{prop:gaitsroz}
The derived formal completion of $\Spec(A)$ along $\Spec(\pi_0 A/I)$ is equivalent to the derived ind-scheme $\colim \Spec(A_p)$.
\end{prop}

As we did in the non-derived case, we will then consider the affinisation $\Spec(\varprojlim A_p)$ of the formal neighbourhood $\colim \Spec(A_p)$. It begins with establishing a derived version of the compatibily between flatness and ring completion.
The simplicial algebra $\varprojlim A_p$ will play the role, in this derived context, of the ring completion $\widehat A$ along the ideal $I$.

\begin{defin}
Let $f \colon A \to B$ be a map of simplicial commutative algebras. We will say
\begin{itemize}
\item that $f$ is strong if for any $i$, the induced map $\pi_0 B \otimes_{\pi_0 A} \pi_i A \to \pi_i B$ is an isomorphism.
\item that $f$ is flat if it is strong and if the induced map of algebras $\pi_0 A \to \pi_0 B$ is flat in the usual sense.
\end{itemize}
\end{defin}

\begin{lem}(Lurie)\label{lem:dagflatcompletion}
The induced map $A \to \varprojlim_p A_p$ is flat. Moreover, there is a canonical isomorphism $\pi_0(\varprojlim A_p) \simeq \varprojlim \pi_0(A_p) \simeq \widehat{\pi_0 A}_I$.
\end{lem}
\begin{proof}
The first statement is \cite[Corollary 3.2.9]{DAG-XII}. For the second one, we use \cite[Lemma 5.2.19]{DAG-XII} to see that the diagram $p \mapsto \pi_1(A_p)$ satisfies the Mittag-Leffler condition. In particular $\varprojlim^1 \pi_1(A_p)$ vanishes and the exact sequence 
\[
\textstyle \varprojlim^1 \pi_1(A_p) \to \pi_0(\varprojlim A_p) \to \varprojlim \pi_0(A_p) \to 0
\]
implies the result.
\end{proof}

The following Lemma is a derived version of our \cref{lem:etale_sent_to_flat}. We will need it to establish our descent results in the context of derived algebraic geometry.

\begin{lem}
Let $A \to B$ be a flat map of almost finitely presented simplicial commutative algebras.
The induced map $\varprojlim A_p \to \varprojlim B_p$ is flat.
\end{lem}
\begin{proof}
We have a commutative diagram
\[
\xymatrix{
A \ar[d] \ar[r] & \varprojlim A_p \ar[d] \\ B \ar[r] & \varprojlim B_p
}
\]
From the assumption and \cref{lem:dagflatcompletion}, we get that the maps $A \to \varprojlim A_p$ and $A \to \varprojlim B_p$ are strong. It follows that the map $\varprojlim A_p \to \varprojlim B_p$ is also strong.
The flatness of $\pi_0(\varprojlim A_p) \to \pi_0(\varprojlim B_p)$ follows from \cref{lem:dagflatcompletion} and \cref{lem:etale_sent_to_flat}.
\end{proof}

At this point we can run again the same construction performed in \cref{subsec:punctured_formal_neighbourhood}, but this time in the derived setting.
Since the construction is basically unchanged, we will skip some of the details given there.
To start with, we fix a derived Artin stack $X$, together with a derived closed substack $Z \hookrightarrow X$.
We let $X\Et$ denote the big derived affine \'etale site of $X$.
In other words, $X\Et$ coincides with the full subcategory of $(\dSt_k)_{/X}$ spanned by those map $Y \to X$ where $Y$ is a derived affine scheme.
The Grothendieck topology on $X\Et$ coincides with the usual derived \'etale topology.
As in the underived case, given $U \in X\Et$, we denote by $Z_U$ the base change $U \times^{\mathrm{d}}_X Z$ and by $\widehat{Z_U}$ the derived ind-scheme representing the formal completion of $U$ along $Z_U$.
We can assemble the assignment $U \mapsto \widehat{Z_U}$ into an $\infty$-functor
\[ \rR \hPsi_{Z/X} \colon X\Et \to \dSt_k . \]
Composing with the affinization functor produces
\[ \rR \hPsi^{\mathrm{aff}}_{Z/X} \coloneqq \Spec \circ \Gamma \circ \hPsi_{Z/X} \colon X\Et \to \dSt_k . \]
If $U = \Spec(A)$, with $A$ a derived ring almost of finite presentation over $k$, then we get, using \cref{prop:gaitsroz}
\[ \rR \hPsi^{\mathrm{aff}}_{Z/X}(U) = \Spec(\varprojlim A_p) \in \mathrm{dAff}_k , \]
where $A_p$ are the derived rings from \cref{defin:Ap}.

Using \cref{lem:dagflatcompletion}, we see that
\[ \trunc \circ \rR \hPsi^\mathrm{aff}_{Z/X} = \hPsi^{\mathrm{aff}}_{\trunc(Z) / \trunc(X)} . \]
For any derived scheme $Y$, the truncation functor yields a canonical equivalence of small Zariski sites
\[ Y_{\mathrm{zar}} \simeq \trunc(Y)_{\mathrm{zar}} . \]
For every $U \in X\Et$, we can use this equivalence to canonically lift $\hPsi^\circ_{\trunc(Z) / \trunc(X)}(\trunc(U))$ to a derived open subscheme
\[ \rR \hPsi^\circ_{Z / X}(U) \hookrightarrow \rR \hPsi_{Z / X}(U) . \]
Since this lift is canonical and functorial in $U$, we obtain a well defined $\infty$-functor
\[ \rR \hPsi^\circ_{Z / X} \colon X\Et \to \dSt_k . \]

On the other hand, we have the functors 
\[
\rR \bfPerf^\otimes ,\, \rR \bfCoh^{\otimes,-} \colon \dSt_k\op \to \Catstmon
\]
defined, as in the underived case, by extending to (derived) stacks the functors defined on (simplicial) algebras, mapping $A$ to the categories of perfect and almost perfect stable modules (see for instance \cite{HAG-II}).
Finally, we define
\begin{gather}
	\rR \hPhi_{Z / X}^\otimes \coloneqq \rR \bfCoh^{\otimes,-} \circ \rR \hPsi_{Z / X}, \quad \rR \hPhi^{\circ, \otimes}_{Z / X} \coloneqq \rR \bfCoh^{\otimes,-} \circ \rR \hPsi^\circ_{Z / X} \\
	\rR \hrP_{Z / X}^\otimes \coloneqq \rR \bfPerf^\otimes \circ \rR \hPsi_{Z / X}, \quad \rR \hrP^{\circ, \otimes}_{Z / X} \coloneqq \rR \bfPerf^\otimes \circ \rR \hPsi^\circ_{Z / X} .
\end{gather}

\begin{thm} \label{thm:etale_descent_derived}
	The functors $\rR \hPhi^{\circ,\otimes}_{Z / X} \colon X\Et\op \to \Catstmon$ and $\rR \hrP^{\circ, \otimes}_{Z / X} \colon X\Et\op \to \Catstmon$ are hypercomplete sheaves for $\tauet$.
\end{thm}

\begin{proof}
	Since the forgetful functor $\Catstmon = \CAlg(\Cat_\infty^{\mathrm{Ex}, \times}) \to \Catst$ is conservative and commutes with limits, we see that it is enough to prove the statement for the downgraded functors
	\[ \rR \hPhi^\circ_{Z / X}, \, \rR \hrP^\circ_{Z/X} \colon X\Et\op \to \Catst . \]
	Furthermore, reasoning as in \cref{prop:restricted_etale_descent_Perf}, we can limit ourselves to prove the theorem for $\rR \hPhi^\circ_{Z / X}$.
	Fix an \'etale cover $\{U_i \to U\}$ in $X\Et$.
	Without loss of generality, we may assume $X$ to be affine, and $U = X$.
	Let $U^0 \coloneqq \coprod U_i$ and let $U^\bullet$ be the \v{C}ech nerve of the total morphism $p \colon U^0 \to X$.
	We have to prove that the canonical functor
	\[ p^{\bullet *} \colon \Coh^-( \rR \hPsi^\circ_{Z / X}(X) ) \to \varprojlim \Coh^-( \rR \hPsi^\circ_{Z / X}(U^\bullet) ) \]
	is an equivalence.
	Observe that $\Coh^-(\rR \hPsi^\circ_{Z / X}(X))$ has a canonical $t$-structure, which is left complete and right bounded.
	Furthermore, combining Lemmas \ref{lem:flat_cosimplicial_t_structure} and \ref{lem:dagflatcompletion}, we deduce that
	\[ \varprojlim \Coh^-(\rR \hPsi^\circ_{Z / X}(U^\bullet)) \]
	has a $t$-structure, which is left complete and right bounded.
	A further application of \cref{lem:dagflatcompletion} shows that $p^{\bullet *}$ is both left and right $t$-exact.
	Therefore, \cref{lem:reduction_to_the_heart_II} guarantees that $p^{\bullet *}$ is an equivalence if and only if it induces an equivalence at the level of the heart of the $t$-structures we are considering.
	
	It is now sufficient to remark that
	\[ \Coh^\heartsuit(\rR \hPsi^\circ_{Z / X}(X)) \simeq \Coh^\heartsuit(\hPsi^\circ_{\trunc(Z) / \trunc(X)}(\trunc(X))) \]
	and that
	\[ \left( \varprojlim \Coh^-(\rR \hPsi^\circ_{Z / X}(U^\bullet) \right)^\heartsuit \simeq \varprojlim \Coh^\heartsuit(\rR \hPsi^\circ_{\trunc(Z) / \trunc(X)}(\trunc(U^\bullet))) . \]
	In this way, we are immediately reduced to the case of \cref{thm:etale_descent}.
\end{proof}

Just as in the underived setting, this theorem allows us to introduce the derived stack of almost perfect modules on the punctured formal neighbourhood.
We let $X\Et^{\mathrm{Geom}}$ be the big \'etale site of derived geometric stacks over $X$.
The inclusion $(X\Et, \tauet) \hookrightarrow (X\Et^{\mathrm{Geom}}, \tauet)$ satisfies the conditions of \cite[Proposition 2.22]{Porta_Yu_Higher_analytic_stacks_2014} and therefore it induces an equivalence
\[ \St_{\Catst}(X\Et, \tauet) \simeq \St_{\Catst}(X\Et^{\mathrm{Geom}}, \tauet) . \]
Thus, the above theorem allows to interpret the functor $\rR \hPhi^\circ_{Z / X}$ as an object in $\St(X\Et^{\mathrm{Geom}}, \tauet)$.
The canonical map $X \to \Spec(k)$ induces a well defined morphism of sites
\[ u \colon (\mathrm{dAff}_k, \tauet) \to (X\Et^{\mathrm{Geom}}, \tauet) . \]
We can therefore give the following definition:

\begin{defin} \label{def:derived_coherent_PFN}
	Let $X$ be a derived Artin stack locally almost of finite presentation and let $Z \hookrightarrow X$ be a derived closed substack.
	We define the functor $\rR \bfCoh^{\otimes,-}_{\hZ \smallsetminus Z}$ as the composition
	\[ \rR \bfCoh^{\otimes,-}_{\hZ \smallsetminus Z} \coloneqq \rR \hPhi^{\circ, \otimes}_{Z / X} \circ u \colon \mathrm{dAff}_k \to \Catstmon . \]
	We further define the symmetric monoidal stable $\infty$-category $\rR \Coh^-(\hZ \smallsetminus Z)$ of almost perfect modules on the punctured formal neighbourhood of $Z$ inside $X$ to be the category of $k$-points of $\rR \bfCoh^{\otimes,-}_{\hZ \smallsetminus Z}$, i.e.\
	\[ \rR \Coh^-(\hZ \smallsetminus Z) \coloneqq \rR \bfCoh^{\otimes,-}_{\hZ \smallsetminus Z}(\Spec(k)) . \]
	Similarly, we define
	\[ \rR \bfPerf^\otimes_{\hZ \smallsetminus Z} \coloneqq \rR \hrP^{\circ, \otimes}_{Z / X} \circ u, \quad \text{and} \quad \rR \Perf(\hZ \smallsetminus Z) \coloneqq \rR \bfPerf^\otimes_{\hZ \smallsetminus Z}(\Spec(k)) .  \]
\end{defin}

We can now easily obtain derived counterparts of our main formal decomposition results, namely \cref{prop:one_step_flag_decomposition_coh} and \cref{prop:flag_decomposition_induction_step}.

We start with the formal gluing along a closed substack.
Let $X$ be a derived Artin stack locally almost of finite presentation over $k$ and let $Z \hookrightarrow X$ be a closed substack.
Let $V \coloneqq X \smallsetminus Z$ and introduce the functor
\[ \rR \fW_V \coloneqq - \times^{\mathrm{d}}_X V \colon X\Et \to \dSt_k . \]
We denote by $\rR \fCoh_V^\otimes$ the composition $\rR \fCoh_V^\otimes \coloneqq \rR \bfCoh^{\otimes,-} \circ \rR \fW_V$ and by $\rR \fPerf_V^\otimes$ the composition $\rR \fPerf_V^\otimes \coloneqq \rR \bfPerf^\otimes \circ \rR \fW_V$.

Coherently with this notation, we let $\rR \fW_X$ denote the inclusion functor:
\[ \rR \fW_X \coloneqq - \times^{\mathrm{d}}_X X \colon X\Et \to \dSt_k . \]
Observe that there is a natural transformation
\[ \alpha \colon \rR \hPsi_Z \to \rR \fW_X , \]
and since $\fW_X$ factors through $\mathrm{dAff}_k \hookrightarrow \dSt_k$, we conclude that $\alpha$ induces a natural transformation
\[ \alpha^{\mathrm{aff}} \colon \rR \hPsi_Z^{\mathrm{aff}} \to \rR \fW_X . \]
Observe also that for every $U \in X\Et$, the canonical map $\rR \hPsi_Z^{\mathrm{aff}}(U) \to U$ is flat in the derived sense.
Introduce the following pullback in $\mathrm{dAff}_k$:
\[ \begin{tikzcd}
W \arrow{r}{g} \arrow[hook]{d} & Z_U \arrow[hook]{d} \\
\rR \hPsi^{\mathrm{aff}}_Z(U) \arrow{r} & U .
\end{tikzcd} \]
Since $\rR \hPsi^{\mathrm{aff}}_Z(U) \to U$ is flat, we deduce that the same thing goes for $g \colon W \to Z_U$.
In particular, this is a strong morphism.
We now observe that \cref{thm:dag_completion_truncation} implies $\trunc(g)$ is an equivalence.
It follows that $g$ is an equivalence as well.

This implies that there is an induced commutative square
\[ \begin{tikzcd}
	\rR \hPsi^\circ_Z(U) \arrow{r}{\beta_U} \arrow[hook]{d}{i_U} & \rR \fW_V(U) \arrow[hook]{d} \\
	\rR \hPsi_Z^{\mathrm{aff}}(U) \arrow{r} & \rR \fW_X(U) .
\end{tikzcd} \]
Since the natural transformation $\rR \fW_V \hookrightarrow \rR \fW_X$ is a Zariski open immersion, we can use the universal property of open immersions (cf.\ \cite[Remark 2.3.4]{DAG-V}) to promote the transformations $\beta_U$ into a natural transformation
\[ \beta \colon \rR\hPsi^\circ_Z \to \rR\fW_V . \]
Composing with $\bfCoh^{\otimes,-}$ and with $\bfPerf^\otimes$, we obtain well defined restriction maps
\[ \beta^* \colon \rR\fCoh_V^\otimes \to \rR\hPhi_Z^{\circ, \otimes} , \quad \beta^* \colon \rR \fPerf_V^\otimes \to \rR \hrP_Z^{\circ, \otimes} . \]
On the other side, the natural transformation $i \colon \rR \hPsi^{\circ, \otimes}_Z \to \hPsi_Z^{\mathrm{aff}}$ induces natural transformations
\[ i^* \colon \rR \hPhi_Z \to \rR \hPhi^{\circ, \otimes}_Z , \quad i^* \colon \rR \hrP^\otimes_Z \to \rR \hrP^{\circ, \otimes}_Z . \]

The following theorem is the derived analogue of theorems \ref{prop:one_step_flag_decomposition_coh} and \ref{thm:one_step_flag_decomposition_perf}.

\begin{thm}\label{prop:one_step_flag_decomposition_coh_derived}
	Let $X$ be a derived Artin stack locally almost of finite type over $k$ and let $Z \hookrightarrow X$ be a closed substack.
	The natural transformations $\beta^* \colon \rR \fCoh^\otimes_V \to \rR \hPhi^{\circ, \otimes}_Z$ and $i^* \colon \rR \hPhi^\otimes_Z \to \rR \hPhi^{\circ, \otimes}_Z$ induce an equivalences in $\Sh_{\Catstmon}(X\Et, \tauet)$:
	\[ \rR \fCoh^\otimes_X \simeq \rR \fCoh^\otimes_V \times^{\mathrm{d}}_{\rR \hPhi^{\circ, \otimes}_Z} \rR \hPhi_Z \quad \text{and} \quad \rR \fPerf^\otimes_X \simeq \rR \fPerf^\otimes_V \times^{\mathrm{d}}_{\rR \hrP^{\circ, \otimes}_Z} \rR \hrP^\otimes_Z . \]
\end{thm}

\begin{proof}
	As usual, it is enough to prove the analogous statement for the downgraded functors taking values in $\Catst$ instead of in $\Catstmon$.
	It is clear that $\rR \fCoh_V$, $\rR \hPhi_Z$, $\rR \fPerf_V$ and $\rR \hrP_Z$ define sheaves on $X\Et$.
	Furthermore, \cref{thm:etale_descent} shows that also $\rR \hPhi^\circ_Z$ and $\rR \hrP^\circ_Z$ are sheaves on $X\Et$.
	Since limits in $\Sh_{\Catst}(X\Et, \tauet)$ can be computed in $\PSh_{\Catst}(X\Et)$, we see that the question is local on $X$.
	We can therefore suppose $X$ to be affine.
	In this case, the result is a consequence of \cite[Theorem 7.4.0.1]{Lurie_SAG} and of the fact that for $U \in X\Et$, the maps $\rR \fW_V(U) \to \rR \fW_X(U)$ and $\rR \hPsi_Z(U) \to \rR \fW_X(U)$ form an fpqc cover of $\rR \fW_X(U) = U$.
\end{proof}

As a consequence of Theorem \ref{prop:one_step_flag_decomposition_coh_derived}, we get the following derived analog of Corollary \ref{cor:one_step_flag_decomposition_cohperf}.

\begin{cor}\label{cor:one_step_flag_decomposition_cohperf_derived}
	Let $X$ be a derived Artin stack locally of finite type over $k$, and let $Z \hookrightarrow X$ be a closed substack.
	Let $V \coloneqq X \smallsetminus Z$ be the open complement.
	Then there are canonical decompositions of derived stacks
	\[ \rR \bfCoh^{\otimes,-}_X \simeq \rR \bfCoh^{\otimes,-}_V \times_{\rR \bfCoh^{\otimes,-}_{\hZ \smallsetminus Z}} \rR \bfCoh^{\otimes,-}_{\hZ} \]
	and
	\[ \rR \bfPerf^\otimes_X \simeq \rR \bfPerf^\otimes_V \times_{\rR \bfPerf^\otimes_{\hZ \smallsetminus Z}} \rR \bfPerf^\otimes_{\hZ} . \]
	In particular, by taking $k$-points, we get equivalences of symmetric monoidal stable $\infty$-categories
	\[ \rR \Coh^-(X) \simeq \rR \Coh^-(V) \times_{\rR \Coh^-(\hZ \smallsetminus Z)} \rR \Coh^-(\hZ) \]
	and
	\[ \rR \Perf(X) \simeq \rR \Perf(V) \times_{\rR \Perf(\hZ \smallsetminus Z)} \rR \Perf(\hZ) . \]
\end{cor}

\bigskip

We now turn to the formal decomposition along a general flag.

\begin{defin} \label{def:derived_flag}
	Let $X$ be a derived Artin stack locally almost of finite type over $k$.
	A flag $\cZ = (Z_0, Z_1, \ldots, Z_n)$ on $X$ is a flag on $\trunc(X)$ in the sense of \cref{definitionofflag}.
\end{defin}

\begin{rem} Note that if $Y_n \hookrightarrow Y_{n-1} \hookrightarrow \cdots \hookrightarrow Y_1 \hookrightarrow Y_0=X$ is a chain of closed immersions of derived stacks, and we put  $Z_i:= \trunc(Y_i)$, then $(Z_0, Z_1, \ldots, Z_n)$ is a flag on $X$.
\end{rem}

Similarly to the underived case, we define derived stacks
\[ \rR \hPhi^\otimes_i,\,\rR \hPhi^{\circ,\otimes}_i,\, \rR \hrP^\otimes_i,\,\rR \hrP^{\circ,\otimes}_i \colon X\Et\op \to \Catst \]
as follows.
For $i \ge 1$, denote by $\cJ_i$ the category of derived closed substacks $Z_i^{(\alpha)}$ of $X$ such that $(Z_i^{(\alpha)})_{\mathrm{red}} \simeq (Z_i)_{\mathrm{red}}$ (as closed substacks).
Then we set
\begin{align*}
&i=1 & \rR\hPhi^{\circ, \otimes}_1 &\coloneqq \rR\hPhi^{\circ, \otimes}_{Z_1/X}, & \rR\hrP^{\circ,\otimes}_1 &\coloneqq \rR\hrP^{\circ,\otimes}_{Z_1/X} \\
&i\geq 2 & \rR\hPhi^{\circ,\otimes}_{i} &\coloneqq \varprojlim_{\alpha \in \cJ_i} \hPhi^{\circ,\otimes}_{Z_{i}/Z_{i-1}^{(\alpha)}}, & \rR\hrP^{\circ,\otimes}_{i} &\coloneqq \varprojlim_{\alpha \in \cJ_i} \rR\hrP^{\circ,\otimes}_{Z_{i}/Z_{i-1}^{(\alpha)}} \\
&i=0 & \rR\hPhi^\otimes_0 &\coloneqq \rR\fCoh^\otimes_X, & \rR\hrP^\otimes_0 &\coloneqq \rR\fPerf^\otimes_X \\
&i=1 & \rR\hPhi^\otimes_1 &\coloneqq \rR\hPhi^\otimes_{Z_1/X}, & \rR\hrP^\otimes_1 &\coloneqq \rR\hrP^\otimes_{Z_1/X} \\
&i\geq 2 & \rR\hPhi^\otimes_{i} &\coloneqq \varprojlim_{\alpha \in \cJ_i} \rR\hPhi^\otimes_{Z_{i}/Z_{i-1}^{(\alpha)}}, & \rR\hrP^\otimes_{i} &\coloneqq \varprojlim_{\alpha \in \cJ_i} \rR\hrP^\otimes_{Z_{i}/Z_{i-1}^{(\alpha)}}.
\end{align*}
Note that we have
\begin{align*}
\rR \hPhi^\otimes_{i+1} &\simeq \varprojlim_{\alpha \in \cJ_i, \beta \in \cJ_{i+1}} \rR \hPhi^\otimes_{Z_{i+1}^{(\beta)} / Z_i^{(\alpha)}} \simeq \rR \hPhi^\otimes_{Z_{i+1} / X} \\
\rR \hrP_{i+1}^\otimes &\simeq \varprojlim_{\alpha \in \cJ_i, \beta \in \cJ_{i+1}} \rR \hrP^\otimes_{Z_{i+1}^{(\beta)} / Z_i^{(\alpha)}} \simeq \rR \hrP^\otimes_{Z_{i+1} / X}
\end{align*}
Finally, we let $V_{i+1}^{(\alpha)}$ be the open complement of $Z_{i+1}$ inside $Z_i^{(\alpha)}$, and we set
\[ \rR \fCoh^\otimes_{V,i+1} \coloneqq \varprojlim_{\alpha \in \cJ_i} \rR \fCoh^\otimes_{V_{i+1}^{(\alpha)}} \hspace{1cm} \text{and} \hspace{1cm} \rR \fPerf^\otimes_{V,i+1} \coloneqq \varprojlim_{\alpha \in \cJ_i} \rR \fPerf^\otimes_{V_{i+1}^{(\alpha)}} . \]
Note that the sequence $\{V_{i+1}^{(\alpha)}\}$ defines a formal open subscheme $\fV_{i+1}$ of
\[ \widehat{Z_i} = \colim_{\alpha \in \cJ_i} Z_i^{(\alpha)} . \]
With these notations, we have $\rR \fCoh^\otimes_{V,i+1}(X) = \rR\Coh^-(\fV_{i+1})$ as well as $\rR \fPerf^\otimes_{V,i+1}(X) = \rR \Perf^-(\fV_{i+1})$.

We have the following derived analog of Proposition \ref{prop:flag_decomposition_induction_step} :

\begin{prop} \label{prop:flag_decomposition_induction_step_derived}
	Let $X$ be a derived Artin stack locally almost of finite presentation over $k$ and let $\cZ = (Z_0, Z_1, \ldots, Z_n)$ be a flag on $X$.
	Then for every $0 \le i \le n-1$ there are natural equivalences in $\Sh_{\Catstmon}(X\Et, \tauet)$:
	\[ \rR \hPhi^\otimes_i \simeq \rR \fCoh^\otimes_{V,i+1} \times^{\mathrm{d}}_{\rR \hPhi^{\circ,\otimes}_{i+1}} \rR \hPhi^\otimes_{i+1} \]
	and
	\[ \rR \hrP^\otimes_i \simeq \rR \hrP^\otimes_{V, i+1} \times^{\mathrm{d}}_{\rR \hrP^{\circ,\otimes}_{i+1}} \rR \hrP^\otimes_{i+1} . \]
\end{prop}

\begin{proof}
	Unravelling the definitions and using the fact that limits commute with limits, we see that this proposition follows by applying iteratively \cref{prop:one_step_flag_decomposition_coh_derived}.
\end{proof}

As a consequence, we get the derived analog of Corollary \ref{cor:full_flag_decomposition} :

\begin{cor} \label{cor:full_flag_decomposition_derived}
	Let $X$ be a derived Artin stack locally almost of finite presentation over $k$ and let $\cZ \coloneqq (Z_0, Z_1, \ldots, Z_n)$ be a flag on $X$.
	Then there are equivalences in $\Sh_{\Catstmon}(\mathrm{dAff}_k, \tauet)$:
	\[ \rR \bfCoh^{\otimes,-}_X \simeq \rR \bfCoh^{\otimes,-}_{\fV_1} \times^{\mathrm{d}}_{\rR \bfCoh^{\otimes,-}_{\widehat{Z_1} \smallsetminus Z_1}} \left( \rR \bfCoh^{\otimes,-}_{\fV_2} \times^{\mathrm{d}}_{\rR \bfCoh^{\otimes,-}_{\widehat{Z_2} \smallsetminus Z_2}} \left( \cdots \left( \rR \bfCoh^{\otimes,-}_{\fV_n} \times^{\mathrm{d}}_{\rR \bfCoh^{\otimes,-}_{\widehat{Z_n} \smallsetminus Z_n}} \rR \bfCoh^{\otimes,-}_{\widehat{Z_n}} \right) \cdots \right) \right) \]
	and
	\[ \rR \bfPerf^\otimes_X \simeq \rR \bfPerf^\otimes_{\fV_1} \times^{\mathrm{d}}_{\rR \bfPerf^\otimes_{\widehat{Z_1} \smallsetminus Z_1}} \left( \rR \bfPerf^\otimes_{\fV_2} \times^{\mathrm{d}}_{\rR \bfPerf^\otimes_{\widehat{Z_2} \smallsetminus Z_2}} \left( \cdots \left( \rR \bfPerf^\otimes_{\fV_n} \times^{\mathrm{d}}_{\rR \bfPerf^\otimes_{\widehat{Z_n} \smallsetminus Z_n}} \rR \bfPerf^\otimes_{\widehat{Z_n}} \right) \cdots \right) \right) . \]
	Furthermore, by taking $k$-points in the above equivalences, we obtain the following equivalences of symmetric monoidal stable $\infty$-categories:
	\begin{multline*}
		\rR \Coh^-(X) \simeq \rR \Coh^-(\fV_1) \times^{\mathrm{d}}_{\rR \Coh^-(\widehat{Z_1} \smallsetminus Z_1)} \bigg( \rR \Coh^-(\fV_2) \times^{\mathrm{d}}_{\rR \Coh^-(\widehat{Z_2} \smallsetminus Z_2)} \Big( \cdots \\
		\cdots \left( \rR \Coh^-(\fV_n) \times^{\mathrm{d}}_{\rR \Coh^-(\widehat{Z_n} \smallsetminus Z_n)} \rR \Coh^-(\widehat{Z_n}) \right) \cdots \Big) \bigg)
	\end{multline*}
	and
	\begin{multline*}
		\rR \Perf(X) \simeq \rR \Perf(\fV_1) \times^{\mathrm{d}}_{\rR \Perf(\widehat{Z_1} \smallsetminus Z_1)} \bigg( \rR \Perf(\fV_2) \times^{\mathrm{d}}_{\rR \Perf(\widehat{Z_2} \smallsetminus Z_2)} \Big( \cdots \\
		\cdots \left( \rR \Perf(\fV_n) \times^{\mathrm{d}}_{\rR \Perf(\widehat{Z_n} \smallsetminus Z_n)} \rR \Perf(\widehat{Z_n}) \right) \cdots \Big) \bigg)
	\end{multline*}
\end{cor}

\bigskip

We conclude this section by obtaining also a derived version of the flag decomposition result for $G$-bundles (\cref{section:gbundles}).
As in \cref{section:gbundles}, we fix an affine $k$-algebraic group $G$. 
Recall that the derived stack of $G$-bundles can be seen as the functor $\dSt_k\op \to \cS$
\[
\rR \Bun_G \colon Y \mapsto \Map(Y,\mathrm{B}G)
\]
where $\Map$ is the mapping space functor and $\mathrm BG$ is the classifying (underived) stack of $G$-bundles.
We define functors $X\Et \to \dSt_k \to \cS^\mathrm{op}$ as follows:
\begin{gather*}
	\rR \Bunhat^\circ_G = \rR \Bun_G \circ \rR \hPsi^\circ_Z , \qquad \rR \Bunhat_G = \rR \Bun_G \circ \rR \hPsi_Z , \\
	\rR \bfBun_G^\circ = \rR \Bun_G \circ \rR \fW_V , \qquad \rR \bfBun_G^X = \rR \Bun_G \circ \rR \fW_X  .
\end{gather*}
As in the underived setting, Tannaka duality (cf.\ \cite[Theorem 3.4.2]{Lurie_Tannaka_duality}) implies the following proposition, the proof being exactly the one of \cref{lem:bgperflocal}.

\begin{prop}\label{prop:derivedbgperflocal}
Let $X \to Y$ be a map of derived stacks.
If the induced functor $\rR \Perf(Y) \to \rR \Perf(X)$ is an equivalence then so is the induced map of spaces $\rR \Bun_G(Y) \to \rR \Bun_G(X)$.
\end{prop}

Using this, and \cref{thm:etale_descent_derived}, we get:

\begin{prop}
	The functors $\rR \Bunhat_G$ and $\rR \Bunhat^\circ_G$ are hypercomplete sheaves for the \'etale topology.
\end{prop}

Recall the morphism of sites $u \colon (\mathrm{dAff}_k, \tauet) \to (X\Et^{\mathrm{Geom}}, \tauet)$.
As in the case of (almost) perfect modules, we define the following derived stacks $\mathrm{dAff}_k\op \to \cS$:
\begin{align*}
& \rR \bfBun_G(X) = \rR \bfBun_G^X \circ u \,\, \text{ of }G\text{-bundles on }X \\
& \rR \bfBun_G(V) = \rR \bfBun_G^\circ \circ u \,\, \text{ of }G\text{-bundles on } V = X \smallsetminus Z \\
& \rR \bfBun_G(\hat{Z}) = \rR \Bunhat_G \circ u \,\, \text{ of }G\text{-bundles on the formal neighbourhood }\hat{Z} \\
& \rR \bfBun_G(\hat{Z} \smallsetminus Z) = \rR \Bunhat^\circ_G \circ u \,\, \text{ of }G\text{-bundles on the punctured formal neighbourhood}
\end{align*}
As in the case of (almost) perfect complexes, we extend those definitions to the case where $X$ is a formal derived ind-stack (i.e. a filtered colimit of derived Artin stacks locally of finite presentation whose reduced part is constant).

Next, combining \cref{prop:one_step_flag_decomposition_coh_derived} and \cref{prop:derivedbgperflocal}, we deduce the following derived analogs of Theorem \ref{thm:gbundles}, and Corollary \ref{flagbung}, respectively.

\begin{thm} \label{thm:flag_bung_derived}
	There are natural equivalences
	\[ \rR \bfBun_G^X \simeq \rR \bfBun^\circ_G \times^{\mathrm{d}}_{\rR \Bunhat_G^\circ} \rR \Bunhat_G . \]
	and
	\[ \rR \bfBun_G(X) \simeq \rR \bfBun_G(X \smallsetminus Z) \times^{\mathrm{d}}_{\rR \bfBun_G(\hZ \smallsetminus Z)} \rR \bfBun_G(\hZ) . \]
\end{thm}

\begin{cor}\label{cor:flag_bung_derived}
	Let $X$ be a derived Artin stack locally almost of finite presentation over $k$ and let $\cZ \coloneqq (Z_0, Z_1, \ldots, Z_n)$ be a flag on $X$ (cf.\ \cref{def:derived_flag}). With the notations introduced right before \cref{prop:flag_decomposition_induction_step_derived}, the stack $\rR \bfBun_G(X)$ is equivalent to
	\[
	\rR \bfBun_G(\fV_1) \times^{\mathrm{d}}_{\rR \bfBun_G(\widehat{Z_1} \smallsetminus Z_1)} \left( \rR \bfBun_G(\fV_2) \times^{\mathrm{d}}_{\rR \bfBun_G(\widehat{Z_2} \smallsetminus Z_2)} \left( \cdots \left( \rR \bfBun_G(\fV_n) \times^{\mathrm{d}}_{\rR \bfBun_G(\widehat{Z_n} \smallsetminus Z_n)} \rR \bfBun_G(\widehat{Z_n}) \right) \cdots \right) \right) .
	\]
	In the particular case where $X$ is a surface $S$ and the flag consists of a point $x$ in a curve $C$ in $S$, we get:
	\[
	\rR \bfBun_G(S) \simeq \rR \bfBun_G(S \smallsetminus C) \times^{\mathrm{d}}_{\rR \bfBun_G(\widehat{C} \smallsetminus C)} \left( \rR \bfBun_G(\widehat{C} \smallsetminus x) \times^{\mathrm{d}}_{\rR \bfBun_G(\hat{x} \smallsetminus x)} \rR \bfBun_G(\hat{x}) \right) 
	\].
\end{cor}

\begin{rem}
The above results improve \cref{cor:full_flag_decomposition} and \cref{flagbung} even in the case where $X$ is a variety. In some sense, the above results are an extension of \cref{cor:full_flag_decomposition} and \cref{flagbung} to the derived structure of those stacks of (almost) perfect complexes or of $G$-bundles.
\end{rem}

\section{Sheaves on the punctured formal neighbourhood: the case of schemes} \label{subsec:case_schemes}

In \cref{{section:sheaves_on_PFN}}, we have constructed a couple of invariants of a closed immersion of Artin stacks (locally of finite presentation over $k$) $Z \hookrightarrow X$: on one side, we defined the stacks $\bfCoh^-_{\hZ \smallsetminus Z}$ and $\bfPerf_{\hZ \smallsetminus Z}$, and on the other side we introduced the categories $\Coh^-(\hZ \smallsetminus Z)$ and $\Perf(\hZ \smallsetminus Z)$.
Furthermore, Propositions \ref{prop:etale_descent_coh} and \ref{prop:etale_descent_Perf} show that if $X = \Spec(A)$ is affine, then
\[ \Coh^-(\hZ \smallsetminus Z) = \Coh^-(\Spec(\widehat{A}) \smallsetminus Z), \quad \Perf(\hZ \smallsetminus Z) = \Perf(\Spec(\widehat{A}) \smallsetminus Z) . \]
In the particular case where $Z$ is defined by a single equation, the scheme $\hPsi^\circ_Z(X) = \Spec(\widehat{A}) \smallsetminus Z$ is affine and therefore the above description yields a particularly satisfactory understanding of these categories.

Nonetheless, in many interesting cases, $X$ will not be affine.
Since the functor $\hPsi^\circ_Z \colon X\Et \to \St_k$ is not continuous for the \'etale topology, we cannot extend it to a functor defined over the whole $X\Et^{\Geom}$ and therefore it is difficult to get a grasp of the categories $\Coh^-(\hZ \smallsetminus Z)$ and $\Perf(\hZ \smallsetminus Z)$.
In fact, the definition itself of these categories passes through a \emph{sheafification} process.\medbreak

The goal of this section is to provide a more direct characterization of the category $\Coh^-(\hZ \smallsetminus Z)$ in the \emph{special case} where $X$ is \emph{a quasi-compact quasi-separated scheme} locally of finite presentation over $k$.
This description (provided by \cref{thm:Zariski_computational_tool} below) has the advantage of not involving any sheafification formula.

In order to state a precise theorem, we need some preliminary observations.
We first notice that the natural transformation
\[ i \colon \hPsi^\circ_Z \to \hPsi_Z \]
induces a natural transformation
\[ i^* \colon \hPhi_Z \to \hPhi_Z^\circ . \]
We let
\[ \cK \coloneqq \ker(i^*) . \]
This is an $\infty$-functor
\[ \cK \colon X\Et \to \Catst . \]

\begin{lem} \label{lem:kernel_sheaf}
	The functor $\cK$ is a hypercomplete sheaf for the \'etale topology.
\end{lem}

\begin{proof}
	One can deduce this directly from \cref{thm:etale_descent}.
	It is nevertheless interesting to observe that there is a simpler direct proof.
	Since $\cK$ is a fully faithful sub-presheaf of $\hPhi_Z$ and since the latter is a hypercomplete sheaf, it is enough to prove that the property of belonging to $\cK$ is local for the \'etale topology.
	Let $U \in X\Et$ and choose an \'etale cover $f \colon U^0 \to U$.
	Let $\cF \in \Coh^-(\widehat{Z_U})$ and suppose that $i_{U^0}^* \hf^*(\cF) \simeq 0$.
	We want to conclude that $i_U^*(\cF) \simeq 0$.
	Using \cref{lem:etale_sent_to_flat}, we see that the morphism
	\[ \widehat{A_U} \to \widehat{A_{U^0}} \]
	obtained by passing to global sections of $\hf$, is faithfully flat.
	Since the square
	\[ \begin{tikzcd}
	\hPsi_Z^\circ(U^0) \arrow{r}{\hPsi_Z^\circ(f)} \arrow{d}{i_{U^0}} & \hPsi_Z^\circ(U) \arrow{d}{i_U} \\
	\Spec(\widehat{A_{U^0}}) \arrow{r}{f} & \Spec(\widehat{A_U})
	\end{tikzcd} \]
	is a pullback in the category of schemes, we conclude that the map $\hPsi_Z^\circ(f)$ is faithfully flat as well.
	As consequence, $i_U^*(\cF)$ is zero if and only if
	\[ \hPhi_Z^\circ(f)(i_U^*(\cF)) = ( \hPsi_Z^\circ(f) )^*(i_U^*(\cF)) \]
	is zero.
\end{proof}

As consequence, we can view $\cK$ as a functor
\[ \cK \colon (X\Et^{\Geom})\op \to \Catst . \]
It comes equipped with a natural transformation $\alpha \colon \cK \hookrightarrow \hPhi_Z$.

\begin{prop} \label{prop:algebraic_universal_property}
	The natural transformation $L \colon \hPhi_Z \to \hPhi_Z^\circ$ exhibits $\hPhi_Z^\circ$ as the cofiber of $\alpha \colon \cK \hookrightarrow \hPhi_Z$ in $\Sh_{\Catst}(X\Et^{\Geom},\tauet)$.
\end{prop}

\begin{proof}
	It is enough to show that $\hPhi_Z^\circ$ is the cofiber of $\alpha$ in $\Sh_{\Catst}(X\Et, \tauet)$.
	In this case, the proposition is a direct consequence of the definition of $\hPhi_Z^\circ$ and of the localization theorem for almost perfect modules, \cref{thm:localization_for_coh_minus}.
\end{proof}

With these preparations, we are ready to state the main result of this section:

\begin{thm} \label{thm:Zariski_computational_tool}
	Let $X$ be a quasi-compact and quasi-separated scheme which is of finite type over $k$.
	Let $Z \hookrightarrow X$ be a closed subscheme.
	Then the sequence
	\[ \cK(X) \hookrightarrow \Coh^-(\hZ) \to \Coh^-(\hZ \smallsetminus Z) \]
	is a cofiber sequence in $\Catst$.
\end{thm}

Let us stress immediately that the above theorem is not an immediate consequence of \cref{prop:algebraic_universal_property}.
Indeed, cofiber sequences of stacks cannot be computed objectwise, at least not a priori.
And indeed, we don't expect the above result to hold in greater generality than this: the fact that $X$ is a (quasi-compact and quasi-separated) \emph{scheme}, rather than an Artin stack, plays a crucial role in the proof.

Notice that \cref{thm:Zariski_computational_tool} can also be seen as a generalization of \cref{thm:localization_for_coh_minus} to formal completions (of closed subschemes inside quasi-compact and quasi-separated schemes of finite type over $k$).
The proof of this result is very long and technical, and it will occupy our attention until the end of this section.
Before starting the proof, we would like to emphasise that this result is completely new and it doesn't rely \emph{at all} on the former results on formal gluing contained in the literature.
More specifically, there are two main difficulties that have to be overcome in order to obtain \cref{thm:Zariski_computational_tool}:
\begin{enumerate}
	\item On one side, the assignment
	\[ X\Et \ni U \mapsto \Coh^-(\widehat{Z_U}) / \cK(U) \in \Catst \]
	is not a sheaf a priori (and, indeed, to prove that it satisfies \'etale descent is equivalent to \cref{thm:Zariski_computational_tool} itself).
	For this reason, the above theorem is a significant improvement of the results obtained in \cite{Ben-Bassat_Temkin_Tubular_2013} because it provides a \emph{global} description of the category of almost perfect modules over the punctured formal neighbourhood.
	\item On the other side, the equations defining $Z$ inside $X$ are not globally defined. For this reason, it is impossible to identify $\Coh^-(\hZ \smallsetminus Z)$ as the category of almost perfect modules on a ``generic fiber'' of $\hZ$, in the spirit of \cite{Bosch_Gortz_Coherent_modules_1998}. Note that this very same issue was already present in our treatment of \'etale descent in \cref{subsec:etale_descent}. Nevertheless, \cref{thm:etale_descent} is local in nature, and this enables us to reduce to the situation where $Z$ is defined by a single global equation in $X$, and then it is a routine argument in deducing the result from the main theorem of \cite{Bosch_Gortz_Coherent_modules_1998}.
	In dealing with \cref{thm:Zariski_computational_tool}, however, we can no longer apply the same trick, because the statement is intrinsically global (as remarked in the previous point). Furthermore, it becomes also impossible to reduce ourselves to the case where $Z$ is (locally) defined by a single equation. For all these reasons, the proof of \cref{thm:Zariski_computational_tool} can be seen as a fairly non-trivial extension of the techniques introduced in \cite{Bosch_Gortz_Coherent_modules_1998}.
\end{enumerate}

\medbreak

As last thing before turning to the proof of \cref{thm:Zariski_computational_tool}, let us emphasise a couple of consequences:

\begin{cor} \label{cor:algebraic_universal_property}
	Let $X$ be a quasi-compact and quasi-separated scheme which is locally of finite type over $k$.
	Let $u \colon (\Aff_k, \tauet) \to (X\Et^{\Geom}, \tauet)$ be the continuous morphism of sites induced by pullback along $X \to \Spec(k)$.
	The sequence
	\[ \cK \circ u \hookrightarrow \bfCoh^-_{\hZ} \to \bfCoh^-_{\hZ \smallsetminus Z}  \]
	is a cofiber sequence in $\Sh_{\Catst}(\Aff_k, \tauet)$.
\end{cor}

\begin{cor} \label{cor:Zariski_computational_derived}
	Let $X$ be a quasi-compact and quasi-separated derived scheme locally almost of finite type over $k$ and let $Z \hookrightarrow X$ be a closed derived subscheme.
	Let $\rR \cK \coloneqq \ker ( \rR \Coh^-(\hZ) \to \rR \Coh^-(\hZ \smallsetminus Z)$.
	Then the triangle
	\[ \rR \cK \hookrightarrow \rR \Coh^-(\hZ) \to \rR \Coh^-(\hZ \smallsetminus Z) \]
	is a cofiber sequence in $\Catst$.
\end{cor}

\begin{proof}
	Choose a Zariski affine hypercover $U^\bullet$ for $X$.
	Then \cref{thm:etale_descent_derived} shows that
	\[ \rR \Coh^-(\hZ \smallsetminus Z) \simeq \varprojlim_{n \in \mathbf \Delta} \rR \Coh^-(\widehat{Z_{U^n}} \smallsetminus Z_{U^n}) . \]
	In particular, we can use \cref{lem:flat_cosimplicial_t_structure} in order to endow $\rR \Coh^-(\hZ \smallsetminus Z)$ with a $t$-structure such that
	\[ ( \rR \Coh^-(\hZ \smallsetminus Z) )^\heartsuit \simeq \varprojlim_{n \in \mathbf \Delta} ( \rR \Coh^-(\widehat{Z_{U^n}} \smallsetminus Z_{U^n}) )^\heartsuit \simeq \varprojlim_{n \in \mathbf \Delta} \Cohh( \trunc( \widehat{Z_{U^n}} ) \smallsetminus \trunc(Z_{U^n}) ) . \]
	It follows that the canonical functor $\rR \Coh^-(\hZ) \to \rR Coh^-(\hZ \smallsetminus Z)$ is $t$-exact.
	In particular, the $t$-structure on $\rR \Coh^-(\hZ)$ restricts to a $t$-structure on $\rR \cK$, and moreover
	\[ \rR \cK^\heartsuit \simeq \ker \left( \rR \Coh^-(\hZ)^\heartsuit \to \rR \Coh^-(\hZ \smallsetminus Z )^\heartsuit \right) \simeq \ker \left( \Cohh( \trunc(\hZ) ) \to \Cohh( \trunc(\hZ) \smallsetminus \trunc(Z) ) \right) \]
	Using \cref{lem:dagflatcompletion} we can thus reason as in the proof of \cref{thm:localization_for_coh_minus} to deduce that the hypotheses of \cref{cor:quotient_t_structure} are satisfied.
	Thus, we can endow the quotient $\rR \Coh^-(\hZ) / \rR \cK$ with a $t$-structure such that the canonical map
	\[ \rR \Coh^-(\hZ) \to \rR \Coh^-(\hZ) / \rR \cK \]
	is $t$-exact.
	Since this map is also essentially surjective, we see that the induced map
	\[ \varphi \colon \rR \Coh^-(\hZ) / \rR \cK \to \rR \Coh^-(\hZ \smallsetminus Z) \]
	is $t$-exact as well.
	In virtue of \cref{lem:reduction_to_the_heart_II}, it is therefore enough to prove that $\varphi$ restricts to an equivalence on hearts in order to deduce that $\varphi$ was an equivalence to start with.
	Nevertheless, on hearts the map is equivalent to
	\[ \rR \Cohh( \trunc( \hZ) ) / \cK^\heartsuit \to \Cohh( \trunc( \hZ ) \smallsetminus \trunc(Z) ) , \]
	which is an equivalence in virtue of \cref{thm:Zariski_computational_tool}.
\end{proof}

\subsection{Preliminary observations and beginning of the proof} \label{subsubsec:preliminaries}

When $X$ is affine, \cref{thm:Zariski_computational_tool} is a direct consequence of \cref{thm:etale_descent}.
In the general case, the difficulty is that we don't know whether the Verdier quotient $\Coh^-(\hZ) / \cK(X)$ satisfies descent in $X$, and it is therefore impossible to reduce to the affine case.
To further complicate matters, when $X$ is not affine, the category $\cK(X)$ has not been defined directly: it is the result of a sheafification process.
The first step in the proof of \cref{thm:Zariski_computational_tool} is to obtain a better understanding of $\cK(X)$.
Applying \cref{prop:internal_characterisation_kernel_localisation}, we can provide an alternative description of the category $\cK(U)$ when $U$ is an affine scheme:

\begin{prop} \label{prop:internal_characterisation_kernel_stacky}
	Fix $U \in X\Et$.
	Then $\cK(U)$ is the smallest full stable subcategory of $\Coh^-(\widehat{Z_U})$ which is left complete and contains the essential image of
	\[ j_* \colon \Coh^-(Z_U) \to \Coh^-(\widehat{Z_U}) . \]
\end{prop}

This definition has the advantage of not involving the functor $\hPsi^\circ_Z$, which we are not able to extend to $X\Et^{\Geom}$.
However, we cannot use this proposition to provide immediately a global description of $\cK(X)$.
The following two difficulties should be noted:
\begin{enumerate}
	\item it is not obvious that the functors $j_{U*} \colon \Coh^-(Z_U) \to \Coh^-(\widehat{Z_U})$, which are defined for $U$ affine in virtue of \cref{lem:basic_formal_GAGA}, can be glued to yield right adjoints to $j_U^* \colon \Coh^-(\widehat{Z_U}) \to \Coh^-(Z_U)$ when $U$ is no longer affine;
	\item even having the functors $j_{U*}$ defined for every $U \in X\Et^{\Geom}$, it is not obvious that these functors can be arranged into a natural transformation $j_* \colon \fCoh_Z \to \hPhi_Z$.
\end{enumerate}
We will see below that (1) is a ``false'' problem: the existence of the right adjoints $j_{U*}$ for $U$ affine, combined with the fact that both $\fCoh_Z$ and $\hPhi_Z$ are sheaves for the \'etale topology is enough to guarantee the existence of $j_{U*}$ for any $U$.
On the other side, (2) is an actual problem.
However, there is a simple way to get around this issue which we are going to describe.

Consider the \emph{small} schematic \'etale site of $X$, which we simply denote by $X\et^{\Sch}$.
Let us recall that the objects of $X\et^{\Sch}$ are \'etale maps $Y \to X$, where $Y$ is a (not necessarily affine) \emph{quasi-compact} scheme which is locally of finite presentation over $k$.
The topology on $X\et^{\Sch}$ is the usual \'etale topology $\tauet$.
There is a canonical morphism of sites
\[ v \colon X\et^{\Sch} \hookrightarrow X\Et^{\Geom} , \]
which is both continuous and cocontinuous (we refer to \cite[Definitions 2.12 and 2.17]{Porta_Yu_Higher_analytic_stacks_2014} for these notions in the current case of $\infty$-sites).
In particular, the restriction functor
\[ v^s \colon \St_{\Catst}(X\Et^{\Geom}, \tauet) \to \St_{\Catst}(X\et^{\Sch}, \tauet) \]
has both a left adjoint $v_s$ and a right adjoint $\tensor*[_s]{v}{}$.
Furthermore, \cite[Lemma 2.20]{Porta_Yu_Higher_analytic_stacks_2014} shows that $\tensor*[_s]{v}{}$ commutes with the inclusion of hypercomplete sheaves into presheaves.
Passing to left adjoints, we deduce that $v^s$ (which coincides with $v^p \coloneqq - \circ v$ by \cite[Lemma 2.13]{Porta_Yu_Higher_analytic_stacks_2014}) commutes with the sheafification process.

As a consequence, we deduce that $\cK(X) = (\cK \circ v)(X)$ can be computed in an alternative way: considering $\cK$ as a presheaf on $X\Et$, we can first restrict it to the small affine \'etale site $X\et$ and then extend it (as a presheaf) to $X\et^{\Sch}$, sheafify it over $(X\et^{\Sch}, \tauet)$ and finally evaluate it over $X$.
This description is useful because it allows to give a very explicit description of $\cK \circ v$.
The main point is that, since all the maps in $X\et^{\Sch}$ are \'etale, it becomes possible to overcome issue (2) explained above.

Having argued about the importance of an explicit construction of $\cK$ over the small schematic \'etale site $X\et^{\Sch}$, we turn to the actual construction.
For this, we will address the problems (1) and (2) stated before.
In order to avoid confusion, let us fix a couple of notations.
Recall (see \cref{notation:frakCoh}) that $\fCoh_Z$ and $\hPhi_Z$ are functors
\[ \fCoh_Z, \hPhi_Z \colon X\Et \to \St_k . \]
Since they are continuous for the \'etale topology, we can extend them to functors
\[ \fCoh_Z^{\Geom}, \hPhi_Z^{\Geom} \colon X\Et^{\Geom} \to \St_k . \]
We let $\fCoh_Z^{\Sch}$ and $\hPhi_Z^{\Sch}$ be respectively the restriction of $\fCoh_Z^{\Geom}$ and of $\hPhi_Z^{\Geom}$ to the category $X\et^{\Sch}$.
Observe that both $\fCoh_Z^{\Sch}$ and $\hPhi_Z^{\Sch}$ are sheaves.
Since the morphism of sites $(X\et, \tauet) \to (X\et^{\Sch}, \tauet)$ is a Morita equivalence by \cite[Proposition 2.22]{Porta_Yu_Higher_analytic_stacks_2014}, it is enough to construct the natural transformation $j_* \colon \fCoh_Z^{\Sch} \to \hPhi_Z^{\Sch}$ as a morphism between the restrictions
\[ \fCoh_Z^{\Aff} \coloneqq \left. \fCoh_Z^{\Sch} \right|_{X\et}, \quad \hPhi_Z^{\Aff} \coloneqq \left. \hPhi_Z^{\Sch} \right|_{X\et} . \]
In order to perform this construction, it is useful to introduce the sheaves
\[ \rI \fCoh_Z^{\Aff}, \rI \hPhi_Z^{\Aff} \colon X\et\op \to \LPromegast \]
formally defined as the compositions
\[ \rI \fCoh_Z^{\Aff} \coloneqq \Ind \circ \fCoh_Z^{\Aff}, \quad \rI \hPhi_Z^{\Aff} \coloneqq \Ind \circ \hPhi_Z^{\Aff} . \]
The natural transformation $j^* \colon \hPhi_Z^{\Aff} \to \fCoh_Z^{\Aff}$ induces a well defined natural transformation $\Ind(j^*) \colon \rI \fCoh_Z^{\Aff} \to \rI \hPhi_Z^{\Aff}$.
Applying the Grothendieck construction, $j^*$ and $\Ind(j^*)$ induce morphisms of Cartesian fibrations
\[ f \colon \int \hPhi_Z^{\Aff} \to \int \fCoh_Z^{\Aff}, \quad \rI f \colon \int \rI \hPhi_Z^{\Aff} \to \int \rI \fCoh_Z^{\Aff} \]
over $X\et$.
These morphisms fit into a commutative square
\[ \begin{tikzcd}
\int \hPhi_Z^{\Aff} \arrow{r}{f} \arrow{d}{\hat{\imath}} & \int \fCoh_Z^{\Aff} \arrow{d}{i} \\
\int \rI \hPhi_Z^{\Aff} \arrow{r}{\rI f} & \int \rI \fCoh_Z^{\Aff} .
\end{tikzcd} \]
By construction, $\int \rI \hPhi_Z^{\Aff}$ and $\int \rI \fCoh_Z^{\Aff}$ are \emph{presentable} fibrations over $X\et$.
It follows that $\rI f$ admits an adjoint, that we temporarily denote as $g$:
\[ g \colon \int \rI \fCoh_Z^{\Aff} \to \int \rI \hPhi_Z^{\Aff} , \]
which is a morphism of \emph{coCartesian} fibrations over $X\et$.
For fixed $U \in X\et$, we can characterize
\[ g_U \colon \Ind( \fCoh_Z^{\Aff}(U) ) \to \Ind( \hPhi_Z^{\Aff}(U) ) \]
as the right adjoint to $\Ind(j_U^*)$.

\begin{prop} \label{prop:gluing_right_adjoints}
	\begin{enumerate}
		\item the morphism $g$ preserves Cartesian edges;
		\item the morphism $g$ restricts to a transformation $\int \fCoh_Z^{\Aff} \to \int \hPhi_Z^{\Aff}$ which preserves Cartesian edges.
		\item For fixed $U \in X\et$, the restriction of $g_U$ to $\fCoh_Z^{\Aff}(U) = \Coh^-(Z_U)$ can be identified with the functor $j_{U*} \colon \Coh^-(Z_U) \to \Coh^-(\widehat{Z_U})$, whose existence is guaranteed by \cref{lem:basic_formal_GAGA}.
	\end{enumerate}
\end{prop}

\begin{proof}
	Let $U \in X\et$.
	We start to prove that the functor
	\[ g \colon \Ind(\Coh^-(Z_U)) \to \Ind(\Coh^-(\widehat{Z_U})) , \]
	formally defined as the right adjoint to $\Ind(j^*)$, preserves compact objects.
	Since $U$ is affine, we can apply \cref{lem:pushforward_affine_case} to deduce that the functor
	\[ j^* \colon \Coh^-(\widehat{Z_U}) \to \Coh^-(Z_U) \]
	admits a right adjoint $j_*$.
	Since the functors
	\[ \hat{\imath} \colon \Coh^-(\widehat{Z_U}) \to \Ind(\Coh^-(\widehat{Z_U})), \quad i \colon \Coh^- (Z_U) \to \Ind(\Coh^- (Z_U)) \]
	are fully faithful, we obtain for 
	\begin{align*}
	\Map( \hat{\imath}(\cF), \hat{\imath}( j_*(\cG) ) ) & = \Map( \cF, j_*(\cG) ) = \Map( j^* \cF, \cG ) \\
	& = \Map( i(j^*(\cF)) , i(\cG) ) = \Map( \Ind(j^*)( \hat{\imath}(\cF)), i(\cG) ) \\
	& = \Map( \hat{\imath}(\cF), g(i(\cG)) ) .
	\end{align*}
	Since the essential image of $\hat{\imath}$ generates under colimits $\Ind(\Coh^-(\widehat{Z_U}))$, we conclude that the canonical transformation
	\[ \hat{\imath} \circ j_* \to g \circ i \]
	is an equivalence.
	This implies that $g$ preserves compact objects, and that there is a canonically commutative diagram
	\[ \begin{tikzcd}
	\Coh^-(\widehat{Z_U}) \arrow{d}{\hat{\imath}} & \Coh^-(Z_U) \arrow{l}[swap]{j_*} \arrow{d}{i} \\
	\Ind(\Coh^-(\widehat{Z_U})) & \Ind(\Coh^-(Z_U)) \arrow{l}[swap]{g} .
	\end{tikzcd} \]
	
	Let now $u \colon U \to V$ be an open immersion in $X\et$.
	To say that $g$ takes cartesian edges to cartesian edges means that the square
	\[ \begin{tikzcd}
	\Ind(\Coh^-(\widehat{Z_U})) \arrow{d}{\Ind(\hat{u}^*)} & \Ind(\Coh^-(Z_U)) \arrow{d}{\Ind(u^*)} \arrow{l}[swap]{g_U} \\
	\Ind(\Coh^-(\widehat{Z_V})) & \Ind(\Coh^-(Z_V)) \arrow{l}[swap]{g_V}
	\end{tikzcd} \]
	commutes.
	Observe that the left adjoints $\Ind(j_U^*)$ and $\Ind(j_V^*)$ of $g_U$ and $g_V$ respectively preserve compact objects.
	Thus, $g_U$ and $g_V$ commute with filtered colimits.
	Furthermore, we proved that $g_U$ and $g_V$ preserve compact objects, and the same goes for $\Ind(u^*)$ and $\Ind(\hat{u}^*)$.
	In conclusion, it is enough to prove that the induced square
	\[ \begin{tikzcd}
	\Coh^-(\widehat{Z_U}) \arrow{d}{\hat{u}^*} & \Coh^-(Z_U) \arrow{d}{u^*} \arrow{l}[swap]{j_{U*}} \\
	\Coh^-(\widehat{Z_V}) & \Coh^-(Z_V) \arrow{l}[swap]{j_{V*}}
	\end{tikzcd} \]
	commutes.
	
	Since $Z_U$ and $Z_V$ are affine and $\widehat{Z_U}$ and $\widehat{Z_V}$ are formal affines, using \cref{lem:basic_formal_GAGA} it will be enough to show that the square of commutative rings
	\[ \begin{tikzcd}
	\widehat{A_V} \arrow{r} \arrow{d} & A_V / J_V \arrow{d} \\
	\widehat{A_U} \arrow{r} & A_U / J_U .
	\end{tikzcd} \]
	is actually a \emph{derived} pushout.
	It follows from \cref{lem:etale_sent_to_flat} that the map $\widehat{A_V} \to \widehat{A_U}$ is flat.
	As consequence, it is enough to check that this square is an ordinary pushout, which is clear since $J_U = J_V A_U$.
	
	In conclusion, since we proved that $g_U$ preserves compact objects, we see that $g$ restricts to a morphism of Cartesian fibrations over $X\et$
	\[ \int \fCoh_Z^{\Aff} \to \int \hPhi_Z^{\Aff} \]
	which is fiberwise right adjoint to $j^*$.
	We denote this morphism by $j_*$.
	Finally, in proving that $g$ preserves cartesian edges, we actually proved that $j_*$ has this property.
	The proof is thus complete.
\end{proof}

We can interpret the above proposition by saying that the right adjoints $j_{U*}$ can be assembled together into a natural transformation defined over $X\et$.
In other words, the above proposition is providing all the higher coherences needed to formally prove this claim.
At this point, since $\fCoh_Z^{\Aff}$ and $\hPhi_Z^{\Aff}$ are the restrictions of the sheaves $\fCoh_Z^{\Sch}$ and $\hPhi_Z^{\Sch}$, we see that this produces a well defined natural transformation
\[ j_* \colon \fCoh_Z^{\Sch} \to \hPhi_Z^{\Sch} . \]

\begin{defin}
	Fix $U \in X\et^{\Sch}$.
	We define the $\infty$-category $\cK^{\Sch}(U)$ as the smallest full stable subcategory of $\hPhi_Z^{\Sch}(U) = \Coh^-(\widehat{Z_U})$ which is left complete and contains the essential image of $j_{U*} \colon \Coh^-(Z_U) \to \Coh^-(\widehat{Z_U})$.
\end{defin}

\begin{prop} \label{prop:qcompact_kernel}
	\begin{enumerate}
		\item The assignment $U \mapsto \cK^{\Sch}(U)$ can be promoted to an $\infty$-functor
		\[ \cK^\Sch \colon (X\et^{\Sch})\op \to \Catst , \]
		and the inclusions $\cK^{\Sch}(U) \hookrightarrow \Coh^-(\widehat{Z_U})$ can be promoted to a natural transformation $\alpha \colon \cK^{\Sch} \hookrightarrow \hPhi_Z$.
		\item $\cK^{\Sch}$ is a sheaf for the \'etale topology. In particular, it equals the restriction of $\cK$ to $X\et$.
	\end{enumerate}
\end{prop}

\begin{proof}
	Define a full subsimplicial set
	\[ K \subset \int \hPhi_Z^{\Sch} , \]
	by saying that a vertex $x \in \int \hPhi_Z^{\Sch}$ lying over $U \in X\et^{\Sch}$ belongs to $K$ if and only if the corresponding object $\cF_x \in \Coh^-(\widehat{Z_U})$ belongs to $\cK^{\Sch}(U)$.
	We claim that the composition
	\[ K \hookrightarrow \int \hPhi_Z^{\Sch} \to X\et^{\Sch} \]
	is a Cartesian fibration, and furthermore that $K \hookrightarrow \int \hPhi_Z^{\Sch}$ preserves Cartesian edges.
	Let therefore $f \colon U \to V$ be a morphism in $X\et^{\Sch}$ and let $g \colon Z_U \to Z_V$ be the map induced by base-change.
	Since $j_*$ is a natural transformation, we obtain a commutative square
	\[ \begin{tikzcd}
		\Coh^-(Z_V) \arrow{r}{j_{V*}} \arrow{d}{g^*} & \Coh^-(\widehat{Z_V}) \arrow{d}{\hat{f}^*} \\
		\Coh^-(Z_U) \arrow{r}{j_{U*}} & \Coh^-(\widehat{Z_U}) .
	\end{tikzcd} \]
	Let $\cC$ be the full subcategory of $\Coh^-(\widehat{Z_V})$ spanned by those almost perfect modules $\cF$ such that $\hat{f}^*(\cF) \in \cK^\Sch(U)$.
	Since $\cK^\Sch(U)$ is a stable subcategory of $\Coh^-(Z_U)$ and since $\hat{f}^*$ is an exact functor, we see that $\cC$ is a stable subcategory of $\Coh^-(\widehat{Z_V})$.
	Next, we claim that $\cC$ is left complete.
	Indeed, observe that both $g^*$ and $\hat{f}^*$ are $t$-exact because the map $f$ was étale to start with, and therefore so are the maps $Z_U^{(n)} \to Z_V^{(n)}$ induced by base-change.
	In particular if $\cF \in \Coh^-(\widehat{Z_V})$ is such that $\tau_{\ge n}(\cF) \in \cC$ for every $n \in \mathbb Z$, then
	\[ \hat{f}^*(\tau_{\ge n}(\cF)) \simeq \tau_{\ge n}(\hat{f}^*(\cF)) \in \cK^\Sch(U) . \]
	Since $\cK^\Sch(U)$ is left complete by definition, it follows that $\hat{f}^*(\cF) \in \cK^\Sch(U)$, i.e.\ $\cF \in \cC$.
	Finally, we remark that the commutativity of the above diagram implies that $\cC$ contains the essential image of $j_{V*}$.
	As consequence, we deduce that $\cK^\Sch(V) \subset \cC$.
	This completes the proof of statement (1).
	
	To prove statement (2), we only need to prove that the property of belonging to $\cK^\Sch(U)$ is \'etale local in $U$.
	So fix $\cF \in \Coh^-(\widehat{Z_U})$ and suppose that there exists an affine \'etale cover $\{u_i \colon U_i \to U\}$ such that $\widehat{u_i}^*(\cF) \in \cK^\Sch(U_i)$.
	We let $U^0 \coloneqq \coprod U_i$, $u \colon U^0 \to U$ be the induced total morphism, and $U^\bullet \coloneqq \Cech(u)$ be the \v{C}ech nerve of $u$.
	
	To prove that $\cF \in \cK^\Sch(U)$, it is enough to prove that $\tau_{\ge n}(\cF) \in \cK^{\Sch}(U)$ for every $n \in \mathbb Z$.
	Since the functors $\widehat{u_i}^*$ are $t$-exact, we have $\widehat{u_i}^*(\tau_{\ge n}(\cF)) \simeq \tau_{\ge n}( \widehat{u_i}^*(\cF) )$.
	We can therefore replace $\cF$ with $\tau_{\ge n}(\cF)$, or, equivalently, suppose that $\cF$ is in finite cohomological amplitude from the very beginning.
	Proceeding by induction on the number of non-vanishing cohomology sheaves, we are easily reduced to the case where $\cF \in \Coh^\heartsuit(\widehat{Z_U})$.
	We now remark that the induced functors
	\[ (j_{U_i *}^{(n)})^\heartsuit \colon \Coh^\heartsuit(Z_{U_i}^{(n)}) \to \Coh^\heartsuit(\widehat{Z_{U_i}}) \]
	are fully faithful.
	It follows from \cref{cor:internal_characterisation_kernel_localisation} that $(\cK^\Sch(U_i))^\heartsuit$ coincides precisely with the union of all the essential images of the functors $(j_{U_i*}^{(n)})^\heartsuit$.
	Since $U$ is quasi-compact, we can find a positive integer $n \in \mathbb N$ such that $\widehat{u_i}^*(\cF) \in \mathrm{Im}((j_{U_i*}^{(n)})^\heartsuit)$ for every index $i$.
	Since the functors $(j_{U_i *}^{(n)})^\heartsuit$ are fully faithful, we conclude that $\cF$ determines an element in
	\[ \Coh^\heartsuit(Z_U^{(n)}) \simeq \varprojlim \Coh^\heartsuit(Z_{U^\bullet}^{(n)}) . \]
	It follows that $\cF$ belongs to the essential image of $(j_{U*^{(n)}})^\heartsuit$.
	This proves that $\cF \in \cK^\Sch(U)$.
	The proof is therefore complete.
\end{proof}

\begin{cor}
	There is a canonical equivalence $\cK^\Sch(X) \simeq \cK(X)$.
\end{cor}

Using the natural transformation $\alpha \colon \cK^{\Sch} \hookrightarrow \hPhi_Z$, we can form the cofiber
\[ \oPhi^\circ_Z \coloneqq \cofib(\alpha) \in \PSh_{\Catst}(X\et) . \]
Since $\cK^{\Sch}$ equals the restriction of $\cK$ to $X\et$, we can use \cref{cor:algebraic_universal_property} and the discussion at the beginning of this subsection to deduce that the sheafification of $\oPhi^\circ_Z$ coincides with the restriction of $\hPhi^\circ_Z$ to $X\et$.
In particular, there is a canonical map $\varphi \colon \oPhi^\circ_Z \to \hPhi^\circ_Z$.
With these preparations, \cref{thm:Zariski_computational_tool} can be stated in the following more precise form:

\begin{prop} \label{prop:Zariski_computational_tool}
	The $\infty$-functor $\varphi = \varphi_X \colon \oPhi^\circ_Z(X) \to \hPhi^\circ_Z(X)$ is an equivalence of $\infty$-categories.
\end{prop}

The proof of the previous proposition consists of several steps, that we emphasise in the exposition below.

\subsection{Reduction to the heart} \label{subsubsec:reduction_to_the_heart}

In the proof of \cref{prop:Zariski_computational_tool} a number of explicit constructions involving a certain number of choices is needed.
Since $\infty$-categories behave poorly with respect to this kind of procedures, the first step toward the proof of \cref{prop:Zariski_computational_tool} is to reduce ourselves to a $1$-categorical statement.

\begin{lem} \label{lem:reduction_to_the_heart_I}
	The $\infty$-categories $\oPhi^\circ_Z(X)$ and $\hPhi^\circ_Z(X)$ admit left complete $t$-structures such that the functor $\varphi \colon \oPhi^\circ_Z(X) \to \hPhi^\circ_Z(X)$ is $t$-exact.
\end{lem}

\begin{proof}
	Let us first construct the $t$-structure on $\hPhi^\circ_Z(X)$.
	Choose an affine hypercover $U^\bullet$ of $X$.
	Since $\hPhi^\circ_Z(X)$ is a hypercomplete $\tauet$-sheaf, we have an equivalence
	\[ \hPhi^\circ_Z(X) \simeq \varprojlim \hPhi^\circ_Z(U^\bullet) . \]
	Combining Theorems \ref{thm:etale_descent} and \ref{thm:localization_for_coh_minus} and \cref{prop:qcompact_kernel} we obtain, for every $n \in \mathbf \Delta$, an equivalence
	\[ \hPhi^\circ_Z(U^n) \simeq \Coh^-(\hPsi^\circ_Z(U^n)) \simeq \Coh^-(\widehat{Z_{U^n}}) / \cK^\Sch(U^n) .  \]
	Using the left-hand-side equivalence, we easily see that all the face maps in the cosimplicial object
	\[ \hPhi^\circ_Z(U^\bullet) \simeq \Coh^-(\hPsi^\circ_Z(U^\bullet)) \]
	are $t$-exact.
	In particular, we can use \cref{lem:flat_cosimplicial_t_structure} to induce a pointwise $t$-structure on the totalization of this cosimplicial object.
	That is, we obtain a $t$-structure on
	\[ \hPhi^\circ_Z(X) \simeq \varprojlim \Coh^-(\hPsi^\circ_Z(U^\bullet)) . \]
	
	Let us now turn to the $t$-structure on $\oPhi^\circ_Z(X)$.
	Reasoning as in the proof of \cref{thm:localization_for_coh_minus}, we see that the hypotheses of \cref{cor:quotient_t_structure} are satisfied in our situation.
	It follows that the quotient
	\[ \oPhi^\circ_Z(X) \coloneqq \Coh^-(\hZ) / \cK^\Sch(X) \]
	inherits a $t$-structure from the stable $\infty$-category $\Coh^-(\hZ)$, and furthermore that the quotient map
	\[ L \colon \Coh^-(\hZ) \to \oPhi^\circ_Z(X) \]
	is $t$-exact.
	To prove that
	\[ \oPhi^\circ_Z(X) \to \hPhi^\circ_Z(X) \]
	is $t$-exact, we only need to check that for every $n \in \mathbf \Delta$ the induced map
	\[ \oPhi^\circ_Z(X) \to \hPhi^\circ_Z(U^n) \]
	is $t$-exact.
	Since $L$ is essentially surjective and $t$-exact, it is furthermore enough to prove that the induced map
	\[ \Coh^-(\hZ) \to \hPhi^\circ_Z(U^n) \]
	is $t$-exact.
	Recall once more that \cref{thm:etale_descent} guarantees that $\oPhi^\circ_Z(U^n) \simeq \hPhi^\circ_Z(U^n)$.
	As consequence, we see that the above map fits in the following commutative square
	\[ \begin{tikzcd}
		\Coh^-(\hZ) \arrow{r} \arrow{d} & \Coh^-(\widehat{Z_{U^n}}) \arrow{d} \\
		\oPhi^\circ_Z(X) \arrow{r} & \oPhi^\circ_Z(U^n) .
	\end{tikzcd} \]
	The conclusion now follows from the fact that $\Coh^-(\hZ) \to \Coh^-(Z_{U^n})$ are $t$-exact (being pullbacks along formally \'etale morphisms) and from the fact that the map \[ \Coh^-(\widehat{Z_{U^n}}) \to \oPhi^\circ_Z(U^n) \]
	is $t$-exact (either by \cref{cor:quotient_t_structure} or by observing that under the identification $\oPhi^\circ_Z(U^n) \simeq \Coh^-(\hPsi^\circ_Z(U^n))$ the previous map becomes the pullback along an open immersion).
\end{proof}

\begin{lem} \label{lem:reduction_to_the_heart_II}
	Let $f \colon \cC \to \cD$ be an exact functor between stable $\infty$-categories.
	Suppose that $\cC$ and $\cD$ are equipped with $t$-structures which are left complete and right bounded and that, with respect to them, $f$ is $t$-exact.
	Then the following statements are equivalent:
	\begin{enumerate}
		\item the functor $f$ is an equivalence;
		\item the induced functor $f^\heartsuit \colon \cC^\heartsuit \to \cD^\heartsuit$ is an equivalence of abelian categories.
	\end{enumerate}
\end{lem}

\begin{proof}
	Suppose first that $f$ is an equivalence and let $g$ be a quasi-inverse for $f$.
	Since $f$ is $t$-exact, so must be $g$.
	Thus, $f$ and $g$ restrict to an adjunction at the level of the hearts $f^\heartsuit \colon \cC^\heartsuit \leftrightarrows \cD^\heartsuit \colon g^\heartsuit$.
	It is straightforward to check that the unit and the counit of this adjunction are equivalences, and thus that $f^\heartsuit$ is an equivalence itself.
	
	Suppose now that $f^\heartsuit$ is an equivalence.
	Let us first prove that $f$ is fully faithful.
	Let $X, Y \in \cC$ and consider the map
	\[ \psi_{X,Y} \colon \stMap_{\cC}(X,Y) \to \stMap_{\cD}(f(X), f(Y)) . \]
	It is enough to prove that $\psi_{X,Y}$ is an equivalence for every $X$ and $Y$.
	Fix $X \in \cC$ and define $\cC_X$ as the full subcategory of $\cC$ spanned by those $Y \in \cC$ such that $\psi_{X,Y}$ is an equivalence.
	Since the functors $\stMap_{\cC}(X,-)$ and $\stMap_{\cD}(f(X),-)$ are exact functors of stable $\infty$-categories, we see that $\cC_X$ is a stable subcategory of $\cC$.
	We claim furthermore that $\cC_X$ is left complete.
	Indeed, suppose that $Y \in \cC$ is an object such that $\tau_{\ge n}(Y) \in \cC_X$ for every $n \in \mathbb Z$.
	Then, since the $t$-structure on $\cC$ is left complete, we have
	\[ Y \simeq \varprojlim_n \tau_{\ge n}(Y) , \]
	while the fact that $f$ is $t$-exact and the left completeness of the $t$-structure on $\cD$ imply that
	\[ f(Y) \simeq \varprojlim_n \tau_{\ge n}(f(Y)) \simeq \varprojlim f( \tau_{\ge n}(Y) ) . \]
	In particular,
	\[ \psi_{X,Y} \simeq \varprojlim_n \psi_{X, \tau_{\ge n}(Y)} \]
	and therefore it is an equivalence.
	Since the $t$-structure on $\cC$ is right bounded, we see that to prove that $\cC_X = \cC$ it is in fact enough to prove that $\cC^\heartsuit \subseteq \cC_X$.
	
	Fix therefore $Y \in \cC^\heartsuit$ and define $\cC_Y$ as the full subcategory of $\cC$ spanned by those $Z \in \cC$ such that $\psi_{Z,Y}$ is an equivalence.
	We have to prove that $X \in \cC_Y$.
	We will prove that $\cC_Y = \cC$.
	Reasoning as above, we see that $\cC_Y$ is a full stable subcategory of $\cC$.
	Moreover, since $f^\heartsuit$ is an equivalence, we know that $\cC^\heartsuit \subseteq \cC_Y$.
	From these two facts, we conclude by induction on the number of non-vanishing cohomology groups, every left bounded object of $\cC$ belongs to $\cC_Y$.
	Fix now an arbitrary $Z \in \cC$.
	In order to prove that $\psi_{Z,Y}$ is an equivalence, it is enough to prove that $\rH^i(\psi_{Z,Y})$ is an equivalence for every $i \in \mathbb Z$.
	Since $Z$ is arbitrary, it is moreover enough to consider the case $i = 0$.
	We now observe that, since $Y \in \cC^\heartsuit$,
	\[ \rH^0 \stMap_\cC(Z, Y) \simeq \pi_0 \left( \tau_{\le 0} \stMap_{\cC}(Z,Y) \right) \simeq \pi_0 \Map_{\cC}(Z,Y) \simeq \pi_0 \Map_{\cC}(\tau_{\ge 0}(Z), Y) . \]
	Since $f$ is $t$-exact, the same argument proves that
	\[ \rH^0 \stMap_{\cD}(f(Z), f(Y)) \simeq \pi_0 \Map_{\cC}(\tau_{\ge 0}(f(Z)), f(Y)) . \]
	In this way, we are reduced to the case where $Z$ is left bounded, which had already been dealt with.
	It follows that $\psi_{Z,Y}$ is an equivalence and thus that $Z \in \cC_Y$.
	Thus $\cC_Y = \cC$, hence $X \in \cC_Y$.
	Equivalently, $Y \in \cC_X$, so that $\cC_X = \cC$.
	We conclude in this way that $f$ is fully faithful.
	
	Let us now prove that $f$ is essentially surjective.
	Let $\cD'$ be the full subcategory of $\cD$ spanned by the essential image of $f$.
	Since $f$ is fully faithful, we see that $\cC \simeq \cD'$ and that $\cD'$ is a stable subcategory of $\cD$.
	Since $f^\heartsuit$ is an equivalence, we know that $\cD^\heartsuit \subset \cD'$.
	Moreover, using the fact that $\cD'$ is a full stable subcategory of $\cD$, we deduce that every bounded object of $\cD$ belongs to $\cD'$.
	Let now $Y \in \cD$ be an arbitrary object.
	Since the $t$-structure on $\cD$ is left complete and right bounded, we can write
	\[ Y \simeq \varprojlim_n \tau_{\ge n}(Y) , \]
	and every object $\tau_{\ge n}(Y)$ is bounded.
	Thus, it belongs to $\cD'$.
	Since $f$ is fully faithful and $t$-exact, we can find a diagram $F \colon \mathbb Z \to \cC$ with the following properties:
	\begin{enumerate}
		\item the composition $f \circ F$ equals the diagram $n \mapsto \tau_{\ge n}(Y) \in \cD$;
		\item every object $F(n)$ is right bounded and belongs to $\cC^{\ge n}$;
		\item for every $n \in \mathbb Z$, the $F(n) \to F(n+1)$ induces a canonical equivalence $\tau_{\ge n+1}(F(n)) \simeq F(n+1)$.
	\end{enumerate}
	Since the $t$-structure on $\cC$ is left complete, we conclude that the inverse limit
	\[ X \coloneqq \varprojlim_n F(n) \]
	exists in $\cC$ and that $\tau_{\ge n}(X) \simeq F(n)$.
	Using once more the fact that $f$ is $t$-exact and that the $t$-structure on $\cD$ is left complete, we deduce that
	\[ f(X) \simeq \varprojlim f(F(n)) \simeq \varprojlim \tau_{\ge n}(Y) \simeq Y. \]
	Thus, $f$ is essentially surjective.
\end{proof}

Using Lemmas \ref{lem:reduction_to_the_heart_I} and \ref{lem:reduction_to_the_heart_II}, we can now reduce the proof of \cref{prop:Zariski_computational_tool} to the proof that the induced functor
\[ \varphi^\heartsuit \colon (\oPhi^\circ_Z(X))^\heartsuit \to (\hPhi^\circ_Z(X))^\heartsuit \]
is an equivalence of abelian categories.
This will be the content of the following section.

\subsection{Base change and formal models}

Before approaching the proof of \cref{thm:Zariski_computational_tool}, we need to carry out a more careful analysis of our functor $\oPhi^\circ_Z$.
More precisely, we prove that $\Ind \circ \oPhi^\circ_Z$ has base change.
We use this property to produce canonical lifts of elements in $\oPhi^\circ_Z$.
We start with two elementary statements concerning (stable) $\infty$-categories.

\begin{lem} \label{lem:categorical_quotient_right_adjointable}
	Let $\cC$ and $\cD$ be two stable $\infty$-categories and let $F \colon \cC \to \cD$ be an exact functor between them.
	Let $i_\cC \colon \cK_\cC \subset \cC$ and $i_\cD \colon \cK_\cD \subset \cD$ be the inclusions of full stable subcategories and suppose that $F$ restricts to a functor $F_\cK \colon \cK_\cC \to \cK_\cD$.
	We thus obtain a ladder of commutative squares in $\Catst$
	\[ \begin{tikzcd}
		\cK_\cC \arrow{r}{i_\cC} \arrow{d}{F_\cK} & \cC \arrow{r}{L_\cC} \arrow{d}{F} & \cC / \cK_\cC \arrow{d}{\oF} \\
		\cK_\cD \arrow{r}{i_\cD} & \cD \arrow{r}{L_\cD} & \cK / \cK_\cD ,
	\end{tikzcd} \]
	where $L_\cC$ and $L_\cD$ are the quotient maps and $\oF$ is the functor induced by universal property.
	Suppose that:
	\begin{enumerate}
		\item the inclusions $i_\cC$ and $i_\cD$ admit right adjoints $R_\cC \colon \cC \to \cK_\cC$ and $R_\cD \colon \cD \to \cK_\cD$, respectively;
		\item the quotient maps $L_\cC$ and $L_\cD$ admit fully faithful right adjoints $j_\cC \colon \cC/\cK_\cC \hookrightarrow \cC$ and $j_\cD \colon \cD / \cK_\cD \hookrightarrow \cD$, respectively;
		\item $\cK_\cC$ is the kernel of $L_\cC$ and $\cK_\cD$ is the kernel of $L_\cD$.
	\end{enumerate}
	Then, if the upper square is right adjointable then so is the bottom one.
\end{lem}

\begin{proof}
	The assumptions of \cref{lem:Verdier_fiber_sequence} are satisfied.
	In particular, for every object $X \in \cC$ we have a fiber sequence:
	\[ i_\cC R_\cC(X) \to X \to j_\cC L_\cC(X) . \]
	Applying $F$ we obtain a commutative diagram whose lines are fiber sequences in $\cD$:
	\[ \begin{tikzcd}
		F i_\cC R_\cC(X) \arrow{r} \arrow{d} & F(X) \arrow{r} \arrow[-, double equal sign distance]{dd} & F j_\cC L_\cC(X) \arrow{dd} \\
		i_\cD F_\cK R_\cC(X) \arrow{d} \\
		i_\cD R_\cD F(X) \arrow{r} & F(X) \arrow{r} & j_\cD L_\cD F(X) .
	\end{tikzcd} \]
	The morphism
	\[ F i_\cC R_\cC(X) \to i_\cD F_\cK R_\cC(X) \to i_\cD R_\cD F(X) \]
	is an equivalence because the top square is commutative and because it is right adjointable.
	It follows that the morphism $F j_\cC L_\cC(X) \to j_\cD L_\cD F(X)$ is an equivalence as well.
	On the other side, the commutativity of the bottom square implies that
	\[ j_\cD L_\cD F(X) \simeq j_\cD \oF L_\cC(X) . \]
	Since $L_\cC$ is essentially surjective, it follows that the natural transformation $F \circ j_\cC \to j_\cD \circ L_\cD$ is an equivalence, thus completing the proof.
\end{proof}

\begin{lem} \label{lem:Ind_Beck_Chevalley}
	Let
	\[ \begin{tikzcd}
		\cK_\cC \arrow{d}{F_\cK} \arrow{r}{i_\cC} & \cC \arrow{d}{F} \\
		\cK_\cD \arrow{r}{i_\cD} & \cD
	\end{tikzcd} \]
	be a commutative diagram of $\infty$-categories and exact functors between them.
	Suppose that $i_\cC$ and $i_\cD$ are the inclusion of full stable subcategories.
	If the above square is right adjointable, then so is the square obtained by applying the functor $\Ind \colon \Cat_\infty \to \LPromega$.
\end{lem}

\begin{proof}
	Let $R_\cC \colon \cC \to \cK_\cC$ and $R_\cD \colon \cD \to \cK_\cD$ be right adjoints for $i_\cC$ and $i_\cD$, respectively.
	First of all, we observe that $\Ind(R_\cC)$ and $\Ind(R_\cD)$ are right adjoints to $\Ind(i_\cC)$ and $\Ind(i_\cD)$, respectively.
	At this point it is straightforward to check that if $\eta \colon F_\cK \circ R_\cC \to R_\cD \circ F$ denotes the Beck-Chevalley transformation, then $\Ind(\eta)$ is the Beck-Chevalley transformation associated to the commutative square
	\[ \begin{tikzcd}
		\Ind(\cK_\cC) \arrow{d}{\Ind(F_\cK)} \arrow{r}{\Ind(i_\cC)} & \Ind(\cK_\cD) \arrow{d}{\Ind(F)} \\
		\Ind(\cK_\cD) \arrow{r}{\Ind(\cD)} & \Ind(\cD) .
	\end{tikzcd} \]
	In particular, since $\eta$ was an equivalence, we conclude that $\Ind(\eta)$ is an equivalence as well, and thus that the above square is right adjointable.
\end{proof}

We now specialize to our setting.
We have the following key fact:

\begin{prop} \label{prop:heart_kernel_right_adjointable}
	Let $f \colon U \to X$ be an open immersion of noetherian, quasi-compact and quasi-separated schemes.
	Let $Z \hookrightarrow X$ be a closed subscheme and let $Z_U \coloneqq U \times_X Z$ be its base change.
	Then the commutative diagram
	\[ \begin{tikzcd}
		\cK^\Sch(X)^\heartsuit \arrow{d}{\hf^*} \arrow{r}{i_X} & \Cohh(\hZ) \arrow{d}{\hf^*} \\
		\cK^\Sch(U)^\heartsuit \arrow{r}{i_U} & \Cohh(\widehat{Z_U})
	\end{tikzcd} \]
	is right adjointable.
\end{prop}

\begin{proof}
	The first task we have to tackle is to construct right adjoints to the inclusions $i_X$ and $i_U$.
	We remark that the functor $X\et \ni U \mapsto \Cohh( \widehat{Z_U} ) \in \Catst$ is a sheaf for the \'etale topology.
	On the other side, we proved in \cref{prop:qcompact_kernel} that the same is true for the functor $X\et \ni U \mapsto \cK^\Sch(U) \in \Catst$.
	Reasoning as in the proof of \cref{prop:gluing_right_adjoints}, we see that it is enough to prove that the above square is right adjointable in the case where both $X=\Spec(A)$ and $U= \Spec(A_U)$ are affine and $U$ is a \emph{principal} open subset of $X$.
	In this case, we can use \cref{lem:basic_formal_GAGA} to get canonical identifications
	\[ \Cohh(\hZ) \simeq \Cohh(\hA), \qquad \Cohh(\widehat{Z_U}) \simeq \Cohh(\widehat{A_U}) . \]
	In particular, since both $\hA$ and $\widehat{A_U}$ are noetherian rings, we can apply \cref{cor:internal_characterisation_kernel_localisation} to identify $\cK^\Sch(X)^\heartsuit$ (resp.\ $\cK^\Sch(U)^\heartsuit$) with the full subcategory of $\Cohh(\hA)$ (resp.\ of $\Cohh(\widehat{A_U})$) spanned by those modules $M$ that are annihilated by some power of the ideal $J$ defining $Z$ inside $\hA$ (resp.\ of the ideal $J_U$ defining $Z_U$ inside $\widehat{A_U}$).
	
	In this case it is clear that the right adjoints for $i_X$ and $i_U$ exist and are given by $\Gamma_J$ and by $\Gamma_{J_U}$, the functor of sections with support in $Z$ and in $Z_U$, respectively.
	Let us also remark that under the equivalence
	\[ \Cohh(\hA) \simeq \varprojlim_n \Cohh(A / J^n) \]
	we can identify $\Gamma_J(M)$ with the sequence of modules $\{\Gamma_J(M)_n\}$ defined by
	\[ \Gamma_J(M)_n \coloneqq \{m \in M / J^n \mid J^i m = 0 \text{ for some } 1 \le i \le n - 1\} . \]
	
	To complete the proof we have to show that for every $M \in \Cohh(\hA)$, the natural map
	\[ \eta \colon \Gamma_J(M) \otimes_{\hA} \widehat{A_U} \to \Gamma_{J_U}( M \otimes_{\hA} \widehat{A_U} ) \]
	is an equivalence.
	By exploiting the equivalence
	\[ \Cohh(\widehat{A_U}) \simeq \varprojlim_n \Cohh(A_U / J_U^n) , \]
	it is enough to check that the map $\eta$ becomes an isomorphism after modding out by $J_U^n$, for every $n$.
	In other words, we have to check that
	\[ \eta_n \colon \Gamma_J(M)_n \otimes_{A / J^n} (A_U / J_U^n) \to \Gamma_{J_U}( M_n \otimes_{A / J^n} (A_U / J_U^n)) \]
	is an equivalence.
	Observe that, since $A_U / J_U^n \simeq A_U \otimes_A (A / J^n)$, the map $A / J^n \to A_U / J_U^n$ is an open Zariski immersion.
	In particular, it is flat, and therefore both $\Gamma_J(M)_n \otimes_{A / J^n} (A_U / J_U^n)$ and $\Gamma_{J_U}( M_n \otimes_{A / J^n} (A_U / J_U^n))$ are submodules of $M_n \otimes_{A / J^n} (A_U / J_U^n)$, and $\eta_n$ is compatible with the natural inclusions.
	As a consequence, we see that $\eta_n$ is a monomorphism.
	Let us prove that it is an epimorphism as well.
	Since $U$ is a principal open subset of $X$, we can write $A_U/J_U^n \simeq A/J^{n}[g\inv]$ for some $g \in A$.
	Thus, we can write any element in $\Gamma_{J_U}(M_n \otimes_{A / J^n} (A_U / J_U^n))$ in the form $\frac{m}{g^l}$, where $m \in M/J^n$.
	Furthermore, we know that for some $1 \le k \le n - 1$ and every $a \in J_U^k$ we have
	\[ a \cdot \frac{m}{g^l} = 0 \Longleftrightarrow g^{l+h} a m = 0 . \]
	Since the ideal $J_U^k$ is finitely generated, we can choose $h \ge 0$ in such a way that the above relation holds true for every $a \in J_U^k$.
	In particular, we conclude that $g^{l+h}m \in \Gamma_J(M)_n$.
	Since
	\[ \frac{m}{g^l} = g^{l+h}m \cdot \frac{1}{g^{2l+h}} \in \Gamma_J(M)_n \otimes_{A/J^n} (A_U / J_U^n) \subset M_n[g\inv] , \]
	we conclude that $\eta_n$ is also surjective.
	The proof is thus complete.
\end{proof}

In particular, by combining \cref{prop:heart_kernel_right_adjointable} and \cref{lem:Ind_Beck_Chevalley} we obtain that for any open immersion $U \to X$, the diagram
\[ \begin{tikzcd}
	\Ind(\cK^\Sch(X)^\heartsuit) \arrow{d}{\Ind(\hf^*)} \arrow{r}{\Ind(i_X)} & \Ind(\Cohh(\hZ)) \arrow{d}{\Ind(\hf^*)} \\
	\Ind(\cK^\Sch(U)^\heartsuit) \arrow{r}{\Ind(i_U)} & \Ind(\Cohh(\widehat{Z_U}))
\end{tikzcd} \]
is also right adjointable.
Observe now that all the categories appearing in the above diagram are Grothendieck abelian categories.
In particular, we can pass to the associated left bounded derived $\infty$-categories $\rD^+$.
Since $\Ind(\hf^*)$ is both left and right exact, we can derive it on both sides, and thus we see that the resulting diagram
\[ \begin{tikzcd}
	\rD^+ ( \Ind(\cK^\Sch(X)^\heartsuit) ) \arrow{r} \arrow{d} & \rD^+ ( \Ind(\Cohh(X)) ) \arrow{d} \\
	\rD^+ ( \Ind(\cK^\Sch(U)^\heartsuit) ) \arrow{r} & \rD^+ ( \Ind(\Cohh(U)) )
\end{tikzcd} \]
is right adjointable as well.
Observe furthermore that since the natural $t$-structures on these categories is right complete, since $\Ind(i_X)$ is $t$-exact and since it is fully faithful on the elements of the heart, the same proof of \cref{lem:reduction_to_the_heart_II} shows that the functors
\[ \rD^+( \Ind(\cK^\Sch(X)^\heartsuit) ) \to \rD^+( \Ind(\Cohh(\hZ)) ) , \quad \rD^+( \Ind(\cK^\Sch(U)^\heartsuit) ) \to \rD^+( \Ind(\Cohh(\widehat{Z_U})) ) \]
are fully faithful.
We let $\cD_X$ and $\cD_U$ be the Verdier quotients in $\Catst$
\[ \cD_X \coloneqq \rD^+( \Ind(\Cohh(\hZ)) ) / \rD^+(\Ind(\cK^\Sch(X)^\heartsuit)) , \quad \cD_U \coloneqq \rD^+(\Ind(\Cohh(\widehat{Z_U}))) / \rD^+( \Ind(\cK^\Sch(U)^\heartsuit) ) . \]
Using \cref{cor:quotient_t_structure}, we can equip $\cD_X$ and $\cD_U$ with $t$-structures such that:
\begin{enumerate}
	\item the functors $\rD^+( \Ind(\Cohh(\hZ))) \to \cD_X$ and $\rD(\Ind(\Cohh(\widehat{Z_U}))) \to \cD_U$ are $t$-exact;
	\item there are canonical equivalences
	\begin{gather*}
	\cD_X^\heartsuit \simeq (\rD^+(\Ind(\Cohh(\hZ))))^\heartsuit / (\rD^+( \Ind(\cK^\Sch(X)^\heartsuit) )  )^\heartsuit \simeq \Ind(\Cohh(\hZ)) / \Ind(\cK^\Sch(X)^\heartsuit) \simeq \oPhi^\circ_Z(X)^\heartsuit , \\
	\cD_U^\heartsuit \simeq (\rD^+(\Ind(\Cohh(\widehat{Z_U}))))^\heartsuit / (\rD^+( \Ind(\cK^\Sch(U)^\heartsuit) )  )^\heartsuit \simeq \Ind(\Cohh(\widehat{Z_U})) / \Ind(\cK^\Sch(U)^\heartsuit) \simeq \oPhi^\circ_Z(U)^\heartsuit .
	\end{gather*}
\end{enumerate}

We have the following important technical lemma:

\begin{lem} \label{lem:checking_hypotheses_for_D+}
	The Verdier quotient $\cD_X$ can be identified with the full subcategory of $\rD^+(\Ind(\Cohh(\hZ)))$ spanned by the objects that are right orthogonal to $\rD^+(\Ind(\cK^\Sch(X)^\heartsuit))$.
	Furthermore, $\rD^+(\Ind(\cK^\Sch(X)^\heartsuit))$ coincides with the kernel of the quotient map
	\[ L^+ \colon \rD^+(\Ind(\Cohh(\hZ)) \to \cD_X . \]
\end{lem}

\begin{proof}
	Let $\cD_X'$ be the full subcategory of $\rD^+(\Ind(\Cohh(\hZ)))$ spanned by the objects that are right orthogonal to $\rD^+(\Ind(\cK^\Sch(X)^\heartsuit))$.
	We start by proving that the fully faithful inclusion $\cD_X' \hookrightarrow \rD^+(\Ind(\Cohh(\hZ)))$ has a left adjoint.
	
	Let $\cF \in \rD^+(\Ind(\Cohh(\hZ)))$ and let
	\[ \cG \coloneqq \cofib( i \rR \Gamma_\cJ(\cF) \to \cF ) . \]
	Applying $\rR \Gamma_\cJ$ we obtain the following sequence in $\rD^+(\Ind(\cK^\Sch(X)^\heartsuit))$:
	\[ \rR \Gamma_\cJ i \rR \Gamma_\cJ(\cF) \to \rR \Gamma_\cJ(\cF) \to \rR \Gamma_\cJ(\cG) . \]
	Since $i$ is fully faithful, we see that the first morphism is an equivalence.
	In particular, $\rR \Gamma_\cJ(\cG) \simeq 0$.
	Now, if $\cH \in \rD^+ (\Ind(\cK^\Sch(X)^\heartsuit))$ is any element, we have
	\[ \Map_{\rD^+(\Ind(\Cohh(X)))}(i(\cH), \cG) \simeq \Map_{\rD^+( \Ind(\cK^\Sch(X)^\heartsuit))}(\cH, \rR \Gamma_\cJ \cG) \simeq 0 . \]
	This shows that $\cG$ is a localization of $\cF$ with respect to $\cD_X'$.
	In other words, $\cD_X' \hookrightarrow \rD^+(\Ind(\Cohh(\hZ)))$ has a left adjoint.
	It follows from \cite[5.2.7.12]{HTT} that $\cD_X'$ is the quotient of $\rD^+(\Ind(\Cohh(\hZ)))$ by $\rD^+(\Ind(\cK^\Sch(X)^\heartsuit))$.
	
	We are left to prove that $\rD^+(\Ind(\cK^\Sch(X)^\heartsuit))$ coincides with the kernel of $L^+$.
	It is clear that
	\[ \rD^+(\Ind(\cK^\Sch(X)^\heartsuit)) \subset \ker(L^+) . \]
	Let now $\cF \in \ker(L^+)$.
	On both $\rD^+(\Ind(\Cohh(\hZ)))$ and $\rD^+(\Ind(\cK^\Sch(X)^\heartsuit))$ there are canonical $t$-structures which are right complete and left bounded.
	Moreover, the inclusion $i$ is $t$-exact.
	Since $L^+$ is a left adjoint, it is therefore enough to prove that
	\[ \ker(L^+) \cap \rD^b(\Ind(\Cohh(\hZ))) \subset \rD^+(\Ind(\cK^\Sch(X)^\heartsuit)) . \]
	Proceeding by induction on the number of non-vanishing cohomologies, we are therefore reduced to prove that
	\[ \ker(L^+) \cap \Ind(\Cohh(\hZ)) \subset \rD^+(\Ind(\cK^\Sch(X)))^\heartsuit . \]
	Using once again the fact that $L^+$ is a left adjoint, we can further limit ourselves to check that
	\[ \ker(L^+) \cap \Cohh(\hZ) \subset \rD^+(\Ind(\cK^\Sch(X)^\heartsuit)) . \]
	In turn, since $L^+$ is $t$-exact, it is enough to prove that
	\[ \ker(L) \cap \Cohh(\hZ) \subset \cK^\Sch(X)^\heartsuit . \]
	Let $\cF \in \ker(L) \cap \Cohh(\hZ)$ and let $f \colon U \subseteq X$ be an affine open subscheme.
	Recall that the diagram
	\[ \begin{tikzcd}
		\Cohh(\hZ) \arrow{d}{L_X} \arrow{r}{\hf^*} & \Cohh(\widehat{Z_U}) \arrow{d}{L_U} \\
		\oPhi^\circ_Z(X) \arrow{r}{\overline{f}^*} & \oPhi^\circ_Z(U)
	\end{tikzcd} \]
	commutes.
	In particular, $L_U( \hf^*(\cF) ) \simeq 0$.
	Since $U$ is affine, it follows from the very definition of $\cK^\Sch(X)^\heartsuit$ (cf.\ \cref{prop:internal_characterisation_kernel_stacky}) that $\hf^*(\cF) \in \cK^\Sch(U)^\heartsuit$.
	Since \cref{prop:qcompact_kernel} shows that the assignment $U \mapsto \cK^\Sch(U)^\heartsuit$ is a sheaf, we conclude that $\cF \in \cK^\Sch(X)^\heartsuit$.
	This completes the proof.
\end{proof}

Using \cref{lem:checking_hypotheses_for_D+}, we see that the hypotheses of \cref{lem:categorical_quotient_right_adjointable} are satisfied, and we obtain that the diagram
\[ \begin{tikzcd}
	\rD^+( \Ind(\Cohh(\hZ)) ) \arrow{r} \arrow{d} & \cD_X \arrow{d} \\
	\rD^+(\Ind(\Cohh(\widehat{Z_U}))) \arrow{r} & \cD_U 
\end{tikzcd} \]
is right adjointable.
We finally get the first main result of this section:

\begin{cor} \label{cor:PFN_right_adjointable}
	Let $f \colon U \to X$ be an open immersion of noetherian, quasi-compact and quasi-separated schemes.
	Let $Z \hookrightarrow X$ be a closed subscheme and let $Z_U \coloneqq U \times_X Z$ be its base-change.
	Then the commutative diagram
	\[ \begin{tikzcd}
		\Ind(\Cohh(\hZ)) \arrow{d}{\Ind(\hf^*)} \arrow{r}{L_X} & \Ind(\oPhi^\circ_Z(X)^\heartsuit) \arrow{d}{\Ind(\overline{f}^*)} \\
		\Ind(\Cohh(\widehat{Z_U})) \arrow{r}{L_U} & \Ind(\oPhi^\circ_Z(U)^\heartsuit)
	\end{tikzcd} \]
	is right adjointable.
\end{cor}

\begin{proof}
	Let $j_X \colon \cD_X \hookrightarrow \rD^+(\Ind(\Cohh(\hZ)))$ and $j_U \colon \cD_U \hookrightarrow \rD^+(\Ind(\Cohh(\widehat{Z_U})))$ be the right adjoint to the natural quotient maps.
	Then, we argued above that the natural transformation
	\[ \rD^+(\hf^*) \circ j_X \to j_U \circ \rD^+(\overline{f}^*) \]
	is an equivalence.
	Since both $\rD^+(\hf^*)$ and $\rD^+(\overline{f}^*)$ are $t$-exact, it follows that
	\[ \rH^0 \circ \rD^+(\hf^*) \circ j_X \simeq \rD^+(\hf^*) \circ \rH^0 \circ j_X . \]
	Since $\rH^0 \circ j_X$ and $\rH^0 \circ j_U$ are the right adjoints to $L_X$ and $L_U$, the proof is complete.
\end{proof}

The above result plays an essential role in the proof of \cref{thm:Zariski_computational_tool}.
For example, it allows us to choose formal models for objects in $\oPhi_Z^\circ(X)$ in an ``almost functorial'' way, in the following precise sense:

\begin{cor} \label{cor:almost_functorial_formal_models}
	Let $X$ be a noetherian quasi-compact and quasi-separated scheme and let $Z \hookrightarrow X$ be a closed subscheme.
	For any open subscheme $U \subset X$ we denote by
	\[ i_U \colon (\oPhi_Z^\circ(U))^\heartsuit \hookrightarrow \Ind((\oPhi_Z^\circ(U))^\heartsuit) \]
	the canonical inclusion and by
	\[ j_{U*} \colon \Ind((\oPhi_Z^\circ(U))^\heartsuit) \hookrightarrow \Ind(\Cohh(\widehat{Z_U})) \]
	the right adjoint to the quotient map
	\[ L \colon \Ind(\Cohh(\widehat{Z_U})) \to \Ind((\oPhi_Z^\circ(U))^\heartsuit) \simeq \Ind(\Cohh(\widehat{Z_U})) / \Ind((\cK^\Sch(U))^\heartsuit) . \]
	
	Suppose that there exists a finite affine Zariski cover $\{U_i\}$ such that the canonical map
	\[ \overline{u}^* \colon (\oPhi_Z^\circ(X))^\heartsuit \to (\oPhi_Z^\circ(U))^\heartsuit \]
	is conservative, where $U \coloneqq \coprod U_i$ and $u \colon U \to X$ is the induced total morphism.
	Then for any $\cF \in (\oPhi_Z^\circ(X))^\heartsuit$ there exists a $\cG \in \Cohh(\hZ)$ equipped with a monomorphism
	\[ \cG \hookrightarrow j_{X*}(\cF) \in \Ind(\Cohh(\hZ)) \]
	and such that the induced composition
	\[ L(\cG) \to L(j_{X*}(\cF)) \simeq \cF \]
	is an equivalence.
\end{cor}

\begin{proof}
	Pick $\cF \in (\oPhi_Z^\circ(X))^\heartsuit$.
	The object $j_{X*}(\cF) \in \Ind(\Cohh(\hZ))$ can be written as filtered colimit
	\[ j_{X*}(\cF) \simeq \colim_{\alpha \in I} \cG_\alpha , \]
	where $\cG_\alpha \in \Cohh(\hZ)$.
	Since $\Ind(\Cohh(\hZ))$ is an abelian category, we can consider for any morphism $\cG_\alpha \to j_{X*}(\cF)$ its epi-mono factorization:
	\[ \cG_\alpha \twoheadrightarrow \cG_\alpha' \hookrightarrow j_{X*}(\cF) . \]
	Observe that since $\hA$ is noetherian, the submodule
	\[ \ker(\cG_\alpha \to j_{X*}(\cF)) \]
	of $\cG_\alpha$ must be coherent. In particular, it follows that $\cG_\alpha' \in \Cohh(\hZ)$ for every $\alpha \in I$.
	Moreover, the universal property of the image shows that the assignment
	\[ I \ni \alpha \mapsto \cG_\alpha' \in \Cohh(\hZ) \]
	can be promoted to a functor $I \to \Cohh(\hZ)$.
	Passing to colimits in $\Ind(\Cohh(\hZ))$, we thus obtain maps
	\[ \colim_{\alpha \in I} \cG_\alpha \xrightarrow{\pi} \colim \cG_\alpha' \xrightarrow{\varphi} j_{X*}(\cF) . \]
	Since colimits commute with colimits, it follows that they also preserve epimorphisms.
	In particular, the map $\pi$ is an epimorphism.
	Furthermore, the composition $\varphi \circ \pi$ is an isomorphism.
	It follows that $\pi$ must be a monomorphism as well.
	Since $\Ind(\Cohh(\hZ))$ is an abelian category, it is balanced: as a consequence $\pi$ is an isomorphism.
	At this point, the two out of three property for isomorphisms implies that $\varphi$ must be an isomorphism as well.
	In conclusion, we proved that $j_{X*}(\cF)$ is a filtered colimits of subobjects.
	In other words, we can assume from the very start that all the maps $\cG_\alpha \to j_{X*}(\cF)$ are monomorphisms.
	We are thus left to prove that we can choose an index $\alpha \in I$ in such a way that the induced map
	\[ L(\cG_\alpha) \to L(j_{X*}(\cF)) \simeq \cF \]
	is an isomorphism.
	
	Suppose first that $X$ is affine, say $X = \Spec(A)$.
	Then $\hZ = \Spf(\hA)$, and therefore
	\[ \Ind(\Cohh(\hZ)) = \hA \Modh . \]
	Furthermore,
	\[ (\oPhi_Z^\circ(X))^\heartsuit = \Cohh( \hPsi_Z^\circ ) = \Cohh( \Spec(\hA) \smallsetminus Z ) \]
	and
	\[ \Ind((\oPhi_Z^\circ(X))^\heartsuit) = \QCohh(\Spec(\hA) \smallsetminus Z) . \]
	Under these identifications, the map $L$ simply becomes the restriction to the open subscheme $\Spec(\hA) \smallsetminus Z$.
	Choose generators for the ideal $J$ defining $Z$, say $J = (f_1, \ldots, f_n)$.
	Let $Z_i = V(f_i)$ and $V_i \coloneqq \Spec(\hA) \smallsetminus Z_i$.
	Then the usual Zariski descent for $\Cohh$ shows that the canonical map
	\[ (\oPhi_Z^\circ(X))^\heartsuit \simeq \varprojlim_{n \in \mathbf \Delta} \Cohh( V^\bullet ) , \]
	where $V^\bullet$ is the \v{C}ech nerve of the total morphism of the maps $V_i \to \Spec(\hA)$.
	It will therefore be enough to check that we can choose $\alpha \in I$ in such a way that for every $i = 1, \ldots, n$ the restriction
	\[ \cG_\alpha |_{V_i} \to j_{X*}(\cF) |_{V_i} = \cF |_{V_i} \]
	is an isomorphism. It is  obviously a monomorphism, since all the maps $\cG_\alpha \to j_{X*}(\cF)$ have been chosen to be monomorphisms.
	Since $V_i = \Spec(\hA[f_i\inv])$ is affine, we can identify $\cF|_{V_i}$ with a coherent $\hA[f_i\inv]$-module $M_i$.
	Choose generators $s_{i1}, \ldots, s_{in_i}$ for $M_i$.
	Up to multiplying by some power of $f_i$ we can suppose that all these sections are the restriction of sections $\tilde{s}_{i1}, \ldots, \tilde{s}_{in_i} \in j_{X*}(\cF)(X)$.
	Since the $V_i$ are in finite number, we can choose $\alpha$ in such a way that the sections $\tilde{s}_{ij}$ for $1 \le i \le n$ and $1 \le j \le n_i$ belong to $\cG_\alpha(X) \subseteq j_{X*}(\cF)(X)$.
	It follows that the induced maps $\cG_\alpha |_{V_i} \to \cF |_{V_i}$ are surjective.
	Since they were monomorphisms to start with, it follows that they are isomorphisms as well.
	Thus, the proof is complete in this case.
	
	Now, suppose that $X$ is a noetherian quasi-compact and quasi-separated scheme, and suppose given the finite affine Zariski cover $\{U_i\}$ satisfying the hypotheses of the statement of this lemma.
	Since  both $L(\cG_\alpha)$ and $L( j_{X*}(\cF) )$ are in $ (\oPhi_Z^\circ(X))^\heartsuit$, and since
	\[ \overline{u}^* \colon (\oPhi_Z^\circ(X))^\heartsuit \to (\oPhi_Z^\circ(U))^\heartsuit \]
	is conservative, it is therefore enough to find $\alpha$ such that
	\[ \overline{u}^*(L(\cG_\alpha)) \to \overline{u}^*(L( j_{X*}(\cF) ))  \]
	is an isomorphism. Now observe that the natural diagram
	\[ \begin{tikzcd}
		\Ind(\Cohh(\hZ)) \arrow{d}{u^*} \arrow{r}{L} & \Ind((\oPhi_Z^\circ(X))^\heartsuit) \arrow{d}{\overline{u}^*} \\
		\Ind(\Cohh( \widehat{Z_{U}} )) \arrow{r}{L} & \Ind((\oPhi_Z^\circ(U))^\heartsuit )
	\end{tikzcd} \]
	commutes. Therefore we are left to prove that there exists $\alpha$ such that
	\[ L(u^*(\cG_\alpha)) \to L(u^*( j_{X*}(\cF) ))  \]
	is an isomorphism. 
	Using \cref{cor:PFN_right_adjointable}, we see in addition that
	\[ u^*( j_{X*}(\cF) ) \simeq j_{U *}( \overline{u}^* (\cF) ) . \]
	Finally, $u^*$ commutes with filtered colimits and, since $u$ is flat, it preserves monomorphisms.
	We can therefore replace $X$ by $U$, $\cF$ by $u^*(\cF)$ and, since $U$ is affine, apply the previous step to the presentation
	\[ u^*(\cF) \simeq \colim_{\alpha \in I} u^*(\cG_\alpha) \]
	 to find the index $\alpha$ we were looking for.
\end{proof}

\subsection{Completion of the proof}

The conclusion of the proof of \cref{thm:Zariski_computational_tool} relies on the following important technical lemma:

\begin{lem} \label{lem:Zariski_devissage}
	Let $X$ be a quasi-compact quasi-separated scheme locally of finite type over $k$.
	Let $Z \hookrightarrow X$ be a closed subscheme.
	Then the functor
	\[ (\oPhi^\circ_Z)^\heartsuit \colon (X\et^{\Sch})\op \to \Catst \]
	informally described by the assignment
	\[ U \mapsto \Cohh(\widehat{Z_U}) / (\cK^\Sch(U))^\heartsuit \]
	is a sheaf for the \emph{Zariski} topology.
\end{lem}

The proof of this result will occupy the remaining part of this section.
Before dealing with it, let us explain how to use this lemma in order to complete the proof of \cref{thm:Zariski_computational_tool}.

\begin{proof}[Proof of \cref{thm:Zariski_computational_tool}]
	As remarked at the end of \cref{subsubsec:preliminaries}, we will actually prove \cref{prop:Zariski_computational_tool}.
	Accordingly to the reduction carried over in \cref{subsubsec:reduction_to_the_heart} it is enough to prove that the induced functor $\varphi^\heartsuit \colon (\oPhi^\circ_Z(X))^\heartsuit \to (\hPhi^\circ_Z(X))^\heartsuit$ is an equivalence.
	
	Combining \cref{thm:etale_descent} and \cref{prop:algebraic_universal_property} we deduce that $(\hPhi^\circ_Z)^\heartsuit$ is the sheafification of the sheaf $(\oPhi^\circ_Z)^\heartsuit$.
	Therefore, it is enough to prove that $(\oPhi^\circ_Z)^\heartsuit$ is a sheaf with respect to $\tauet$.
	Choose an affine Zariski hypercover $U^\bullet$ of $X$.
	Using \cref{lem:Zariski_devissage}, we see that
	\[ (\oPhi^\circ_Z(X))^\heartsuit \simeq \varprojlim (\oPhi^\circ_Z(U^\bullet))^\heartsuit . \]
	Since limits commute with limits, we are immediately reduced to prove that the presheaves defined by
	\[ (X\et^\Sch)\op \ni V \mapsto (\oPhi^\circ_Z(V \times_X U^n))^\heartsuit \in \AbCat \]
	are sheaves.
	It is therefore enough to prove that the presheaves $(\oPhi^\circ_{Z_{U^n}})^\heartsuit \colon ((U^n)\et^{\Sch})\op \to \AbCat$ are sheaves.
	In other words, we can assume $X = U^n$, and we are thus reduced to the case where $X$ is affine.
	In this case, we remark that it is enough to test the sheaf condition on affine covers.
	As consequence, the conclusion follows directly from \cref{thm:etale_descent}.
\end{proof}

\begin{rem}
	It is also clear that \cref{thm:Zariski_computational_tool} implies \cref{lem:Zariski_devissage}, so the two statements are actually equivalent.
	However, it is easier to provide a direct proof of \cref{lem:Zariski_devissage}: this is because we only have to deal with Zariski covers on a quasi-compact scheme, and not with an arbitrary \'etale cover.
\end{rem}

Thus, all what is left to do is to prove \cref{lem:Zariski_devissage}.

\begin{proof}[Proof of \cref{lem:Zariski_devissage}]
	The proof being quite long, we split it into several steps.
	
	\emph{Step 1: Fully faithfulness.}
	Fix a Zariski hypercover $u^\bullet \colon U^\bullet \to X$.
	To ease notations, we write $\cK^\heartsuit(V)$ instead of $(\cK^\Sch(V))^\heartsuit$.
	Consider the induced commutative diagram:
	\[ \begin{tikzcd}[row sep = small, column sep = small]
		{} & \Cohh(\hZ) \arrow{rr}{\hat{u}^{\bullet *}} \arrow{dd} \arrow{dl}[swap]{j} & & \varprojlim \Cohh(\widehat{Z_{U^\bullet}}) \arrow{dl}[swap]{j^\bullet} \arrow{dd} \\
		\Ind(\Cohh(\hZ)) \arrow[crossing over]{rr} \arrow{dd} & & \varprojlim \Ind(\Cohh(\widehat{Z_{U^\bullet}})) \\
		{} & \Cohh(\widehat{Z}) / \cK^\heartsuit(X) \arrow{dl}[swap]{i} \arrow{rr}[near start]{v^{\bullet *}} & & \varprojlim \Cohh(\widehat{Z_{U^\bullet}}) / \cK^\heartsuit(U^\bullet) \arrow{dl}[swap]{i^\bullet} \\
		\Ind(\Cohh(\hZ)) / \cK^\heartsuit(X)) \arrow{rr} & & \varprojlim \Ind( \Cohh(\widehat{Z_{U^\bullet}}) ) / \cK^\heartsuit(U^\bullet) ) \arrow[leftarrow, crossing over]{uu} .
	\end{tikzcd} \]
	The adjoint functor theorem guarantees that both the functors
	\[ \Ind(\Cohh(\hZ)) \to \varprojlim \Ind(\Cohh(\widehat{Z_{U^\bullet}})) \quad \text{and} \quad \Ind(\Cohh(\hZ) / \cK^\heartsuit(X)) \to \varprojlim \Ind( \Cohh(\widehat{Z_U}) / \Ind(\cK^\heartsuit(U^\bullet) )) \]
	have right adjoints, which we denote respectively $\hat{u}^\bullet_*$ and $v^\bullet_*$.
	It follows from the discussion following \cite[Corollary 8.6]{Porta_Yu_Higher_analytic_stacks_2014} that these two right adjoints can be respectively identified with the functors informally given by the assignments:
	\[ \{\cF^n\} \mapsto \varprojlim_{n \in \mathbf \Delta} \hat{u}^n_* \cF^n , \qquad \{\cG^n\} \mapsto \varprojlim_{n \in \mathbf \Delta} v^n_* \cG^n . \]
	Note that both $\Ind(\Cohh(\hZ))$ and $\Ind(\Cohh(\hZ) / \cK^\heartsuit(X))$ are $1$-categories.
	Since the inclusion $\mathbf \Delta^{\le 1} \to \mathbf \Delta$ is $1$-cofinal, we conclude that we can identify those limits with the equalizers:
	\[ \varprojlim_{n \in \mathbf \Delta} \hat{u}^n_* \cF^n \simeq \mathrm{eq}( \hat{u}^0_* \cF^0 \rightrightarrows \hat{u}^1_* \cF^1 ), \qquad \varprojlim_{n \in \mathbf \Delta} v^n_* \cG^n \simeq \mathrm{eq}( v^0_* \cG^0 \rightrightarrows v^1_* \cG^1 ) . \]
	Furthermore, just as in \cref{lem:reduction_to_the_heart_I} we can use \cref{cor:quotient_t_structure} in order to endow
	\[ \Ind(\Coh^-(Z) / \cK^\Sch(X)) \simeq \Ind(\Coh^-(Z)) / \Ind(\cK^\Sch(X)) \]
	with a $t$-structure having the property that the quotient map
	\[ \Ind(\Coh^-(\hZ)) \to \Ind(\Coh^-(Z)) / \Ind(\cK^\Sch(X)) \]
	is $t$-exact.
	Passing to the heart, we obtain that the induced functor
	\[ \Ind(\Cohh(\hZ)) \to \Ind(\Cohh(\hZ) / \cK^\heartsuit(X)) \]
	is an exact functor.
	As consequence, we can deduce that the commutative square
	\[ \begin{tikzcd}
		\Ind(\Cohh(\hZ)) \arrow{r} \arrow{d}{\Ind(L)} & \varprojlim \Ind(\Cohh(\widehat{Z_{U^\bullet}})) \arrow{d}{\Ind(L^\bullet)} \\
		\Ind(\Cohh(\hZ) / \cK^\heartsuit(X)) \arrow{r} & \varprojlim \Ind( \Cohh(\widehat{Z_{U^\bullet}}) / \cK^\heartsuit(U^\bullet) )
	\end{tikzcd} \]
	is right adjointable.
	On the other hand, \cref{lem:right_adjointable_key_lemma} shows that the square
	\[ \begin{tikzcd}
		\Cohh(\hZ) \arrow{r} \arrow{d}{j} & \varprojlim \Cohh(\widehat{Z_{U^\bullet}}) \arrow{d}{j^\bullet} \\
		\Ind(\Cohh(\hZ)) \arrow{r} & \varprojlim \Ind(\Cohh(\widehat{Z_{U^\bullet}}))
	\end{tikzcd} \]
	is right adjointable as well.
	Since $\hat{u}^{\bullet *}$ is an equivalence, we can use these informations to conclude that if $\cF \in \Cohh(\hZ)$, then the unit
	\[ \Ind(L)(j(\cF)) \to v^\bullet_*( v^{\bullet *}(\Ind(L)(j(\cF)))) \]
	is an equivalence.
	Fix now $\cF, \cG \in \Cohh(\hZ) / \cK^\heartsuit(X)$.
	Since the quotient map $L \colon \Cohh(\hZ) \to \Cohh(\hZ) / \cK^\heartsuit(X)$ is essentially surjective, we can find $\widetilde{\cF}, \widetilde{\cG}$ such that
	\[ L(\widetilde{\cF}) \simeq \cF, \qquad L(\widetilde{\cG}) \simeq \cG . \]
	In particular, it follows that the unit of the adjunction $v^{\bullet *} \dashv v^\bullet_*$ evaluated on
	\[ i(\cF) \simeq i( L(\widetilde{\cF}) ) \simeq \Ind(L)(j(\cF)) \quad \text{and on} \quad i(\cG) \simeq i( L(\widetilde{\cG}) ) \simeq \Ind(L)(j(\cG)) \]
	is an equivalence.
	As consequence we obtain a commutative diagram
	\[ \begin{tikzcd}
		\Map(\cF, \cG) \arrow{r} \arrow{d} & \Map(v^{\bullet *} \cF, v^{\bullet *} \cG) \arrow{d} \\
		\Map(i(\cF), i(\cG)) \arrow{r} & \Map(\Ind(v^{\bullet *})( i(\cF) ), \Ind(v^{\bullet*})( i(\cG) )) ,
	\end{tikzcd} \]
	where the vertical arrows as well as the bottom horizontal one are equivalences.
	We can finally conclude that the top horizontal map is an equivalence as well.
	In other words, we proved that $v^{\bullet *}$ is fully faithful.
	
	\emph{Step 2: effectivity in the separated case.}
	Assume that $X$ is a separated scheme.
	Since $X$ is quasi-compact, we can limit ourselves to deal with finite affine Zariski covers $\cU \coloneqq \{U_1, \ldots, U_n\}$ of $X$.
	We let $\cU^\bullet$ be the \v{C}ech nerve of $\cU$ and we let $i^p \colon U^p \hookrightarrow X$ be the induced open Zariski immersion.
	Observe that since $X$ is separated, each intersection $U^p$ is affine.
	Let
	\[ i^{\bullet *}_\cU \colon \Cohh(\hZ) \to \varprojlim_{p \in \mathbf \Delta} \Cohh(\widehat{Z_{U^p}}) , \quad \overline{\imath}^{\bullet *}_\cU \colon (\oPhi^\circ_Z(X))^\heartsuit \to \varprojlim_{p \in \mathbf \Delta} (\oPhi^\circ_Z(U^p))^\heartsuit \]
	be the induced maps.
	We know that $i^{\bullet *}_\cU$ is an equivalence, and we want to prove that $\overline{\imath}^{\bullet *}_\cU$ is an equivalence as well.
	We proceed by induction on the size $n$ of the affine Zariski cover $\cU$.
	If $n = 1$, the statement is obvious.
	
	We now let $n > 1$. Set
	\[ V_1 \coloneqq \bigcup_{i = 1}^{n-1} U_i, \quad V_2 \coloneqq U_n , \quad V_{12} \coloneqq V_1 \cap V_2 . \]
	As usual, we let $\cJ$ be the ideal sheaf defining $Z$ inside $X$ and we denote by $\cJ_1$, $\cJ_2$, $\cJ_{12}$ the restrictions to $V_1$, $V_2$ and $V_{12}$, respectively.
	To ease notations, we will also write $Z_1$, $Z_2$ and $Z_{12}$ instead of $Z_{V_1}$, $Z_{V_2}$ and $Z_{V_{12}}$, respectively.
	We apply a similar convention to the formal completions $\widehat{Z_1}$, $\widehat{Z_2}$ and $\widehat{Z_{12}}$.
	
	Fix a descent datum $\{\cF_1, \cF_2, \alpha\}$ for $(\oPhi^\circ_Z)^\heartsuit$ relative to the cover $\{V_1, V_2\}$ of $X$.
	By this, we mean that
	\[ \cF_1 \in (\oPhi^\circ_Z(V_1))^\heartsuit, \quad \cF_2 \in (\oPhi^\circ_Z(V_2))^\heartsuit , \quad \alpha \coloneqq \cF_{12} \simeq \cF_{21} , \]
	where we set
	\[ \cF_{12} \coloneqq \cF_1 |_{V_{12}}, \quad \cF_{21} \coloneqq \cF_2 |_{V_{12}} . \]
	\textbf{Claim}. There exists a descent datum
	\[ \{\cG_1, \cG_2, \beta\} \in \varprojlim \left( \Cohh(\widehat{Z_1}) \times \Cohh(\widehat{Z_2}) \rightrightarrows \Cohh(\widehat{Z_{12}}) \right) \]
	which induces $\{\cF_1, \cF_2, \alpha\}$ via the canonical functors
	\[ L_1 \colon \Cohh(\widehat{Z_1}) \to (\oPhi_Z^\circ(V_1))^\heartsuit, \quad L_2 \colon \Cohh(\widehat{Z_2}) \to ( \oPhi^\circ_Z(V_2) )^\heartsuit, \quad L_{12} \colon \Cohh(\widehat{Z_{12}}) \to ( \oPhi^\circ_Z(V_{12}) )^\heartsuit . \]
	Assuming the Claim, it is easy to conclude that $\overline{\imath}^{\bullet *}_{\cU}$ is an equivalence the separated case: since
	\[ \Cohh(\hZ) \simeq \varprojlim \left( \Cohh(\widehat{Z_1}) \times \Cohh(\widehat{Z_2}) \rightrightarrows \Cohh(\widehat{Z_{12}}) \right) , \]
	we can find $\cG \in \Cohh(\hZ)$ such that $\cG |_{V_i} \simeq \cG_i$.
	Let $\cF \coloneqq L(\cG) \in (\oPhi^\circ_Z(X))^\heartsuit$.
	Then $\cF$ induces by restriction the descent datum $(\cF_1, \cF_2, \alpha)$.
	It follows that $\overline{\imath}^{\bullet *}_{\cU}$ is essentially surjective, and therefore that it is also an equivalence.
	
	\textbf{Proof of Claim.}
	We start by applying \cref{cor:almost_functorial_formal_models} to choose a lift $\cG_1 \in \Cohh(\widehat{Z_{V_1}})$ of $\cF_1$ equipped with a monomorphism
	\[ \cG_1 \hookrightarrow j_{1*}(\cF_1) \in \Ind(\Cohh(\widehat{Z_{V_1}})) . \]
	This is possible because the induction hypothesis shows that $\oPhi_Z^\circ$ has descent with respect to the affine Zariski open cover $\{U_1, \ldots, U_{n-1}\}$ of $V_1$.
	On the other side, we can write
	\[ j_{2*}(\cF_2) \simeq \colim_{i \in I} \cG_i , \]
	where $\cG_i \in \Cohh(\widehat{Z_{V_2}})$ and the category $I$ is filtered.
	Furthermore, we can assume just as in the proof of \cref{cor:almost_functorial_formal_models} that all the maps $\cG_i \to j_{2*}(\cF_2)$ are monomorphism.
	We now apply \cref{cor:PFN_right_adjointable} to deduce that the canonical morphisms in $\Ind( \Cohh( \widehat{Z_{12}} ) )$
	\[ j_{1*}(\cF_1) |_{\widehat{Z_{12}}} \to j_{12*}(\cF_{12}), \quad j_{2*}(\cF_2)|_{\widehat{Z_{12}}} \to j_{12*}(\cF_{21}) \]
	are isomorphisms.
	In particular, the descent datum $\alpha \colon \cF_{12} \simeq \cF_{21}$ induces an equivalence
	\[ j_{12*}(\alpha) \colon j_{12*}(\cF_{12}) \simeq j_{12*}(\cF_{21}) . \]
	Let $i_{21} \colon \widehat{Z_{12}} \to \widehat{Z_2}$ be the induced formal open Zariski immersion.
	It follows that $i_{21}^*$ is a left adjoint and preserves monomorphisms.
	In particular, we obtain
	\[ j_{12*}(\cF_{21}) \simeq \colim_{i \in I} i_{21}^*( \cG_i ) = \colim_{i \in I} \cG_i |_{\widehat{Z_{12}}} , \]
	and every morphism $\cG_i |_{\widehat{Z_{12}}} \to j_{12*}(\cF_{21})$ is a monomorphism.
	For the same reason, we see that also $\cG_1 |_{\widehat{Z_{12}}} \to j_{12*}(\cF_{21})$ is a monomorphism.
	We claim that we can choose $i$ so that the monomorphism
	\[ \cG_1 |_{\widehat{Z_{12}}} \hookrightarrow j_{12*}(\cF_{12}) \stackrel{j_{12*}(\alpha)}{\simeq} j_{12*}( \cF_{21} ) \]
	factors through the monomorphism $\cG_i |_{\widehat{Z_{12}}} \hookrightarrow j_{12*}(\cF_{21})$.
	Remark that, since both maps are monomorphism, this is a property of $\cG_i |_{\widehat{Z_{12}}}$.
	In particular, since $U \mapsto \Cohh( \widehat{Z_U} )$ satisfies descent, we can test this property locally on $V_{12}$.
	Using \cref{cor:PFN_right_adjointable} and the fact that $V_{12}$ admits a \emph{finite} affine cover, we can thus \emph{replace} the scheme $V_{12}$ with an open affine subscheme.
	In other words, we can assume that $V_{12}$ is affine from the very beginning.
	In this case, we can identify $\cG_1 |_{\widehat{Z_{12}}}$ with a finitely generated submodule of
	\[ j_{12*}(\cF_{21}) = \colim_{i \in I} \cG_i |_{\widehat{Z_{12}}} . \]
	In particular, there must exist an index $i \in I$ such that $\cG_i |_{\widehat{Z_{12}}}$ contains all the generators of $\cG_1 |_{\widehat{Z_{12}}}$.
	In conclusion, we can choose an index $i \in I$ so that the descent datum $\alpha$ induces a well defined map
	\[ \beta \colon \cG_1 |_{\widehat{Z_{12}}} \to \cG_i |_{\widehat{Z_{12}}} . \]
	Note that we are free to replace $i$ with any index $i'$ satisfying $i' \ge i$ inside $I$.
	Therefore, using \cref{cor:almost_functorial_formal_models} we can also assume that $\cG_i$ is a lift for $\cF_2$.
	Furthermore, we observe that $\beta$ fits into the commutative diagram
	\[ \begin{tikzcd}
		\cG_1 |_{\widehat{Z_{12}}} \arrow{r}{\beta} \arrow[hook]{d} & \cG_i |_{\widehat{Z_{12}}} \arrow[hook]{d} \\
		j_{12*}(\cF_{21}) \arrow{r}{j_{12*}(\alpha)} & j_{12*}(\cF_{21})
	\end{tikzcd} \]
	in $\Ind(\Cohh(\widehat{Z_{12}}))$.	
	Since $j_{12*}(\alpha)$ is an isomorphism, we conclude that $\beta$ must be a monomorphism.
	However, it needs not to be an epimorphism.
	
	To bypass this issue, we let $\cG_i'$ be the subsheaf of $\cG_i$ spanned by those sections whose restriction to $\widehat{Z_{12}}$ belong to the image of $\beta$.
	Since $\cG_i$ is coherent and we are working in a noetherian environment, $\cG_i'$ is again coherent.
	By construction, $\beta$ factors through $\left. \cG_i' \right|_{\widehat{Z_{12}}}$ and induces an isomorphism with $\cG_1 |_{\widehat{Z_{12}}}$.
	We claim that the inclusion $\cG_i' \hookrightarrow \cG_i$ becomes an equivalence in $(\oPhi^\circ_Z(V_2))^\heartsuit$.
	Recall that $V_2$ is affine, and write $V_2 = \Spec(A_2)$.
	Choose generators $f_1, \ldots, f_r$ for the ideal $\cJ_2$ defining $Z_2$ inside $V_2$.
	Furthermore, let $W_j$ be the open complement of the zero locus of the function $f_j$ on $\Spec(\widehat{A_2})$.
	Then each $W_j$ is affine, and by definition $\{W_j\}_{j = 1, \ldots, r}$ is an open Zariski cover of $\hPsi^\circ_Z(V_2)$.
	To complete the proof of the claim, It is therefore enough to prove that the induced map
	\[ \cG'_i |_{W_j} \to \cG_i |_{W_j} \simeq \cF_2 |_{W_j} \]
	is surjective.
	Since $W_j$ is affine, we can test surjectivity after passing to the global sections.
	Fix any section $s \in \cF_2(W_j)$.
	Since $W_j$ is an open Zariski subset of $\Spec(\widehat{A_2})$, we can review $s$ as a section of $j_{2*}(\cF_2)$ over $W_j$.
	We claim that there exists an integer $l \ge 0$ such that $f_j^l s$ belongs to $\cG'_i(W_j) \subset j_{2*}(\cF_2)(W_j)$.
	Since $f_j$ is invertible over $W_j$, this claim implies immediately the surjectivity of $\cG'_i(W_j) \to \cG_i(W_j)$.
	
	We are therefore left to prove the existence of such an integer $l$.
	We start by remarking that we can certainly find an integer $l'$ so that $f_j^{l'} s$ is the restriction of a section $\tilde{s} \in \cG_i(V_2)$.
	It is thus enough to prove that there exists $l$ so that $f_j^l \tilde{s}$ belongs to $\cG'_i(V_2)$.
	The latter statement is equivalent to say that the section $f_j^l \tilde{s} |_{\widehat{Z_{12}}}$ of $\cG_i |_{\widehat{Z_{12}}}$ belongs to the image of $\beta$.
	Using \cref{cor:PFN_right_adjointable} once more, we see that this statement can be checked locally on $V_{12}$.
	As a consequence, we can assume again $V_{12}$ to be affine (since $V_{12}$ admits a finite affine cover).
	Set $W_j^{12} \coloneqq W_j \times_{\widehat{Z_2}} \widehat{Z_{12}}$.
	We now observe that $(\tilde{s} |_{\widehat{Z_{12}}})|_{W_j^{12}}$ defines an element of
	\[ (\cG_i |_{\widehat{Z_{12}}})(W_j^{12}) = (\cG_1 |_{\widehat{Z_{12}}})(W_j^{12}) , \]
	and that $W_j^{12}$ is the principal Zariski open inside $\Spec(\widehat{A_{12}})$ defined by the image of the element $f_j$.
	Since $\cG_1 |_{\widehat{Z_{12}}} \subset \cG_i |_{\widehat{Z_{12}}}$, it follows that there exists an integer $l$ such that $f_j^l \tilde{s} |_{\widehat{Z_{12}}}$ belongs to $\cG_1 |_{\widehat{Z_{12}}}$.
	Thus, the proof of the claim is finally achieved.

	\emph{Step 3: effectivity in the general case.}
	We now consider the general case of $X$ being quasi-separated.
	In fact, the same argument used above goes through.
	The only place where one has to apply some care is the following: in order to construct the sections $t_i$, we used applied the induction hypothesis to the \emph{affine} cover $\cV_{12} \coloneqq \{U_1 \cap U_n, \ldots, U_{n-1} \cap U_n\}$ of $V_{12}$.
	This is always a cover, but we can say that it is an affine cover only if $X$ is separated.
	Nevertheless, $V_{12}$ is an open Zariski subset of $V_2$, which is affine.
	In particular, $V_{12}$ is separated.
	Therefore, the separated case shows that $\oPhi^\circ_Z$ satisfies descent for the cover $\cV_{12}$, even though it is not affine.
	Since $X$ is quasi-compact, we can now refine this cover by an affine hypercover. The previous step, shows that $\oPhi^\circ_Z$ has descent also for this refinement.
	At this point, we are reduced to a situation where we can apply again \cref{thm:localization_for_coh_minus}, and the proof proceeds exactly as above.
\end{proof}

\section{Vistas}\label{vistas}

As alluded to in the Introduction, the results of this paper are part of a more general program. 
In this final section we sketch some of the further stages in this program, as we understand them today, that will be fully addressed in a forthcoming paper.
The exact details of the constructions, and of the expected results are not fully established at the moment, so we hope the reader will forgive us for the lack of precise statements.
On the other hand, we believe that our presentation will be able to convey the main ideas. \medbreak

In order to ease notations, we will stick to the case of a projective smooth complex algebraic surface $S$.
It will be perfectly clear how to generalise our constructions to complex projective smooth varieties of arbitrary dimension, and even to algebraic stacks.\\

\paragraph{\textbf{The flag decomposition as a stack over the flag Hilbert scheme.}}

Let us consider the following flag Hilbert scheme $\cH^{\mathrm{Fl}}(S/\mathbb{C})$ of $S/\mathbb{C}$.
Let $ \mathrm{FH_{(P_0, P_1)}^{S}}$ be the flag Hilbert scheme of $S/\mathbb{C}$ parametrizing flags of closed subschemes $Z_0 \subset Z_1 \subset S $ where $Z_i$ has Hilbert polynomial $P_i$, $i=0,1$ (see \cite[\S 4.5]{Sernesi_Deformation_2006}).
Then, we define $\cH^{\mathrm{Fl}}(S/\mathbb{C})$ as the subscheme
\[ \cH^{\mathrm{Fl}}(S/\mathbb{C}) \subset \coprod_{\deg P_0 = 0, \,\deg P_1=1} \mathrm{FH_{(P_0, P_1)}^{S}} , \]
consisting of flags of relative codimension $1$ subschemes.
For $T$ a $\mathbb{C}$-scheme, the set $\cH^{\mathrm{Fl}}(S/\mathbb{C})(T)$ of $T$-points of $\cH^{\mathrm{Fl}}(S / \mathbb{C})$ consists of flags of closed immersions
\[ Z_0 \hookrightarrow Z_1 \hookrightarrow S \times T \]
such that $Z_0 \hookrightarrow Z_1$ (resp.\ $Z_1 \hookrightarrow S \times T$) is a relative codimension $1$ closed subscheme, flat over $T$.

Let $\mathbf{D}$ denote either of the symbols $\bfCoh^-$, $\bfPerf$, $\bfBun_G$, $\rR \bfCoh^-$, $\rR \bfPerf$ or $\rR \bfBun_G$.
Let $T= \Spec(R)$ be a noetherian affine $\mathbb{C}$-scheme, equipped with a map
\[ u \colon T \to \cH^{\mathrm{Fl}}(S / \mathbb C) \]
of $\mathbb{C}$-schemes.
As recalled above, the map $u$ corresponds to a flag  
\[ \underline{Z}_u \coloneqq(Z_0 \hookrightarrow Z_1 \hookrightarrow S \times T). \]
If we denote by $U_1$ the open complement of $Z_1 \hookrightarrow S \times T$, and $U_{01}$ the open complement of $Z_0 \hookrightarrow \widehat{(Z_1)}_{S \times T}$, we can consider the functor
\[ \rD_S^{\mathrm{Fl}} \colon \left(\Aff_{\mathbb C} / \cH^{\mathrm{Fl}}(S / \mathbb C) \right)\op \longrightarrow \Catst, \quad
(u \colon T = \Spec(R) \to \cH^{\mathrm{Fl}}(S / \mathbb C) ) \longmapsto \mathbf{D}_{\underline{Z}_u}^{\mathrm{Fl}} , \]
where (cf. the statement of\ \cref{cor:full_flag_decomposition_surface})
\[ \mathbf{D}_{\underline{Z}_u}^{\mathrm{Fl}} \coloneqq \mathbf{D}(U_1) \times_{\mathbf{D}(\widehat{(Z_1)}_{S \times T} \smallsetminus Z_1)} (\mathbf{D}(U_{01}) \times_{\mathbf{D}(\widehat{(Z_0)}_{\widehat{(Z_1)}_{S \times T}} \smallsetminus Z_0)} \mathbf{D}(\widehat{(Z_0)}_{\widehat{(Z_1)}_{S \times T}})). \]
The functor $\rD_{S}^{\mathrm{Fl}}$ is a sheaf for the \'etale topology.

Now, \cref{cor:full_flag_decomposition_surface} give us decomposition equivalences
\[ \mathbf{D}(S \times T) \simeq \mathbf{D}(U_1) \times_{\mathbf{D}(\widehat{(Z_1)}_{S \times T} \smallsetminus Z_1)} \mathbf{D}(\widehat{(Z_1)}_{S \times T}) \simeq \mathbf{D}_{\underline{Z}_u}^{\mathrm{Fl}}, \]
so that we may think of  $\rD_{S}^{\mathrm{Fl}}$ as a locally constant sheaf of stable $\infty$-categories on the \'etale site of $\cH^{\mathrm{Fl}}(S / \mathbb{C})$.
Therefore we expect an action of the \'etale fundamental group $\pi_1^{\mathrm{\acute{e}t}}(\cH^{\mathrm{Fl}}(S/\mathbb{C}))$ on $\rD_{S}^{\mathrm{Fl}}$, and more generally an action of the \emph{full} \'etale homotopy type of $\cH^{\mathrm{Fl}}(S/\mathbb{C})$ on $\rD_{S}^{\mathrm{Fl}}$.
In particular, one might be able to obtain this way interesting auto-equivalences of the $\infty$-category $\mathbf{D}(S)$, if the action is non-trivial, and we don't see reasons why it should be trivial in general. \medbreak

However, this action seems to see only a tiny bit of the actual geometry of the flag Hilbert scheme $\cH^{\mathrm{Fl}}(S / \mathbb{C})$ of $S$.
In particular, one can envisage interesting ``operations'' among the flags (like e.g.\ taking their disjoint union or their intersections, contracting $(-1)$-rational divisors when possible, etc.) that one would like to see ``acting'' in some way on the stack $\rD_{S}^{\mathrm{Fl}}$, or, at least, on the $\infty$-category $\mathbf{D}(S)$.
At the moment, we see two possible flavors of formalization for this vague notion of operations and action.
One possibility is that the flag Hilbert scheme of $S$ organizes itself into a kind of \emph{operadic-like} structure encoding the various geometrical operations on flags, and this operadic-like structure will act on $\rD_{S}^{\mathrm{Fl}}$ and/or on $\infty$-category $\mathbf{D}(S)$.
The analogy here is with the well-known example of the modular operad $\overline{\cM} = \{ \overline{\cM}_{g,n} \}$ of moduli spaces of curves of genus $g \geq 0$, with $n \geq 0$ marked points, that precisely encodes various geometrical operations on such families of curves (see e.g.\ \cite[\S II.5.6]{Markl_Shnider_Stasheff_Operads_2002}).
The other possibility is that the stack $\rD_{S}^{\mathrm{Fl}}$ has the structure of some higher version  of \emph{factorizable stack} over the flag Hilbert scheme of $S$, in a similar way as one can talk about factorizable sheaves on a curve in the geometric Langlands program (e.g. \cite{Beilinson_Drinfeld_Chiral_2004}).
In order to start exploring this second possibility, we first need to make sense of a \emph{Ran version} of the flag Hilbert scheme of $S$. \medbreak

\paragraph{\textbf{Ran versions of the flag Hilbert scheme.}}

We give here a few proposals for the \emph{Ran version} of the flag Hilbert scheme $\cH^{\mathrm{Fl}}(S / \mathbb{C})$ of $S$. Other versions are  currently under investigations, but the ones presented below should carry enough features for our general discussion here. \medbreak

Let $\mathrm{FinSets}$ denote the category of finite non-empty sets, with surjective maps as morphisms.
We write $\mathrm{FinSets}'$ for the category of pairs $(I, J \to I)$ where $I \in \mathrm{FinSets}$, and $J \to I$ is a morphism in $\mathrm{FinSets}$.
Morphisms $(I, J \to I) \to (I', J' \to I')$ in $\mathrm{FinSets}'$ are defined in the obvious way as pairs of maps $I \to I'$, and $J \to J'$ in $\mathrm{FinSets}$ such that the obvious diagram commutes. \medbreak

Let $\cH_{i}(S)$ be the Hilbert scheme of closed $i$-dimensional subschemes of $S$, $i=0, 1$.
If $\mathrm{Hilb}^{P}(S/\mathbb{C})$ denotes the Hilbert scheme of closed subschemes of $S$ having Hilbert polynomial $P$ (with respect a polarization fixed once and for all), we have
\[ \cH_i(S/ \mathbb{C}) = \coprod_{\deg P = i} \mathrm{Hilb}^P(S / \mathbb C) . \]
For $(I, \varphi \colon J \to I) \in \mathrm{FinSets}'$, let us consider the incidence subscheme $\cF\ell^{\mathrm{inc}}_{(I, J \to I)} \subset \cH_1(S/ \mathbb{C})^{I} \times \cH_0 (S)^J$, whose $T$-points, for $T$ a $\mathbb{C}$-scheme, are defined as
\[ \cF\ell^{\mathrm{inc}}_{(I, J \to I)}(T) \coloneqq \{ ((Z_i)_{i \in I}, (z_j)_{j\in J}) \mid z_j \text{ is a $T$-relative codimension $1$ closed subscheme of } Z_i, \forall j \in \varphi^{-1}(i) \} , \]
where $(Z_i \hookrightarrow S \times T)_{i \in I} \in \cH_1(S/ \mathbb{C})(T)^{I}$, and $(z_j \hookrightarrow S \times T)_{j \in J} \in \cH_0(S)(T)^J$.
Note that since $\varphi \colon J \to I$ is surjective, all the ``flags'' in  $\cF\ell^{\mathrm{inc}}_{(I, J \to I)}(T)$ are complete, and moreover the apparently sloppy condition of ``being a relative codimension $1$ closed subscheme '' in the definition of $\cF\ell_{(I, J \to I)}(T)$ is in fact precise: if there exists a closed $S \times T$-immersion, it is then unique, since the $Z_i$'s and the $z_j$'s are given as closed subschemes of $S\times T$, and closed immersions are monomorphisms of schemes. \medbreak

If $f=(f_1 \colon I \to I', f_0 \colon J\to J')$ is a morphism $(i, J \to I) \to (I', J' \to I')$ in $\mathrm{FinSets}'$, there is an induced scheme map
\[ f^* \colon \cF\ell^{\mathrm{inc}}_{(I', J' \to I')} \longrightarrow \cF\ell^{\mathrm{inc}}_{(I, J \to I)} \]
such that
\[ f^*(T) \colon \cF\ell^{\mathrm{inc}}_{(I', J' \to I')}(T) \longrightarrow \cF\ell^{\mathrm{inc}}_{(I, J \to I)}(T) \]
sends $((Z'_{i'})_{i' \in I'}, (z'_{j'})_{j' \in J'})$ to $((Z_i\coloneqq Z'_{f_{1} (i)})_{i \in I}, (z_j\coloneqq z'_{f_{0} (j)})_{j\in J})$. \medbreak

It is easy to check that the assignment $(I, J\to I) \mapsto \cF\ell^{\mathrm{inc}}_{(I, J \to I)}$ defines a functor
\[ \cF\ell^{\mathrm{inc}} \colon \mathrm{FinSets}'^{op} \longrightarrow \PSh_{\mathbb{C}} , \]
where $\PSh_{\mathbb{C}}$ denotes the category of functors $\Sch\op_{\mathbb{C}} \to \rSet$.

\begin{defin}
	For a complex smooth projective surface $S$, we define
	\begin{itemize}
		\item the \emph{incidence Ran flag Hilbert space} of $S$, as
		\[ \mathrm{Ran}\cF\ell^{\mathrm{inc}} \coloneqq \colim \cF\ell^{\mathrm{inc}}  \in \PSh_{\mathbb{C}} ; \]
		\item the \emph{Ran flag Hilbert space} of $S$, as the Ran space
		\[ \mathrm{Ran}(\cH^{\mathrm{Fl}}(S/\mathbb{C}) ) \in \PSh_{\mathbb{C}} \]
		of the flag Hilbert scheme  $\cH^{\mathrm{Fl}}(S/\mathbb{C})$ (see \cite[1.2.3]{Gaitsgory_Contractibility_2011} for the definition of $\mathrm{Ran}$);
		\item the \emph{identification Ran flag Hilbert space} of $S$, as the topological space of complex points $\cH^{\mathrm{Fl}}(S/\mathbb{C})(\mathbb{C})$ (endowed with the analytic topology) modulo the equivalence relation identifying flags having the same topological support in $S(\mathbb{C})$\footnote{A similar construction was also suggested by N. Rozenblyum.}.
	\end{itemize}
\end{defin}

In a forthcoming paper we plan to investigate monoidal structures on sheaves/stacks on the above Ran versions of the flag Hilbert scheme of $S$, together with generalized factorization structures on $\rD_{S}^{\mathrm{Fl}}$ viewed as a stack on the Ran versions of the flag Hilbert scheme of $S$. \medbreak

\begin{rem}\label{Cliff}
Let us remark that an interesting Ran version of the Hilbert scheme of points for higher dimensional varieties $X$ has been recently considered in \cite{Cliff_Thesis_2015}. In this work, Cliff does not consider the flag Hilbert scheme but rather the Hilbert scheme of points, and then a factorization space $\mathcal{H}\mathrm{ilb}_{\mathrm{Ran}\, X}$ over the Hilbert scheme of points of $X$, out of which she is indeed able to construct, by linearization, a usual factorization algebra (\cite[Chapter 2]{Cliff_Thesis_2015}). Form our point of view, this is a very interesting first step; what is left is to make sense of a richer, higher factorization structure taking into account full flags in $X$.
\end{rem}

\paragraph{\textbf{The case of $G$-bundles.}}

The stack $(\rR \bfBun_G)_{S}^{\mathrm{Fl}}$ (i.e. the stack $\rD_{S}^{\mathrm{Fl}}$ for $\rD= \rR \bfBun_G$) is of special interest in our program of understanding the geometric structures underlying Geometric Langlands for surfaces. One reason is that it should enable a construction of a Hecke-like action, related to the higher factorization structure over a Ran version of the flag Hilbert scheme of $S$. We are currently investigating this intriguing direction. Note that it is crucial to work with the derived version of the stack of $G$-bundles (and using Corollary \ref{cor:flag_bung_derived}) as soon as the ambient scheme ($S$ here) has dimension strictly bigger than one. \medbreak

Let us add that, in \cite{Faonte_Hennion_Kapranov}, Giovanni Faonte, Benjamin Hennion and Mikhail Kapranov are investigating the \emph{degenerate} case of a fixed point $x$ in $S$ (i.e. the ``flag'' $\{x\} = Z_0 \subset S$). In this case, the stack $\rR\bfBun_G(\hat{x} \smallsetminus \{x\})$ can be thought as some loop space in the stack of $G$-bundles (in the sense of \cite{kapranovvasserot:fl1} and \cite{hennion:floops}). As such, it should carry a factorisation structure (as in \cite{Cliff_Thesis_2015}).
In \cite{Faonte_Hennion_Kapranov}, the authors identify some higher dimensional and derived version of Laurent current and Kac-Moody algebras and give an infinitesimal action on the stack of $G$-bundles of $S$ trivialized at the neighborhood of the point $x$.

\appendix

\section{Comparison with the work of Ben-Bassat and Temkin} \label{appendix:comparison}

The problem of constructing some incarnation of the punctured formal neighbourhood has already been faced in the literature, and solutions different from ours have been proposed before.
Perhaps, one of the most successful ones is the construction of the punctured formal neighbourhood as a Berkovich analytic space carried over in \cite{Ben-Bassat_Temkin_Tubular_2013}.
In this section, we briefly review this construction, and we compare it with ours whenever they both apply, showing that indeed we obtain the same categories of coherent sheaves.

In \cite{Ben-Bassat_Temkin_Tubular_2013} the authors work over a base field $k$, which is seen as a non-archimedean field equipped with the trivial valuation.
They consider a $k$-variety $X$ and a closed subvariety $Z \hookrightarrow X$.
Out of this data, they introduce the formal completion $\fX \coloneqq \widehat{X}_Z$, seen as a special formal scheme over $k$.
Following \cite{Berkovich_Vanishing_II_1996}, they next consider the generic fiber $\fX_\eta$, which is a Berkovich $k$-analytic space.\footnote{The reader might find it useful to be reminded that the generic fiber of $\Spf(k\llb T_1, \ldots, T_n \rrb)$ is the open analytic disc of radius $1$, independently of which valuation we consider $k$ equipped with. In the case of trivial valuation, it is furthermore true that the global sections of $\Spf(k\llb T_1, \ldots, T_n\rrb)_\eta$ are isomorphic to $k\llb T_1, \ldots, T_n \rrb$ itself.}
Reviewing $Z$ as a discrete formal $k$-scheme topologically of finite presentation, it is possible to consider also the generic fiber $Z_\eta$.
In \cite{Ben-Bassat_Temkin_Tubular_2013}, the punctured formal neighbourhood is defined to be a Berkovich $k$-analytic space, formally defined as
\[ W = W_X(Z) \coloneqq \fX_\eta \smallsetminus Z_\eta . \]
It is argued in loc.\ cit.\ that this space is non-empty.
It is furthermore showed there that the construction of $W_X(Z)$ satisfies Zariski descent in $X$, in the sense that a Zariski open cover of $X$ is sent to a $G$-cover of $W$.

In a former version of this paper, we were using the non-archimedean framework in a much more essential way. The idea was to define our key functor $\hPhi^\circ_Z$ via the universal property stated in \cref{prop:algebraic_universal_property}, and then use \cref{thm:localization_for_coh_minus} in order to identify it with the sheaf informally described by
\[ U \mapsto \Coh^-(W_U(Z_U)) . \]
In fact, the non-archimedean point of view allows also to give a more ``geometrical'' interpretation of the stack $\bfCoh^-_{\hZ \smallsetminus Z}$ itself: it would have to be interpreted as the restriction of the analytic \emph{mapping stack} $\bfMap(W(X,Z), (\bfCoh^-)\an)$ to the algebraic category.
Nevertheless, this point of view, as nice as it might seem, presents one serious drawback: for our purposes, Zariski descent is insufficient, and it has to be replaced by \'etale descent.
Now, while fpqc descent for coherent sheaves has been worked out in the framework of Berkovich $k$-analytic geometry whenever $k$ has a non-trivial valuation (cf.\ \cite{Bosch_Gortz_Coherent_modules_1998,Conrad_Descent_for_coherent_2003}), we could not locate in the literature an analogous result in the trivial valuation case.
T.\ Y.\ Yu later confirmed us the lack of such a result in the literature.
Even more, it is rather clear that the methods used in \cite{Bosch_Gortz_Coherent_modules_1998,Conrad_Descent_for_coherent_2003} do not generalize to the trivial valuation case: both need to rely heavily on the existence of formal models, and Raynaud's theory needs a non-trivial valuation.

Another difficulty arising from this approach is that in \cite{Ben-Bassat_Temkin_Tubular_2013} $X$ and $Z$ are explicitly assumed to be \emph{varieties}.
While we could not understand whether this was essential to their approach, it is an important feature of our work that this hypothesis can be dropped.
This is extremely useful for us: the proof of the iterated flag decomposition (cf.\ Corollaries \ref{cor:full_flag_decomposition} and \ref{flagbung}) heavily relies on the fact that our \'etale descent theorem (cf.\ \cref{thm:etale_descent}) and our flag decomposition result (cf.\ \cref{prop:one_step_flag_decomposition_coh}) are valid for possibly singular and non-reduced $X$.

This being said, there are many interesting situations where both our construction and Ben-Bassat and Temkin's one make sense.
It is therefore useful to compare them.
The following lemma is the key to the comparison:

\begin{lem} \label{lem:higher_codimension_induction_step}
	Let $k$ be a field equipped with the trivial valuation.
	Let $A$ be a special $k$-algebra topologically of finite presentation (cf.\ \cite[\S 1]{Berkovich_Vanishing_II_1996}) and fix $f \in A$.
	Then $(f)$ is a closed ideal in $A$ and $A / (f)$ is a special $k$-algebra.
	Define $W_A(f) \coloneqq \Spf(A)_\eta \smallsetminus \Spf(A / (f))_\eta$.
	Then there is a canonical functor
	\[ \Coh^-(A[f\inv]) \to \Coh^-(W_A(f)) \]
	which is an equivalence of stable $\infty$-categories.
\end{lem}

\begin{proof}
	Since $A$ is a special $k$-algebra, every ideal is closed and the quotients of $A$ are again special $k$-algebras by \cite[Lemma 1.1]{Berkovich_Vanishing_II_1996}.
	Now, since $A$ is topologically of finite presentation and the valuation on $k$ is trivial, we can find an ideal $I$ in $k[T_1, \ldots, T_m] \llb S_1, \ldots, S_n \rrb$ and a topological $k$-isomorphism
	\[ k[T_1, \ldots, T_m] \llb S_1, \ldots, S_n \rrb / I \simeq A . \]
	By definition, $\Spf(A)_\eta$ is the closed analytic subspace of $E(0,1)^m \times D(0,1)^n$ defined by the same ideal $I$.
	Recall that the product $E(0,1)^m \times D(0,1)^n$ has an admissible affinoid cover of the form
	\[ \{E(0,1)^m \times E(0,r)^n\}_{0 < r < 1} , \]
	where, by definition
	\[ E(0,r) \coloneqq \SpB(k\{r\inv S\}) . \]
	Since $0 < r < 1$ and the valuation on $k$ is trivial, we see that
	\[ k\{r \inv S\} = \left\{f(S) = \sum_{i = 0}^\infty a_i S^i \in k \llb S \rrb \mid \lim_{i \to +\infty} |a_i| r^i = 0\right\} = k \llb S \rrb . \]
	Pulling back this cover along the closed immersion $\Spf(A)_\eta \hookrightarrow E(0,1)^m \times D(0,1)^m$ we obtain an admissible affinoid cover of $\Spf(A)_\eta$ itself.
	Observe that
	\[ Y_r \coloneqq \Spf(A)_\eta \times_{E(0,1)^m \times D(0,1)^n} (E(0,1)^m \times E(0,r)^n) \]
	is the closed analytic subspace of $E(0,1)^m \times E(0,r)^n = \SpB(k\{T_1, \ldots, T_m, r\inv S_1, \ldots, r\inv S_n\})$ defined again by the ideal $I$.
	Under the identification of the global sections of $E(0,1)^m \times E(0,r)^n$ with $k[T_1, \ldots, T_m] \llb S_1, \ldots, S_n\rrb$, we can thus identify the global sections of $Y_r$ with the algebra $A$ itself.
	In particular, invoking Tate's acyclicity theorem, we deduce that
	\[ \Coh^-(Y_r) \simeq \Coh^-(A) . \]
	Using the descent of $\Coh^-$ for the $G$-topology,\footnote{In the case of trivial valuation, this is proved in \cite[Lemma 2.1.6]{Ben-Bassat_Temkin_Tubular_2013}.} we obtain
	\[ \Coh^-(\Spf(A)_\eta) \simeq \varprojlim_{0 < r < 1} \Coh^-(Y_r) \simeq \Coh^-(A) . \]
	We now observe that $W_A(f)$ is an admissible open of $\Spf(A)_\eta$.
	Indeed, it is defined by the condition $|f| > 0$:
	\[ W_A(f) = \{x \in \Spf(A)_\eta \mid |f(x)| > 0\} . \]
	We can therefore produce an admissible affinoid cover of $W_A(f)$ as follows.
	For every $0 < r, s < 1$, define
	\[ Y_{r,s}(f) \coloneqq \{x \in Y_r \mid |f(x)| \ge s\} . \]
	This is an affinoid domain. For $r$ fixed, it is easy to see that $\{Y_{r,s}(f)\}$ is an admissible affinoid cover of $Y_r \times_{\Spf(A)_\eta} W_A(f)$.
	Therefore, $\{Y_{r,s}(f)\}$ is an admissible affinoid cover of $W_A(f)$.
	We now remark that
	\[ Y_{r,s}(f) = \SpB((k\{T_1, \ldots, T_m, r\inv S_1, \ldots, r\inv S_n\} / I)\{s T\} / (fT - 1) ) = \SpB(A\{sT\} / (fT - 1)) . \]
	Since $s < 1$ and the valuation on $k$ is trivial, we have an isomorphism of algebras:
	\[ A\{s T\} = A[T] . \]
	It follows that there is an isomorphism of algebras
	\[ A\{sT\} / (fT - 1) \simeq A[f\inv] . \]
	Invoking again Tate's acyclicity theorem and the descent of $\Coh^-$ for the $G$-topology, we finally conclude that there is a canonical equivalence of stable $\infty$-categories
	\[ \Coh^-(W_A(f)) \simeq \Coh^-(A[f\inv]) . \]
	The proof is thus complete.
\end{proof}

\begin{cor} \label{cor:non-archimedean_help}
	Fix $U = \Spec(A_U) \in X\et$.
	Let $\Upsilon \coloneqq \Spec(\widehat{A_U}) \smallsetminus Z_U$.
	Then there is a canonical equivalence of stable $\infty$-categories
	\[ \Coh^-(W_U(Z)) \simeq \Coh^-(\Upsilon) . \]
\end{cor}

\begin{rem}
	When $Z$ is defined by a single equation in $A_U$, this is precisely the content of \cite[Lemma 4.2.1(ii)]{Ben-Bassat_Temkin_Tubular_2013}.
\end{rem}

\begin{proof}
	Let $J_U$ be the ideal defining $Z_U$ inside $U$.
	Since $A_U$ is noetherian, we can choose generators $J_U = (f_1, \ldots, f_n)$.
	We proceed by induction on the number of generators $n$.
	When $n = 1$, this is follows directly from \cref{lem:higher_codimension_induction_step} taking $A = \widehat{A_U}$ and $f = f_1$.
	
	Suppose now that the statement has already been proven for all the ideals of length $n$, and let us prove it for ideals of length $n+1$.
	Write $J = (f_1, \ldots, f_{n+1})$.
	Define
	\begin{gather*}
		Z_1 \coloneqq \Spec(A / (f_1, \ldots, f_n)), \quad Z_2 \coloneqq \Spec(A / (f_{n+1})) , \quad Z_{12} \coloneqq \Spec(A / (f_1f_{n+1}, \ldots, f_n f_{n+1})) , \\
		W_1 \coloneqq \Spf(\widehat{A_U})_\eta \smallsetminus (Z_1)_\eta, \quad W_2 \coloneqq \Spf(\widehat{A_U})_\eta \smallsetminus (Z_2)_\eta, \quad W_{12} \coloneqq \Spf(\widehat{A_U})_\eta \smallsetminus Z_{12} .
	\end{gather*}
	Observe that
	\[ W_{12} \simeq W_1 \times_{\Spf(\widehat{A_U})_\eta} W_2 , \]
	and that $W_1$ and $W_2$ form an admissible cover of $W_U(Z)$.
	In particular,
	\[ \Coh^-(W_U(Z)) \simeq \Coh^-(W_1) \times_{\Coh^-(W_{12})} \Coh^-(W_2) . \]
	On the other side, we observe that the open Zariski subsets of $\Spec(\widehat{A_U})$
	\[ V_1 \coloneqq \Spec(\widehat{A_U}) \smallsetminus Z_1, \quad V_2 \coloneqq \Spec(\widehat{A_U}) \smallsetminus Z_2 \]
	form an open Zariski cover of $\Upsilon$, whose intersection is
	\[ V_{12} = \Spec(\widehat{A_U}) \smallsetminus Z_{12} . \]
	In particular,
	\[ \Coh^-(\Upsilon) \simeq \Coh^-(V_1) \times_{\Coh^-(V_{12})} \Coh^-(V_2) . \]
	We now remark that the induction hypothesis guarantees that
	\[ \Coh^-(V_1) \simeq \Coh^-(W_1), \quad \Coh^-(V_{12}) \simeq \Coh^-(W_{12}) , \]
	and that \cref{lem:higher_codimension_induction_step} implies that
	\[ \Coh^-(V_2) \simeq \Coh^-(W_2) . \]
	The proof is thus complete.
\end{proof}

\begin{cor} \label{cor:comparison}
	Let $X$ be a $k$-variety and let $Z \hookrightarrow X$ be a closed subvariety.
	Then there is a canonical equivalence
	\[ \Coh^-(\hZ \smallsetminus Z) \simeq \Coh^-(W_X(Z)) . \]
	In particular, $\Coh^-(W_X(Z))$ is equivalent to the Verdier quotient $\Coh^-(\hZ) / \cK(X)$, where $\cK(X)$ is the category introduced in \cref{lem:kernel_sheaf}.
\end{cor}

\begin{proof}
	It is sufficient to observe that both $\Coh^-(\hZ \smallsetminus Z) = \hPhi^\circ_Z(X)$ and $\Coh^-(W_X(Z))$ satisfy Zariski descent in $X$: for the first category, this is a consequence of \cref{thm:etale_descent} and for the latter this has been proved in \cite[Lemmas 4.3.1 and 4.4.1]{Ben-Bassat_Temkin_Tubular_2013}.
	Finally, the second statement is a consequence of \cref{thm:Zariski_computational_tool}, because a variety is quasi-compact and quasi-separated.
\end{proof}

\begin{rem}
	Remark that the above corollary can be also interpreted as an example where the localization theorem holds in the non-archimedean setting.
	This is possible exactly because the situation considered here is not ``truly analytic'', in the sense that only the trivial valuation is considered.
	It is in fact well known that the statement of localization theorem is generally false in the analytic setting.
	An easy counterexample in the \canal setting is the following: let $X = \mathbf A^1_\mathbb C$ be the \canal affine line and let $Z = \{0\}$ be the origin inside $X$.
	Then $\Cohh(X) \to \Cohh(X \smallsetminus Z)$ cannot be essentially surjective (and therefore it cannot be a quotient in $\AbCat$): indeed, we can endow the subset $\{\frac{1}{n}\}_{n \ge 1, n \in \mathbb Z}$ with the structure of an analytic subset of $X \smallsetminus Z$. It is therefore defined by a coherent sheaf $\cF$, and $\cF$ cannot be extended to a coherent sheaf $\cG \in \Cohh(X)$: otherwise, $\mathrm{supp}(\cG)$ would have to contain $\mathrm{supp}(\cF)$ and therefore it would have an accumulation point, which is impossible in virtue of the analytic continuation property for holomorphic functions.
\end{rem}

\section{Verdier quotients} \label{sec:Verdier quotients}

The goal of this note is to show that cofibers exist in $\Catst$.

\begin{lem} \label{lem:compact_generators_localisation}
	Let $F \colon \cC \rightleftarrows \cD \colon G$ be an adjunction between cocomplete $\infty$-categories, where $F$ is the left adjiont.
	Suppose that:
	\begin{enumerate}
		\item $\cC$ is generated under filtered colimits by a collection of objects $\{X_\alpha\}_{\alpha \in A}$;
		\item $G$ is fully faithful.
	\end{enumerate}
	Then $\cD$ is generated under filtered colimits by the collection $\{F(X_\alpha)\}$.
\end{lem}

\begin{proof}
	Let $Y \in \cD$.
	Choose a filtered category $\cJ$ and a functor $H \colon \cJ \to \cC$ such that:
	\begin{enumerate}
		\item $H$ factors through the full subcategory spanned by $\{X_\alpha\}_{\alpha \in A}$.
		\item There is an equivalence
		\[ \colim_\cJ H \simeq G(Y) . \]
	\end{enumerate}
	Since $F$ commutes with arbitrary colimits, we see that
	\[ \colim_\cJ F \circ H \simeq F(G(Y)) . \]
	Since $G$ is fully faithful, the counit transformation $\epsilon_Y \colon F(G(Y)) \to Y$ is an equivalence.
	This completes the proof of the lemma.
\end{proof}

\begin{thm} \label{thm:general_Verdier_quotients}
	Let $\Catst$ be the $\infty$-category of stable $\infty$-categories with exact functors between them.
	Let $\cC$ be a stable $\infty$-category and let $j \colon \cA \hookrightarrow \cC$ be a stable full subcategory of $\cC$.
	Then in $\Catst$ there exists a pushout diagram:
	\[ \begin{tikzcd}
	\cA \arrow[hook]{r}{j} \arrow{d} & \cC \arrow{d} \\
	0 \arrow{r} & \cD .
	\end{tikzcd} \]
\end{thm}

\begin{proof}
	Applying $\Ind$ to $j \colon \cA \hookrightarrow \cC$, we obtain a functor $\Ind(j) \colon \Ind(\cA) \to \Ind(\cC)$ which is fully faithful and commutes with filtered colimits and with finite limits.
	Since $\Ind(\cA)$ and $\Ind(\cC)$ are stable by \cite[1.1.3.6]{Lurie_Higher_algebra}, it follows that $\Ind(j)$ commutes with arbitrary colimits.
	Furthermore, since compact objects in $\Ind(\cA)$ (resp.\ in $\Ind(\cC)$) are precisely the retracts of elements in $\cA$ (resp.\ in $\cC$), we see that this functor preserves compact objects.
	In other words, it is a morphism in $\LPromega_{\mathrm{st}}$.
	It follows (for example from \cite{Blumberg_Gepner_Universal_2013}) that the cofiber
	\[ \begin{tikzcd}
	\Ind(\cA) \arrow{r}{\Ind(j)} \arrow{d} & \Ind(\cC) \arrow{d}{L} \\
	0 \arrow{r} & \cD'
	\end{tikzcd} \]
	exists in $\LPromega_{\mathrm{st}}$.
	Let $\cD$ be the smallest full stable subcategory of $\cD$ containing the essential image of the composition
	\[ \begin{tikzcd}
	\cC \arrow{r}{y_\cC} & \Ind(\cC) \arrow{r}{L} & \cD' ,
	\end{tikzcd} \]
	where $y_\cC$ is the Yoneda embedding.
	We claim that the canonical map $\overline{L} \colon \cC \to \cD$ exhibits $\cD$ as cofiber of $\cA \to \cC$.
	Observe that there is a clear nullhomotopy $\overline{L} \circ j \simeq 0$.
	
	Before starting the proof, we claim that $\Ind(\cD) \simeq \cD'$.
	Let $i \colon \cD' \to \Ind(\cC)$ be a right adjoint for $L$. Then $i$ is fully faithful and it commutes with filtered colimits.
	Since $\Ind(\cC)$ is generated under filtered colimits by $\cC$, it follows from \cref{lem:compact_generators_localisation} that $\cD'$ is generated under filtered colimits by $L(\cC)$, hence by $\cD$.
	A standard argument shows that every object in $(\cD')^\omega$ is a retract of one in $\cD$.
	In other words, the natural inclusion $\cD \subseteq (\cD')^\omega$ exhibits $(\cD')^\omega$ as the idempotent completion of $\cD$ (cf.\ \cite[5.1.4.1]{HTT}).
	Therefore
	\[ \Ind(\cD) \simeq \Ind((\cD')^\omega) \simeq \cD' . \]
	
	Now, let $\cE$ be any stable $\infty$-category.
	Let
	\[ \Map_{\Catst}^0(\cC, \cE) \]
	be the full subsimplicial set of $\Map_{\Catst}(\cC, \cE)$ spanned by those $f \colon \cC \to \cE$ such that $f(c)$ is a zero object for $\cE$ for every $c \in \cE$.
	Clearly, we have
	\[ \Map_{\Catst}^0(\cC, \cE) \simeq \Map_{\Catst}(\cC, \cE) \times_{\Map_{\Catst}(\cA, \cE)} \Map_{\Catst}(0, \cE) , \]
	the fiber product being computed in $\cS$ (i.e.\ it is a homotopy fiber product).
	Now, let $g \colon \cD \to \cE$ be any exact $\infty$-functor.
	Then the nullhomotopy $\overline{L} \circ j \simeq 0$ induces a nullhomotopy $g \circ \overline{L} \circ j \simeq 0$, and therefore $g \circ \overline{L}$ belongs to $\Map_{\Catst}^0(\cC, \cE)$.
	In other words, precomposition with $\overline{L}$ provides us with a morphism of simplicial sets
	\[ \eta \colon \Map_{\Catst}(\cD, \cE) \to \Map_{\Catst}^0(\cC, \cE) . \] 
	To complete the proof, it is enough to prove that this is an equivalence in $\cS$.
	
	Pick any $f \in \Map_{\Catst}^0(\cC, \cE)$.
	It will be enough to prove that the fiber of $\eta$ at $f$ is weakly contractible.
	Since
	\[ \Map_{\Catst}^0(\cC, \cE) \to \Map_{\Catst}(\cC, \cE) \]
	is $(-1)$-truncated by definition, we see that it is enough to prove that the fiber product
	\begin{equation} \label{eq:fiber_space}
	\Map_{\Catst}^{\cC /}(\cD, \cE) \simeq \{f\} \times_{\Map_{\Catst}(\cC, \cE)} \Map_{\Catst}(\cD, \cE)
	\end{equation}
	is contractible.
	Observe that, since $\cE \to \Ind(\cE)^\omega$ is fully faithful, the map
	\[ \Map_{\Catst}^{\cC /}(\cD, \cE) \to \Map_{\Catst}^{\cC /}(\cD, \Ind(\cE)^\omega) \]
	is $(-1)$-truncated.
	Moreover, using the universal property of the idempotent completion, we have:
	\[ \Map_{\Catst}^{\cC /}(\cD, \Ind(\cE)^\omega) \simeq \Map_{\Catstidem}^{\Ind(\cC)^\omega /}((\cD')^\omega, \Ind(\cE)^\omega) \simeq \Map_{\LPromega_{\mathrm{st}}}^{\Ind(\cC)/}(\cD', \Ind(\cE)) . \]
	The universal property of the cofiber shows now that this last space is contractible.
	
	In conclusion, it will be enough to prove that the space \eqref{eq:fiber_space} is non empty.
	In order to see this, let us observe that the solid diagram
	\[ \begin{tikzcd}
	\cC \arrow{r}{y_\cC} \arrow{d}{\overline{L}} \arrow[bend right = 50pt]{dd}[swap]{f} & \Ind(\cC)^\omega \arrow{d}{L^\omega} \arrow{r} & \Ind(\cC) \arrow{d}{L} \arrow[bend left = 50pt]{dd}{\Ind(f)} \\
	\cD \arrow{r}{y_\cD} \arrow[dashed]{d}{\overline{g}} & (\cD')^\omega \arrow{r} \arrow{d}{g^\omega} & \Ind(\cD) \arrow{d}{g} \\
	\cE \arrow{r}{y_\cE} & \Ind(\cE)^\omega \arrow{r} & \Ind(\cE)
	\end{tikzcd} \]
	Since $y_\cE$ is fully faithful, in order to show that the composition $g^{\omega} \circ y_\cD$ factors through $y_{\cE}$ via the dotted arrow $\overline{g}$, it is enough to prove that for every $X \in \cD$, $g^{\omega}(y_\cD(X)) \in \cE$.
	Since all the functors are exact and since $\cD$ is generated under fiber sequences by the essential image of $\overline{L}$, it follows that it is enough to prove this claim when $X$ is of the form $\overline{L}(Y)$.
	But then
	\[ g^\omega(y_\cD(\overline{L}((Y)))) = g^\omega( L^\omega( y_\cC(Y))) = y_{\cE}(f(Y)) . \]
	This completes the proof.
\end{proof}

\bibliographystyle{plain}
\bibliography{dahema}

\end{document}